\documentclass[a4paper]{article}
\usepackage[top=1in, bottom=1in, left=1.1in, right=1.8in]{geometry}
\usepackage{amsfonts}
\usepackage{amsmath}
\usepackage{amsthm,amssymb}
\usepackage{CJK,graphicx}
\usepackage{amscd}
\usepackage{amssymb}
\usepackage{mathrsfs}
\usepackage[all,cmtip]{xy}
\usepackage{lmodern}
\usepackage[Symbol]{upgreek}
\usepackage{bm}
\usepackage{eucal}
\usepackage[nopar]{lipsum}
\newtheorem{theorem}{Theorem}[section]
\newtheorem{corollary}{Corollary}[section]
\newtheorem{definition}{Definition}[section]
\newtheorem{lemma}{Lemma}[section]
\newtheorem{proposition}{Proposition}[section]

\newtheorem{conjecture}{Conjecture}[section]
\AtEndDocument{\bigskip{\footnotesize
  \textsc{University College London} \par
  \textit{E-mail address}: \texttt{yin.li.16@ucl.ac.uk} \par
}}
\begin{document}
\title{\textbf{Disjoinable Lagrangian tori and semisimple symplectic cohomology}}\author{Yin Li}\date{}\maketitle
\begin{abstract}
We derive constraints on Lagrangian embeddings in completions of certain stable symplectic fillings whose symplectic cohomologies are semisimple. Manifolds with these properties can be constructed by generalizing the boundary connected sum operation to our setting, and are related to birational surgeries like blow-downs and flips. As a consequence, there are many non-toric (non-compact) monotone symplectic manifolds whose wrapped Fukaya categories are proper.
\end{abstract}

\section{Introduction}

\subsection{Motivations and related works}
The number of disjoinable (non-displaceable) Lagrangian submanifolds with certain topology in a symplectic manifold $M$ is an invariant which measures the ``size" or ``complexity" of $M$.\bigskip

Among considerations along these lines, three cases are of particular interest, namely when $L\subset M$ is diffeomorphic to $S^n$, $\mathbb{CP}^{n/2}$ and $T^n$, where $n=1/2\dim_\mathbb{R}(M)$, as they correspond to interesting surgeries in the symplectic or algebraic category.

In the case of Lagrangian spheres, this viewpoint is addressed in the construction of $\cite{sty}$, where the surgery replaces conifold singularities with Lagrangian $S^3$'s. On the other hand, when $M$ is a Liouville manifold which carries a dilation $b\in\mathit{SH}^1(M)$ (in the sense of Seidel-Solomon $\cite{ss}$) located in the degree 1 symplectic cohomology, Seidel proved in $\cite{ps}$ that there is an integer $N$ depending on the twisted Floer cohomology $\widetilde{H}^\ast$ such that if $(L_1,\cdot\cdot\cdot,L_r)$ is a collection of disjoinable Lagrangian spheres, then $r\leq N$.

It's an easy observation that with slight modifications, Seidel's arguments in $\cite{ps}$ can be adapted to the case when $(L_1,\cdot\cdot\cdot,L_r)$ is a collection of Lagrangian submanifolds diffeomorphic to $\mathbb{CP}^{n/2}$ in Liouville manifolds with dilations (when $n$ is a multiple of 4, one works over a field $\mathbb{K}$ with $\mathrm{char}(\mathbb{K})=2$). This is the case, for example, when $M=T^\ast\mathbb{CP}^{n/2}$. From the Weinstein neighborhood theorem, it is easy to see that a collection of Lagrangian $\mathbb{CP}^{n/2}$'s provides the starting point of the Mukai flop $\cite{yk}$.\bigskip

For the above two cases, the method of $\cite{ps}$ applies mainly because Lagrangian submanifolds diffeomorphic to $S^n$ or $\mathbb{CP}^{n/2}$ can be equipped with a $b$-equivariant structure when $M$ admits a dilation $b$, see Section 4 of $\cite{ss}$. From a more algebraic point of view established in $\cite{ps1}$, spherical and projective objects in the Fukaya category $\mathcal{F}(M)$ are $\mathbb{C}^\ast$-equivariant as $A_\infty$-modules over certain endomorphism algebra of objects in $\mathcal{F}(M)$ since they are rigid and simple, or more concretely
\begin{equation}\mathit{HF}^1(L_i,L_i)=0,\mathit{HF}^0(L_i,L_i)\cong\mathbb{C}\end{equation}
for $1\leq i\leq r$. However, taking a mirror symmetric point of view, Lagrangian tori are less likely to be equivariant because they correspond to skyscraper sheaves of points in $D^b\mathit{Coh}(M^\vee)$, the derived category of coherent sheaves of the mirror.

When $M$ is closed and monotone, its mirror is expected to be a Landau-Ginzburg model $(M^\vee,W)$ $\cite{ako}$, where $W:M^\vee\rightarrow\mathbb{K}$ is the superpotential taking values in an algebraically closed field $\mathbb{K}$. Under the additional assumption that $W$ is Morse, the triangulated category of matrix factorizations $H^0\left(\mathit{MF}(M^\vee,W)\right)$ is then split-generated by skyscraper sheaves supported at the critical points of $W$; or equivalently, a set of idempotents after passing to its split-closure. From this point of view, the semisimplicity assumption imposed on the \textit{small} quantum cohomology $\mathit{QH}^\ast(M)$ in $\cite{ep}$, which is expected to be mirror to the semisimplicity of $H^0\left(\mathit{MF}(M^\vee,W)\right)$, seems to be natural. More precisely, let $(L_1,\cdot\cdot\cdot,L_r)$ be a collection of monotone Lagrangian tori in $M$ such that they are disjoinable by Hamiltonian isotopies and $\mathit{HF}^\ast(L_i,L_i)\neq0$ for all $i$, then it is essentially a consequence of Theorem 1.25 of $\cite{ep}$ that
\begin{equation}r\leq\dim_\mathbb{K}\mathit{QH}^{\mathrm{even}}(M).\end{equation}
This is a beautiful example where the closed string invariant $\mathit{QH}^\ast(M)$ is used to give global constraints on Lagrangian embeddings, which belongs to the open string sector. See also the work of Biran-Cornea $\cite{bc}$ for studies of the same flavor.\bigskip

The assumption that $\mathit{QH}^\ast(M)$ is semisimple imposes strong restrictions on $M$, and many known examples of symplectic manifolds $M$ with semisimple $\mathit{QH}^\ast(M)$ can be obtained by starting from known ones and applying birational surgeries, such as blow-ups and reverse flips to $M$. The open string counterpart of this observation has been investigated recently by Charest-Woodward:

\begin{theorem}[Charest-Woodward $\cite{cw}$]\label{theorem:wood}
Let $(M_+,\omega_{M_+})$ be a compact symplectic manifold with $[\omega_{M_+}]\in H^2(M_+;\mathbb{Q})$. Suppose $M_+$ is obtained from a compact symplectic manifold $M_-$ with $[\omega_{M_-}]\in H^2(M_-;\mathbb{Q})$ by a reverse simple flip or blow-up with trivial center with multiplicity $r$, then in a contractible neighborhood of the exceptional locus there exist a Lagrangian torus $L\subset M_+$, $r$ distinct local systems $\xi_L^1,\cdots,\xi_L^r$ and weak bounding cochains $b_L^1,\cdots,b_L^r$, such that $\mathit{HF}^\ast((L,\xi_L^i,b_L^i),(L,\xi_L^i,b_L^i))\neq0$ for $1\leq i\leq r$.
\end{theorem}

This paper continues the exploration of the above picture in the case when $M$ is a connected monotone symplectic manifold obtained by completing a \textit{stable symplectic filling} of a $(2n-1)$-dimensional contact manifold $(V,\xi)$, where $\xi$ is a cooriented contact structure. Briefly, this means that $V$ is the boundary of a codimension 0 submanifold $M^\mathrm{in}\subset M$, and there exists a contact form $\theta_V$ which, together with the restriction of $\omega_M$, form a stable Hamiltonian structure on $V$, see Section \ref{section:ham} for details. This assumption is made here so that the symplectic cohomology $\mathit{SH}^\ast(M)$ can be shown to be well-defined as a ring by standard arguments based on maximum principles, see Section \ref{section:ham}. For a discussion of Hamiltonian Floer theory on more general open manifolds, we refer to $\cite{yg}$.\bigskip

In the special case when $M$ is the total space of a negative line bundle $\mathcal{L}\rightarrow B$ over a closed monotone symplectic manifold $(B,\omega_B)$, its Floer theory has been studied extensively in $\cite{ao,ar,ar1,rs}$. These are examples of \textit{convex} symplectic manifolds, namely $d\theta_V=\omega_V$, so the symplectic filling $(M^\mathrm{in},\omega_M)$ of $(V,\xi)$ is \textit{strong} in the sense of $\cite{mnw}$. Roughly speaking, the upshot is that the symplectic topology of $M$ has features similar to closed monotone symplectic manifolds, with quantum cohomology replaced by symplectic cohomology.

More precisely, the split-generation of certain summands of the monotone Fukaya category $\mathcal{F}_\lambda(M)$ by Lagrangian tori (equipped with local systems) is similar to the closed case, from which one can deduce the following closed string counterpart of the homological mirror symmetry conjecture:
\begin{equation}\mathit{SH}^\ast(M)\cong\mathit{Jac}(W),\end{equation}
see $\cite{ar1}$ and $\cite{rs}$ for a proof in the case of negative line bundles, and refer to Section \ref{section:example} of this paper for a more general statement. This suggests that $\mathit{SH}^\ast(M)$ should be the appropriate replacement of $\mathit{QH}^\ast(M)$ for the purpose of studying Lagrangian embeddings in the current setting. Here $\mathit{Jac}(W)$ is the Jacobi ring of the superpotential $W:(\mathbb{K}^\ast)^n\rightarrow\mathbb{K}$ on the mirror Landau-Ginzburg model.\bigskip

Motivated by this we will prove an analogue of Entov-Polterovich's theorem in the non-compact setting, which gives an upper bound for the number of certain disjoinable non-displaceable Lagrangian tori in terms of $\mathit{SH}^\ast(M)$. See Theorem \ref{theorem:main}.\bigskip

On the other hand, it is a viewpoint established by Seidel and Smith in $\cite{ps2}$ that the existence of non-displaceable Lagrangian tori should result in the non-vanishing of the symplectic cohomology. Combining with Theorem \ref{theorem:wood}, this suggests that for stable symplectic fillings $M$, reverse $\mathit{MMP}$ transitions (there is no minimal model program for non-compact varieties, but it still makes sense to talk about $\mathit{MMP}$ transitions) will contribute non-trivially to $\mathit{SH}^\ast(M)$. In particular, an analogue of Conjecture 1.1 in $\cite{cw}$ adapted to the non-compact case should imply the following:
\begin{conjecture}\label{conjecture:mmp}
Let
\begin{equation}M_+\dashrightarrow\cdot\cdot\cdot\dashrightarrow M_i\dashrightarrow\cdot\cdot\cdot\dashrightarrow M_-\end{equation}
be a sequence of symplectic $\mathit{MMP}$ transitions among completions $M_i$ of stable fillings such that the quantum cohomology $\mathit{QH}^\ast(Z_i)$ of the center $Z_i\subset M_i$ (which is a compact symplectic submanifold) in each step has semisimple non-zero eigensummands. Suppose $\mathit{SH}^\ast(M_-)=0$ or is semisimple, then $\mathit{SH}^\ast(M_+)$ is semisimple.
\end{conjecture}
A piece of this conjecture will be proved in Corollaries \ref{corollary:blow-up} and \ref{corollary:flip}.

\subsection{New results}
We now turn to the main contents of this paper. From now on, the coefficient field $\mathbb{K}$ will mean the Novikov field
\begin{equation}\mathbb{K}:=\left\{\sum_{i=0}^\infty a_iq^{n_i}\left| a_i\in\mathbb{F},n_i\in\mathbb{R},\lim_{i\rightarrow\infty}n_i=\infty\right.\right\},\end{equation}
where $\mathbb{F}$ is an algebraically closed field with $\mathrm{char}(\mathbb{F})=0$. This particular choice is picked here simply for convenience, but the use of a Novikov field is necessary due to the non-existence of an a priori energy estimate for Floer solutions.

We follow the convention of Section 11.1 of $\cite{ms}$ for the discussion of quantum cohomology $\mathit{QH}^\ast(M)$. For any $a\in\mathit{QH}^\ast(M)$, denote by $a^{\star j}$ its $j$-th power under quantum cup product. Let $N_M$ be the minimal Chern number of $M$, recall that $\mathit{QH}^\ast(M)$ can be equipped with a $\mathbb{Z}/2N_M$-grading and as $\mathbb{K}$-vector spaces we have the decomposition
\begin{equation}\mathit{QH}^{2i}(M)=\bigoplus_{2j\equiv 2i\textrm{ mod }2N_M}H^{2j}(M;\mathbb{K}).\end{equation}
To ensure that the symplectic cohomology $\mathit{SH}^\ast(M)$ is well-defined, it suffices to impose the assumption that $M$ is \textit{semi-positive}, i.e. for any class $[u]\in\pi_2(M)$ we have
\begin{equation}\label{eq:semi-positive}3-n\leq c_1(M)\left([u]\right)<0\Rightarrow \omega_{M}\left([u]\right)\leq0.\end{equation}

The following definition is motivated by the work of Ritter $\cite{ar,ar1}$, where the symplectic cohomologies of negative line bundles $\mathcal{O}(-m)\rightarrow\mathbb{CP}^{n-1}$, with $1\leq m\leq n-1$ are studied. His main result, which says that $\mathit{SH}^\ast(\mathcal{O}_{\mathbb{P}^{n-1}}(-m))$ is the quotient of $\mathit{QH}^\ast(\mathcal{O}_{\mathbb{P}^{n-1}}(-m))$ by the zero eigensummand of the quantum multiplication by the first Chern class, can be regarded as a symplectic analogue of the quantum Lefschetz hyperplane theorem $\cite{ypl}$. 

\begin{definition}\label{definition:Lefschetz}
Let $M$ be the completion of a stable filling $M^\mathrm{in}$ of the contact boundary $(V,\xi)$. Suppose $M$ is semi-positive. We say that $M^\mathrm{in}$ is a Lefschetz domain at level $j$ if
\begin{itemize}
\item[(i)] There exists an integer $j$ with $1\leq j\leq n-1$ such that $H^{2j}(V;\mathbb{Q})=0$ and $c_1(M)^{\star j}\in H^{2j}(M;\mathbb{K})\subset\mathit{QH}^\mathrm{even}(M)$.
\item[(ii)] As a ring, $\mathit{SH}^\ast(M)$ is semisimple and is the localization of $\mathit{QH}^\ast(M)$ at $c_1(M)$.
\end{itemize}
\end{definition}

The topological restriction on $V$ imposed by condition (i) above is satisfied for a large class of good contact toric manifolds by the computations in $\cite{sl}$. This will be clarified in the proof of Proposition \ref{proposition:OC} and in Section \ref{section:con-tor}. By the aforementioned works of Ritter, (ii) is satisfied for the negative line bundles $\mathcal{O}(-m)\rightarrow\mathbb{CP}^{n-1}$ with $1\leq m\leq n-1$. (i) is not needed for the the proof of Theorem \ref{theorem:main} if $M^\mathrm{in}$ is a strong symplectic filling, although it actually holds for $\mathcal{O}(-m)\rightarrow\mathbb{CP}^{n-1}$, and the unit disc bundle $\mathcal{O}(-m)_{\leq1}$ is a Lefschetz domain at level $j$ for any $1\leq j\leq n-1$. In the simplest case when $m=1$, $c_1(M)^{\star j}=c_1(M)^j\in H^{2j}(M;\mathbb{K})$.

For our first theorem, only the $\mathbb{Z}/2$-grading on $\mathit{SH}^\ast(M)$ will be relevant.
\begin{theorem}\label{theorem:main}
Assume $M^\mathrm{in}$ is a monotone Lefschetz domain (at any level $j$) in the sense of Definition \ref{definition:Lefschetz}. Let $(L_1,\cdot\cdot\cdot,L_r)$ be a collection of closed monotone Lagrangian submanifolds in $M$ which are oriented and $\mathit{Spin}$. Assume $(L_1,\cdot\cdot\cdot,L_r)$ are pairwise disjoinable by Hamiltonian isotopies and $m_0(L_i)\neq0$ for all $i$. Suppose that $\mathit{HF}^\ast(L_i,L_i)\neq0$ and $pt\in C_\ast(L_i)$ defines a cocycle in each $\mathit{HF}^\ast(L_i,L_i)$, then
\begin{equation}\label{eq:bound}r\leq\dim_\mathbb{K}\mathit{SH}^0(M).\end{equation}
\end{theorem}
In the above, the value $m_0(L_i)\in\mathbb{K}$ is determined by the Maslov index 2 holomorphic discs bounded by $L_i$, see Section \ref{section:fuk}.\bigskip

Following $\cite{bc}$, we say that an unobstructed Lagrangian submanifold $L\subset M$ is \textit{wide} if $\mathit{HF}^\ast(L,L)\cong H^\ast(L;\mathbb{K})$, and is \textit{narrow} if $\mathit{HF}^\ast(L,L)=0$. Every known monotone Lagrangian submanifold is either wide or narrow. It is proved in $\cite{bc}$ that any two closed non-narrow monotone Lagrangian submanifolds in $\mathbb{CP}^n$ intersect. We have the following partial analogue in the non-compact case.
\begin{corollary}\label{corollary:wide}
Let $M$ be the total space of $\mathcal{O}(-m)^{\oplus n/(m+1)}\rightarrow\mathbb{CP}^{mn/(m+1)}$ with $1\leq m\leq n-1$, then any two closed wide monotone Lagrangian submanifolds $L_1,L_2\subset M$ with $m_0(L_i)\neq0$ have non-trivial intersections.
\end{corollary}
\begin{proof}
When $m=1$, it can be checked that $M^\mathrm{in}$ is a Lefschetz domain, see Lemma \ref{lemma:boundary}, Proposition \ref{proposition:computation} and Corollary \ref{corollary:Lefschetz}. By Theorem \ref{theorem:main}, it suffices to show that $\dim_\mathbb{K}\mathit{SH}^0(M)=1$, but this follows from the computations in Section \ref{section:computation}. The remaining cases when $2\leq m\leq n-1$ are completely parallel to the $m=1$ case, see Section \ref{section:computation} for details.
\end{proof}

Let us move onto the construction of Lefschetz domains. Our method here is based on a generalization of the boundary connected sum operation. Such an operation is originally introduced by Weinstein $\cite{aw}$ for convex symplectic manifolds and has been generalized to weak symplectic fillings in $\cite{gz}$, see also $\cite{gnw}$ for a further extension, which allows one to add a subcritical handle to any weak symplectic filling. We actually need a slight modification of their construction so that it can be adapted to the category of stable fillings, see Section \ref{section:handle} for details. It then follows that if both of $M^\mathrm{in}$ and $(M')^\mathrm{in}$ are stable fillings, our construction will yield another stable filling $M^\mathrm{in}\#_\partial (M')^\mathrm{in}$. Meanwhile, it follows from our construction that if $M$ and $M'$ are semi-positive, then so is $M\#_\partial M'$, the completion of $M^\mathrm{in}\#_\partial (M')^\mathrm{in}$. In particular, $\mathit{SH}^\ast(M\#_\partial M')$ is therefore well-defined.
\begin{theorem}\label{theorem:flip}
Let $M^\mathrm{in}$ and $(M')^\mathrm{in}$ be stable symplectic fillings of the contact manifolds $(V,\xi)$ and $(V',\xi')$, whose completions $M$ and $M'$ are semi-positive. Then
\begin{equation}\mathit{SH}^\ast(M\#_\partial M')\cong\mathit{SH}^\ast(M)\oplus\mathit{SH}^\ast(M')\end{equation}
as rings.
\end{theorem}
Theorem \ref{theorem:flip} can be adapted to verify some special cases of Conjecture \ref{conjecture:mmp}. The following are two examples, their proofs will appear in Section \ref{section:symp-birational}.
\begin{corollary}\label{corollary:blow-up}
Let $M_-$ be a semi-positive symplectic manifold which is the completion of a stable filling, such that $\mathit{SH}^\ast(M_-)=0$ or $\mathit{SH}^\ast(M_-)$ is semisimple. Suppose $M_+=\mathit{Bl}_x(M_-)$ is the blow-up of $M_-$ at a point $x$ on the cylindrical end $M\setminus M^\mathrm{in}$, then $\mathit{SH}^\ast(M_+)$ is semisimple.
\end{corollary}
Denote by $E^\mathrm{in}$ the stable filling of the total space of the unit sphere bundle $\partial E^\mathrm{in}$ of $\mathcal{O}(-1)^{\oplus n_1}\rightarrow\mathbb{CP}^{n_2}$. Since $\partial E^\mathrm{in}$ can be realized as a circle bundle over $\mathbb{CP}^{n_1-1}\times\mathbb{CP}^{n_2}$, there is a standard contact structure $\xi_{BW}$ constructed in $\cite{bw}$. When $1\leq n_1\leq n_2$, $c_1(E)>0$, to indicate this we shall denote it by $E_+$. Under the local model of the standard flip
\begin{equation}\mathcal{O}(-1)^{\oplus n_1}\rightarrow\mathbb{CP}^{n_2}\dashrightarrow\mathcal{O}(-1)^{\oplus n_2+1}\rightarrow\mathbb{CP}^{n_1-1},\end{equation}
$E_+$ is mapped to the negative vector bundle $E_-$, which is not semi-positive. Suppose that there is some regularization scheme under which $\mathit{SH}^\ast(E_-)$ is well-defined, then one can argue similarly as in Section \ref{section:computation} to deduce that $\mathit{SH}^\ast(E_-)=0$.

On the other hand, it follows from the computations in Section \ref{section:computation} that $\mathit{SH}^\ast(E_+)$ is semisimple. As a consequence we have the following:
\begin{corollary}\label{corollary:flip}
Let $U^\mathrm{in}$ be a stable filling and $M_-$ be the symplectic manifold obtained by attaching a cylindrical end to the boundary connected sum $U^\mathrm{in}\#_\partial E_-^\mathrm{in}$, with $\mathit{SH}^\ast(U)=0$ or being semisimple. Under the reverse simple flip $M_-\dashrightarrow M_+$ which replaces $\mathbb{CP}^{n_1-1}\subset E_-^\mathrm{in}$ by $\mathbb{CP}^{n_2}\subset E_+^\mathrm{in}$, $\mathit{SH}^\ast(M_+)$ is semisimple.
\end{corollary}
Another implication of Theorem \ref{theorem:flip} is that it can be used to produce more examples of Lefschetz domains which satisfy the constraints on Lagrangian embeddings established in Theorem \ref{theorem:main}. See Section \ref{section:Lefschetz}. This enables us to apply Theorem \ref{theorem:main} to some non-toric monotone symplectic manifolds, see Section \ref{section:example}. In particular, for the blow-ups of points on $\mathbb{C}^2$ studied in $\cite{is}$ and $\cite{ll}$, the bound obtained in Theorem \ref{theorem:main} is sharp.\bigskip

If $M^\mathrm{in}$ is a monotone Lefschetz domain, and the equality in the bound (\ref{eq:bound}) can be achieved by a certain collection of monotone Lagrangian subamnifolds $(L_1,\cdot\cdot\cdot,L_r)$, or more generally, a collection of monotone Lagrangian branes $\big((L_1,\xi_{L_1}),\cdot\cdot\cdot,(L_r,\xi_{L_r})\big)$, where
\begin{equation}\xi_{L_i}:\pi_1(L_i)\rightarrow U_\mathbb{K}\end{equation}
is a group homomorphism, with $U_\mathbb{K}\subset\mathbb{K}$ being the group of units, then a byproduct of Theorem \ref{theorem:main} is the following:
\begin{theorem}\label{theorem:generation}
$M^\mathrm{in}$ is a monotone Lefschetz domain. Let $\big((L_1,\xi_{L_1}),\cdot\cdot\cdot,(L_r,\xi_{L_r})\big)$ be a collection of monotone Lagrangian branes with $m_0(L_i,\xi_{L_i})\neq0$ and $r=\dim_\mathbb{K}\mathit{SH}^0(M)$, such that $m_0(L_i,\xi_{L_i})\neq m_0(L_j,\xi_{L_j})$ whenever $L_i\cap L_j\neq\emptyset$. Suppose
\begin{equation}\mathit{HF}^\ast\big((L_i,\xi_{L_i}),(L_i,\xi_{L_i})\big)\neq0, 1\leq i\leq r,\end{equation}
and every $pt\in C_\ast(L_i)$ defines a Floer cocycle. Then the non-zero eigensummands of the derived Fukaya category $\bigsqcup_{\lambda\neq0}D^\pi\mathcal{F}_\lambda(M)$ and the derived wrapped Fukaya category $\bigsqcup_{\lambda\neq0}D^\pi\mathcal{W}_\lambda(M)$ are split-generated by $(L_1,\xi_{L_1}),\cdot\cdot\cdot,(L_r,\xi_{L_r})$.
\end{theorem}
We remark that $L_i=L_j$ for $i\neq j$ is allowed in the above. Recall that the monotone wrapped Fukaya category $\mathcal{W}(M)$ for $M^\mathrm{in}$ a strong filling is defined in $\cite{as}$ and $\cite{rs}$. The definition can be easily extended to the case of stable fillings by allowing non-compact monotone Lagrangian submanifolds $L\subset M$ which are \textit{stable Lagrangian fillings} of their Legendrian boundaries, i.e. those of the form $\partial L^\mathrm{in}\times [1,\infty)$ on the cylindrical end, where $\partial L^\mathrm{in}\subset V$ is a Legendrian submanifold and $\theta_V|L^\mathrm{in}$ vanishes near $\partial L^\mathrm{in}$. See Section \ref{section:fuk}.\bigskip

In the special case when $M$ is the total space of $\mathcal{O}(-m)\rightarrow\mathbb{CP}^{n-1}$, where $1\leq m\leq n-1$, the above theorem has been proved by Ritter $\cite{ar1}$, since in this case the eigensummands $\mathit{QH}^\ast(M)_\lambda$ and $\mathit{SH}^\ast(M)_\lambda$ for $\lambda\neq0$ are 1-dimensional. However, the symplectic manifold obtained by blowing up $\mathbb{C}^2$ at more than one points (with equal amounts) mentioned above already gives an example where $\mathit{QH}^\ast(M)_{\lambda\neq0}$ and $\mathit{SH}^\ast(M)_{\lambda\neq0}$ are semisimple but not 1-dimensional.

To apply Theorem \ref{theorem:generation} to concrete examples, it remains to find a collection of Lagrangian branes satisfying the conditions above. For toric negative line bundles, this can be done using standard toric techniques, see $\cite{rs,ar1}$. With our tools, it is easy to generalize their results to toric negative vector bundles which split as a direct sum of line bundles, see Section \ref{section:example}. We also find the generators for the Fukaya categories of $\mathit{Bl}_S(\mathbb{C}^n)$, where $S$ is a finite set of distinct points, using essentially elementary methods. This gives a non-toric example for which the non-zero eigensummands of the wrapped Fukaya category $\bigsqcup_{\lambda\neq0}\mathcal{W}_\lambda(M)$ are cohomologically finite.

\subsection{Contents}
The structure of this paper is as follows. In Section \ref{section:structure} we collect some basic algebraic preliminaries which are needed in the proof of Theorem \ref{theorem:main}. Both of the closed and the open string invariants can be generalized to the setting of stable symplectic fillings, at least when their completions are monotone, so do the open-closed string maps relating these two flavors of Floer theory. These generalizations are mainly based on the relevant maximum principles, which can be extended to the current set up with the help of a convenient class of tame almost complex structures. Theorem \ref{theorem:main} will be proved in Section \ref{section:pf-main}. The proof is a combination of an argument outlined in $\cite{ps}$ and a modification of the proof of the non-vanishing of $\mathit{OC}^0$ appeared in $\cite{rs}$.

Section \ref{section:semisimplicity} is devoted to another important issue of this paper, namely the semisimplcity of symplectic cohomologies. We translate the insight provided by Woodard's Theorem \ref{theorem:wood} to construct manifolds with semisimple symplectic cohomologies using reverse $\mathit{MMP}$ transitions. This step is far from complete and has only been carried out for very restrictive cases, e.g. Corollaries \ref{corollary:blow-up} and \ref{corollary:flip}. Our approach here is to replace certain birational surgeries with a surgery which connects the original manifold to the exceptional pieces (created by blow-ups or reverse flips) by symplectic 1-handles. Such a symplectic handle attachment is by no means new, and has been studied in great detail by Cieliebak $\cite{kc}$ and McLean $\cite{mm}$. Our contribution here is solely to observe that based on the works of Cieliebak-Volkov $\cite{cv}$ and Massot-Niederkr\"{u}ger-Wendl $\cite{mnw}$, such a surgery can be carried out within the category of stable fillings (Section \ref{section:handle}). Theorem \ref{theorem:flip} can then be proved by mimicking the arguments presented in $\cite{cv}$ and $\cite{mm}$, see Section \ref{section:attachment}. To get back to the original surgeries via birational maps, one applies a deformation argument which depends on the invariance of $\mathit{SH}^\ast(M)$ under certain symplectomorhisms (Section \ref{section:symp-birational}). The computation of the symplectic cohomologies of the exceptional pieces is possible by generalizing Ritter's work on Seidel representations $\cite{ar,ar1}$, see Sections \ref{section:seidel} and \ref{section:computation}.

The last section contains some interesting examples and implications of the main results. In particular, there are monotone symplectic manifolds with proper wrapped Fukaya categories which are neither convex nor toric.

\section*{Acknowledgement}
I would like to thank Xiaowen Hu, Qingyuan Jiang and Chris Wendl for useful discussions during the preparation of this paper, Dmitry Tonkonog for his useful feedbacks which improved this paper. Special thanks to Yank{\i} Lekili for his various suggestions and bringing to my attention the paper of Yoel Groman $\cite{yg}$. I'm also grateful for Huai-Liang Chang, Jonny Evans, Yohsuke Imagi, Yi-Jen Lee, Cheuk-Yu Mak, and Weiwei Wu for their interests and encouragements.

\section{Basic structures}\label{section:structure}
This section collects some standard materials on the Hamiltonian and Lagrangan Floer theories, together with the open-closed string maps which relate these two flavors of Floer theories. Standard references concerning these topics are $\cite{ma,as,rs,ps,ps1,ps4,ns}$, see also $\cite{bc,bc1,ep}$ for some precursors.
Once the relevant maximum principles are established, the constructions of the $A_\infty$ structures and open-closed string maps in our case are essentially the same with the fundamental work of Ritter-Smith $\cite{rs}$, whether the symplectic filling $M^\mathrm{in}$ is strong plays no role. For this reason, our account here will be quite brief, detailed constructions can be found in $\cite{rs}$.

\subsection{Hamiltonian Floer theory}\label{section:ham}

In order to clarify our geometric set up, it should be suitable here to recall some standard notions about symplectic fillings. Throughout this paper, $(V,\xi)$ will be used to denote a $(2n-1)$-dimensional co-oriented closed contact manifold. Let $(M^\mathrm{in},\omega_M)$ be a compact symplectic manifold whose boundary $\partial M^\mathrm{in}\cong V$ as oriented manifolds. Denote by $\omega_\xi$ the restriction of $\omega_M$ on $\xi$. The following notion is first introduced in $\cite{mnw}$ in order to study the flexibility of tight contact structures on manifolds with dimension larger than 3.

\begin{definition}[Massot-Niederkr\"{u}ger-Wendl $\cite{mnw}$]\label{definition:weak-fill}
We say that $(M^\mathrm{in},\omega_M)$ is a weak symplectic filling of $(V,\xi)$ if for every choice of the contact form $\theta_V$ defining $\xi$,
\begin{equation}\theta_V\wedge(d\theta_V+\omega_\xi)^{n-1}>0,\theta_V\wedge\omega_\xi^{n-1}>0.\end{equation}
\end{definition}

For the purpose of having well-defined Floer theories on $M^\mathrm{in}$, we need to impose further restrictions on the symplectic filling $M$. Denote by $\omega_V$ the restriction of the symplectic form $\omega_M$ on $V$, we say that $(\omega_V,\theta_V)$ form a \textit{stable Hamiltonian structure} on $V$ if
\begin{equation}\label{eq:stab-Ham}\theta_V\wedge\omega_V^{n-1}>0,\ker\omega_V\subset\ker d\theta_V.\end{equation}
Note that in our case, since $\ker d\theta_V$ is 1-dimensional, we actually have $\ker\omega_V=\ker d\theta_V$. Because of this, we will not distinguish between the Reeb vector field associated to the stable Hamiltonian structure (namely the one generating $\ker\omega_V$ and normalized to 1 by $\theta_V$) and the usual Reeb vector field on the contact manifold $(V,\theta_V)$, and will use $R$ to denote both of them.
\begin{definition}[Latschev-Wendel $\cite{lw}$]\label{definition:stable-fill}
A weak filling $(M^\mathrm{in},\omega_M)$ of $(V,\xi)$ is said to be \textit{stable} if there exists a contact form $\theta_V$ on $(V,\xi)$ such that $(\omega_V,\theta_V)$ is a stable Hamiltonian structure on $V$.
\end{definition}

Let $V\subset M$ be an oriented hypersurface in a $2n$-dimensional symplectic manifold $(M,\omega_M)$, and $\xi\subset TV$ is the co-oriented hyperplane distribution induced by a nowhere vanishing 1-form $\theta_V$ on $V$, such that $\omega_\xi$ is symplectic and induces positive orientation. By Lemma 2.6 of $\cite{mnw}$ that a neighborhood of $V$ in $M$ is symplectomorphic to
\begin{equation}
\left((1-\varepsilon,1+\varepsilon)\times V,\omega_V+d((r-1)\theta_V)\right)
\end{equation}
for some $\varepsilon>0$, where $V\subset M$ is identified naturally as $\{1\}\times V$, and the direction of $\frac{\partial}{\partial r}$ is such that $\iota_{\frac{\partial}{\partial r}}\omega_M^n=\theta_V\wedge\omega_M^{n-1}$.

In particular, when $(M^\mathrm{in},\omega_M)$ is a weak filling of $(V,\xi)$, there exists a tubular neighborhood of $V\subset M^\mathrm{in}$ symplectomorphic to
\begin{equation}\left((1-\varepsilon,1]\times V,\omega_V+d((r-1)\theta_V)\right),\end{equation}
where $\varepsilon>0$ is taken to be sufficiently small. This is usually called a \textit{collar neighborhood} of $V$. Note that when $M^\mathrm{in}$ is a stable filling of $(V,\xi)$, every hypersurface $\{r\}\times V$, where $1-\varepsilon<r\leq1$ in the collar neighborhood is stably filled.

Using the vector field $\frac{\partial}{\partial r}$, we can extend the weak symplectic filling $M^\mathrm{in}$ to a non-compact symplectic manifold $M$ by attaching to $M^\mathrm{in}$ a \textit{cylindrical end}, namely
\begin{equation}M=M^\mathrm{in}\cup_{\partial M^\mathrm{in}}[1,\infty)\times V,\end{equation}
where the symplectic form on the cylindrical end is given by
\begin{equation}
\omega_M|M\setminus M^\mathrm{in}=\omega_V+d((r-1)\theta_V),r\in[1,\infty).
\end{equation}
We call $M$ the \textit{completion} of $M^\mathrm{in}$.\bigskip

From now on $M^\mathrm{in}$ will be a stable symplectic filling of $(V,\xi)$, whose completion $M$ satisfies the semi-positivity condition (\ref{eq:semi-positive}). Pick a Hamiltonian function $H:M\rightarrow\mathbb{R}$ which outside a compact subset of $M$ has the form $h(r)$, where $h(r)$ is linear in $r$ with positive slope $h'(r)>0$ not equal to the Reeb period. Hamiltonians which possess this shape will be called \textit{admissible}. Using the stability condition condition (\ref{eq:stab-Ham}), it is easy to see that the corresponding Hamiltonian vector field is given by $X_H=h'(r)R$ for $r\gg 0$, where $R$ is the Reeb vector field for $(V,\theta_V)$. Just as in the strong filling case, we can choose the contact form $\theta_V$ generically subject to the restriction that $(\omega_V,\theta_V)$ is still a stable Hamiltonian structure on $V$, while the Reeb periods of $R$ form a discrete period spectrum $\mathcal{P}_M\subset\mathbb{R}$. We shall assume from now on that such a choice of $\theta_V$ has been fixed.

Throughout this paper, we shall define the Floer complex $\mathit{CF}^\ast(H)$ using only contractible periodic orbits of $X_H$. Since we are mainly interested in the cases when $M$ is simply-connected, this convention should be appropriate here and it also simplifies some of our arguments below.

Let $\mathcal{L}_0M\subset\mathcal{L}M$ be the connected component of the free loop space of $M$ consisting of contractible loops. We can define a covering $\widetilde{\mathcal{L}_0}M\rightarrow\mathcal{L}_0M$ by considering the pairings $(u,x)$ where $u:\mathbb{D}\rightarrow M$ is a disc with boundary $x\in\mathcal{L}_0M$. $\widetilde{\mathcal{L}_0}M$ is then defined by moding out the equivalence relation which identifies two pairs $(u,x)$ and $(u',x')$ if $u\#\overline{u'}\in\pi_2(M)_0$, namely both $\omega_M$ and $c_1(M)$ vanish on the sphere $u\#\overline{u'}$, where by $\overline{u'}$ we mean the disc $u'$ with its orientation reversed. Recall that with this covering, an action functional
\begin{equation}A_H:\widetilde{\mathcal{L}_0}M\rightarrow\mathbb{R}\end{equation}
can be defined and its critical points correspond to the generators of the Floer complex $\mathit{CF}^\ast(H)$. On the other hand, there is also a well-defined Conley-Zehnder index $\mathrm{ind}_\mathit{CZ}(u,x)\in\mathbb{Z}$ for every element $(u,x)\in\widetilde{\mathcal{L}_0}M$, which can be used to equip the vector space $\mathit{CF}^\ast(H)$ over $\mathbb{K}$ a $\mathbb{Z}$-grading.

To define the Floer differential on $\mathit{CF}^\ast(H)$, we need to write down the Floer equation. To this end, we need to specify a set of almost complex structures to work with.
\begin{definition}\label{definition:adm-J}
An almost complex structure $J$ on $M$ tamed by $\omega_M$ in the interior of the domain $M^\mathrm{in}$ is said to be admissible if
\begin{itemize}
\item[(i)] it is of contact type on the cylindrical end, i.e. $dr\circ J=-\theta_V$;
\item[(ii)] $J$ preserves the contact structure $\xi$ and $J|\xi$ is tamed by $\omega_\xi$ and $d\theta_V$.
\end{itemize}
\end{definition}
We need the following important result proved in $\cite{mnw}$.
\begin{theorem}[Theorem D and Proposition 2.1 of $\cite{mnw}$]
$M^\mathrm{in}$ is a weak symplectic filling if and only if it admits an admissible almost complex structure $J$ in the sense of Definition \ref{definition:adm-J}. Moreover, once there exists an admissible almost complex structure on $M$, then the space of such almost complex structures, which we denote by $\mathcal{J}(M)$, is contractible.
\end{theorem}

\begin{lemma}
Let $J\in\mathcal{J}(M)$, then it is also tamed by $\omega_M$ on the cylindrical end.
\end{lemma}
\begin{proof}
To show this, one can make use of the decomposition $TM=\xi\oplus\langle R\rangle\oplus\langle\partial_r\rangle$ of the tangent bundle on the cylindrical end of $M$, and consider separately the cases when the vector field $X$ lies in $\xi$, $\langle R\rangle$ or $\langle\partial_r\rangle$.

In the case when $X$ lies in $\xi$, the fact that $\omega_M(X,JX)>0$ is a direct consequence of the condition (ii) in Definition \ref{definition:adm-J}.

When $X$ coincides with a non-zero multiple of the Reeb vector field $R$, $\omega_M(X,JX)>0$ follows from the contact type condition (i) in Definition \ref{definition:adm-J}.

In the remaining case when $X$ is a non-zero multiple of $\frac{\partial}{\partial r}$, use again condition (i) of Definition \ref{definition:adm-J}, and the simple fact that $\omega_M\left(\partial_r,R\right)=1$.
\end{proof}

Let $(S,j)$ be a Riemann surface which is topologically a punctured sphere equipped with a sub-closed 1-form $\gamma$. Fix a set of cylindrical ends $\varepsilon_k:\mathbb{R}_\pm\times S^1\rightarrow S$ for the punctures, with coordinates $(s,t)$ so that $j\partial_s=\partial_t$. With such a parametrization, $\varepsilon_k^\ast\gamma=w_kdt$ on every cylindrical end of $S$, where $w_k>0$ is the weight associated to $\varepsilon_k$. Suppose $u:S\rightarrow M$ is a smooth map which solves the equation $(du-X_H\otimes\gamma)^{0,1}=0$ (where the (0,1)-part is taken with respect to a $J\in\mathcal{J}(M)$) and converges to the 1-periodic orbits $x_1,\cdot\cdot\cdot,x_d;y$ of $X_H$ when $s\rightarrow\pm\infty$. For the purpose of defining the Floer differential on $\mathit{CF}^\ast(H)$ and the product structure on $\mathit{SH}^\ast(M)$, we need to show that for $(H,J)$ admissible, the image of any such solution $u$ lies in a compact subset of $M$ determined by $r(x_1),\cdot\cdot\cdot,r(x_d)$ and $r(y)$, the radial coordinates of the periodic orbits. The proof of the following lemma is a slight modification of that of Lemma D.1 of $\cite{ar2}$.
\begin{lemma}\label{lemma:pair-of-pants}
If the function $\rho=r\circ u$ has a local maximum, then it is constant.
\end{lemma}
\begin{proof}
Choose local holomorphic coordinates $(s,t)$ on $S$ so that it is compatible with our parametrizations on the cylindrical ends fixed above, then $\gamma=\gamma_sds+\gamma_tdt$. Using this we can rewrite the Floer equation locally as
\begin{equation}\label{eq:Floer}\partial_su+J\partial_tu=X_H\gamma_s+JX_H\gamma_t.\end{equation}
On the cylindrical end, we have the decomposition $TM=\xi\oplus\langle R\rangle\oplus\langle\partial_r\rangle$. The trick is to eliminate the effect of $\omega_V$ so that the proof goes along the same lines as in the case of strong fillings $\cite{ar2}$. To do this, write correspondingly the Floer solution $u$ as $(v,\rho)$, where $v$ takes its value in $V$. Projecting to the $\langle R\rangle$ and $\langle\partial_r\rangle$ directions one gets
\begin{equation}\left\{\begin{array}{l}\theta_V(\partial_sv)+\partial_t\rho-\gamma_sh'(\rho)=0, \\ \partial_s\rho-\theta_V(\partial_tv)+h'(\rho)\gamma_t=0.\end{array}\right.\end{equation}
From this and the fact that $d\theta_V(\partial_sv,\partial_tv)\geq0$ (since $d\theta_V$ tames $J|\xi$), we deduce
\begin{equation}\label{eq:MP}
\Delta\rho+\frac{h''(\rho)d\rho\wedge\gamma}{ds\wedge dt}\geq-\frac{h'(\rho)d\gamma}{ds\wedge dt}.
\end{equation}
Observe that the right hand side is non-negative, so the statement follows from the maximum principle for elliptic operators.
\end{proof}
Recall that in general we need to perform a $t$-dependent perturbation of $H$ on the cylindrical end to ensure that all the Hamiltonian orbits are non-degenerate, which creates an additional term $-(\partial_th_t')dt\wedge\gamma$ in the above computations. However, such a term actually vanishes by our requirement that $\gamma=w_kdt$ on the cylindrical ends, so the argument above still holds.

With Lemma \ref{lemma:pair-of-pants}, one can build a moduli space $\mathcal{M}(x;y)$ which consists of solutions of the Floer equation asymptotic to the Hamiltonian orbits $x$ and $y$, modding out the reprarametrizations $\mathbb{R}$. $\mathcal{M}(x;y)$ is a smooth manifold with expected dimension by choosing a regular $J\in\mathcal{J}_\mathit{reg}(M)\subset\mathcal{J}(M)$, whose existence is governed by the Sard-Smale theorem. One can further separate the moduli space $\mathcal{M}(x;y)$ according to the lifts of $x$ and $y$ in $\widetilde{\mathcal{L}_0}M$, which gives us another moduli space $\mathcal{M}\left(\tilde{x};\tilde{y}\right)$, where $\tilde{x},\tilde{y}\in\widetilde{\mathcal{L}_0}M$ is connected by a lift of $u$. Let
\begin{equation}E(u)=\frac{1}{2}\int_S||du-X_H\otimes\gamma||^2\mathrm{dvol}_S\end{equation}
be the energy of $u:S\rightarrow M$. Since we have assumed that $M$ is semi-positive, the maximum principle above and the a priori energy estimate
\begin{equation}\label{eq:estimate}E(u)=A_H(\widetilde{x})-A_H(\widetilde{y}),u\in\mathcal{M}(x;y)\end{equation}
ensures that $\mathcal{M}\left(\tilde{x};\tilde{y}\right)$ can be compactified by adding broken trajectories, by Gromov compactness for tame almost complex structures. The Floer differential $d$ on $\mathit{CF}^\ast(H)$ is the defined by counting rigid elements of the compactified moduli space $\overline{\mathcal{M}}\left(\tilde{x};\tilde{y}\right)$. This defines the Hamiltonian Floer cohomology group $\mathit{HF}^\ast(H)$. It is a standard continuation argument to show that $\mathit{HF}^\ast(H)$ is independent of the choice of the admissible almost complex structure $J\in\mathcal{J}(M)$.
\bigskip

To define the symplectic cohomlogy of $M$ (or $M^\mathrm{in}$), one must establish the required continuation maps. To do this, we need to specify the particular class of homotopies of the Floer data that are allowed. Let $(H_s)$ be a \textit{monotone homotopy} of admissible Hamiltonians, i.e. $\partial_sh'_s\leq0$, then a similar argument as in the proof of Lemma \ref{lemma:pair-of-pants} can be used to prove a maximum principle for the solutions of $(du-X_{H_s}\otimes\gamma)^{0,1}=0$. This in particular shows that the continuation maps
\begin{equation}\kappa:\mathit{HF}^\ast(wH)\rightarrow\mathit{HF}^\ast\left((w+1)H\right)\end{equation}
are well-defined provided that $wh'(r),(w+1)h'(r)\notin\mathcal{P}_M$. It's easy to see these Hamiltonian Floer cohomologies form a directed system indexed by all the possible slopes at infinity. Taking its direct limit we get the symplectic cohomology
\begin{equation}\mathit{SH}^\ast(M)\cong{\varinjlim}_w\mathit{HF}^\ast(wH).\end{equation}

Lemma \ref{lemma:pair-of-pants} together with an a priori energy estimate similar to (\ref{eq:estimate}) guarantees that the pair-of-pants product
\begin{equation}\mathit{HF}^\ast(w_1H)\otimes\mathit{HF}^\ast(w_\infty H)\rightarrow\mathit{HF}^\ast(w_0H)\end{equation}
can be defined, where the weights are chosen so that $w_1+w_\infty=w_0$ and
\begin{equation}w_0h'(r),w_1h'(r),w_\infty h'(r)\notin\mathcal{P}_M.\end{equation}
Passing to direct limits we get a product on $\mathit{SH}^\ast(M)$, which makes it an algebra over $\mathbb{K}$.
\bigskip

When $M$ is monotone, for the purpose of defining open-closed string maps, we remark that using the telescope construction ($\cite{as}$, Section 2), the direct limit can be taken on the chain level, which yields a complex computing $\mathit{SH}^\ast(M)$:
\begin{equation}\mathit{SC}^\ast(M):=\bigoplus_{w=1}^\infty\mathit{CF}^\ast(wH)[\mathbf{q}],\end{equation}
where $\mathbf{q}$ is a formal variable and the differential is given by
\begin{equation}\nu^1(x+\mathbf{q}y)=(-1)^{\deg(x)}dx+(-1)^{\deg(y)}(\mathbf{q}dy+\kappa y-y).\end{equation}
The definitions of the Floer differential $d:\mathit{CF}^\ast(wH)\rightarrow\mathit{CF}^\ast(wH)$, the continuation map $\kappa:\mathit{CF}^\ast(wH)\rightarrow\mathit{CF}^\ast((w+1)H)$, and the pair-of-pants product on the chain level in this set up make use of the moduli spaces of weighted popsicles, see Section \ref{section:fuk}. By Lemma \ref{lemma:pair-of-pants} and the monotonicity assumption on $M$, the same construction as in Section 4.6 of $\cite{rs}$ carries over with no modification to the stable filling case.

\subsection{Fukaya categories}\label{section:fuk}
When passing to open string invariants, we further restrict ourselves to the case when $M^\mathrm{in}$ is a stable symplectic filling whose completion $M$ is monotone, in order to avoid possible technical complexities. The Lagrangian submanifolds $L\subset M$ we shall consider are assumed to be oriented and monotone. Since we are working over a field $\mathbb{K}$ with $\mathrm{char}(\mathbb{K})\neq2$, we also require that $L$ is $\mathit{Spin}$, and actually fix a choice of $\mathit{Spin}$ structure, so that various moduli spaces appeared below will be oriented. When $L$ is non-compact, we also require that on the cylindrical end, $L$ is modelled on a Legendrian cone. More precisely,
\begin{itemize}
\item[(i)] $L$ intersects $V$ transversely along $\partial L^\mathrm{in}$, where $L^\mathrm{in}=L\cap M^\mathrm{in}$;
\item[(ii)] $\theta_V|L$ vanishes on $\partial L^\mathrm{in}\times(1-\varepsilon,\infty)$.
\end{itemize}
Lagrangian submanifolds $L\subset M$ satisfying these constraints will be called \textit{admissible}.

Let $S=\mathbb{D}\setminus\{z_0,\cdot\cdot\cdot,z_d\}$ be a disc with $d+1$ boundary punctures, and equip it with strip-like ends $\varepsilon_k:\mathbb{R}_\pm\times[0,1]\rightarrow S$ for $i=1,\cdot\cdot\cdot,d$. Denote by $\partial_iS\subset\partial S$ the $i$-th boundary component between $z_i$ and $z_{i+1}$. Another important auxiliary datum on $S$ is the choice of a sub-closed 1-form $\gamma$, such that $\varepsilon^\ast_k\gamma=dt$, $d\gamma=0$ near $\partial S$ and $\gamma|\partial S=0$.

We still work with the class of tame almost complex structures $\mathcal{J}(M)$ specified by Definition \ref{definition:adm-J}. Consider a solution $u:S\rightarrow M$ of
\begin{equation}\label{eq:CR}\left\{\begin{array}{l}(du-X_H\otimes\gamma)^{0,1}=0, \\ u(\partial_iS)\subset L_i, \\ \lim_{s\rightarrow\pm\infty}u\left(\varepsilon_k(s,\cdot)\right)=x_k,\end{array}\right.\end{equation}
where $H$ is an admissible Hamiltonian defined in Section \ref{section:ham} and $x_k$ is a time-1 chord of the Hamiltonian flow of $X_H$ with ends on two adjacent Lagrangian submanifolds corresponding to the boundary labels of $\partial S$. The proof of the following lemma is similar to that of Lemma D.2 of $\cite{ar2}$.
\begin{lemma}\label{lemma:lag-max}
The function $\rho=r\circ u$ with $u:S\rightarrow M$ a solution of (\ref{eq:CR}) can't have a local maximum unless it is constant.
\end{lemma}
\begin{proof}
One can argue similarly as in the proof of Lemma \ref{lemma:pair-of-pants}, except that in this case $\rho$ may achieve its maximum on $\partial S$. To exclude this possibility, note that with our parametrization $\partial_t$ is in the outward or inward normal direction of $\partial S$. When $\partial_t$ is outward pointing, by Hopf's lemma $\partial_t\rho>0$ at any maxima on $\partial S$. But by (\ref{eq:CR}), we have
\begin{equation}\partial_t\rho=\gamma_sh'(\rho)-\theta_V(\partial_sv)=\theta_V(\partial_sv),\end{equation}
using the fact that $\gamma|\partial S=0$. On the other hand, since $\partial_sv\in TL$ and $\theta_V$ vanishes outside a compact subset of $L$, $\partial_t\rho=0$. Contradiction. One can argue similarly when $\partial_t$ is inward pointing, and get a contradiction as well.
\end{proof}
Using this lemma one can associate various algebraic structures to admissible Lagrangian submanifolds in $M$, provided that the Lagrangian submanifolds under consideration are tautologically unobstructed. Since $M$ is monotone, it remains to consider the Floer theory of admissible Lagrangian submanifolds with minimal Maslov number 2. For this we need the following lemma.
\begin{lemma}\label{lemma:m0}
For $J\in\mathcal{J}(M)$, any $J$-holomorphic disc $u:\mathbb{D}\rightarrow M$ bounded by an admissible Lagrangian submanifold $L\subset M$ satisfies
\begin{equation}u(\partial\mathbb{D})\subset M^\mathrm{in}\cup\{r_0\}\times V\end{equation}
for some fixed $r_0>1$.
\end{lemma}
\begin{proof}
We can argue similarly as in Lemma \ref{lemma:pair-of-pants}. By projecting the Cauchy-Riemann equation to the directions $\langle R\rangle$ and $\langle\partial_r\rangle$, we get in this case $\Delta\rho\geq0$. On the other hand, by the argument in Lemma \ref{lemma:lag-max}, a local maximum can't appear on $\partial\mathbb{D}$. From these one concludes that if $u(\partial\mathbb{D})\nsubseteq M^\mathrm{in}$, then $\rho$ must be constant on $u(\partial\mathbb{D})\cap(M\setminus M^\mathrm{in})$, which proves the lemma.
\end{proof}
By the above lemma, the construction of the moduli spaces $\mathcal{M}_1(L,\beta)$ of Maslov index 2 $J$-holomorphic discs with 1 boundary marked point in the class $\beta\in\pi_2(M,L)$ reduces to the closed monotone case. Using this we can define the obstruction (modulo usual transversality and invariance issues)
\begin{equation}\label{eq:obs}
\mathfrak{m}_0(L)=\sum_{\beta\in\pi_2(M,L)}q^{\omega_M(\beta)}\mathit{ev}_\ast\left[\mathcal{M}_1(L,\beta)\right]=m_0(L)[L]
\end{equation}
for any admissible Lagrangian $L\subset M$, where $\mathit{ev}:\mathcal{M}_1(L,\beta)\rightarrow L$ is the evaluation map at the boundary marked point.

From now on all the Lagrangian submanifolds involved are assumed to be admissible and have the same $m_0$ value. Denote by $\mathit{CF}^\ast(L_0,L_1;wH)$ the free $\mathbb{K}$-module generated by time-1 chords of the flow of $wX_H$, where $w$ is an integer. Note that in order for the Floer complex $\mathit{CF}^\ast(L_0,L_1;wH)$ to be finitely generated, one needs to choose the contact form $\theta_V$ generically (subject to the condition that it defines a stable Hamiltonian structure on $V$ together with the original $\omega_V$) so that there is no Reeb chord of integer period. By counting the solutions of (\ref{eq:CR}) in the case when $S$ is a strip with one input and one output, the Floer differential
\begin{equation}\delta:\mathit{CF}^\ast(L_0,L_1;wH)\rightarrow\mathit{CF}^{\ast+1}(L_0,L_1;wH)\end{equation}
can be defined. By our assumption that $m_0(L_0)=m_0(L_1)$, $\delta^2=0$, and the associated cohomology group will be denoted by $\mathit{HF}^\ast(L_0,L_1;wH)$. As in the case of closed strings, one can take the direct limit with respect to $w$ by building the continuation maps
\begin{equation}\kappa:\mathit{HF}^\ast(L_0,L_1;wH)\rightarrow\mathit{HF}^\ast\left(L_0,L_1;(w+1)H\right),\end{equation}
these maps can be shown to be well-defined by combining the parametrized version of Lemma \ref{lemma:pair-of-pants} and Lemma \ref{lemma:lag-max} above. This defines the wrapped Floer cohomology $\mathit{HW}^\ast(L_0,L_1)$.

However, for the construction of $A_\infty$ structures, we need to work on the chain level and run the telescope construction $\cite{as}$. Analogous to $\mathit{SC}^\ast(M)$, the wrapped Floer complex is defined to be
\begin{equation}\label{eq:wrap-cpx}\mathit{CW}^\ast(L_0,L_1)=\bigoplus_{w=1}^\infty\mathit{CF}^\ast(L_0,L_1;wH)[\mathbf{q}],\end{equation}
with the differential
\begin{equation}\mu^1(x+\mathbf{q}y)=(-1)^{\deg(x)}\delta x+(-1)^{\deg(y)}(\mathbf{q}\delta y+\kappa y-y).\end{equation}
The complex $\left(\mathit{CW}^\ast(L_0,L_1),\mu^1\right)$ computes the wrapped Floer cohomology $\mathit{HW}^\ast(L_0,L_1)$.\bigskip

The admissible Lagrangian submanifolds $L\subset M$ form the objects of the wrapped Fukaya category $\mathcal{W}(M)$, and the morphisms between two admissible Lagrangian submanifolds $L_0,L_1$ are defined to be the wrapped Floer complex $\mathit{CW}^\ast(L_0,L_1)$. Note that this is only well-defined when the Lagrangian submanifolds involved have the same $m_0$. Because of this, $\mathcal{W}(M)$ is understood as a disjoint union of the full subcategories $\mathcal{W}_\lambda(M)$ consisting of admissible Lagrangian submanifolds with $m_0(L)=\lambda\in\mathbb{K}$.\bigskip

The construction of the $A_\infty$ structures on $\mathcal{W}_\lambda(M)$ is rather involved compared to the exact case $\cite{ma}$, due to the fact that there is no obvious way to bypass the telescope construction. Instead of working with Riemann surfaces $S=\mathbb{D}\setminus\{z_0,\cdot\cdot\cdot,z_d\}$ equipped with strip-like ends and sub-closed 1-forms, we need to endow $S$ with an additional structure. To describe this, fix a finite collection of labels $p_f\in\{1,\cdot\cdot\cdot,d\}$ indexed by the set $F$. This determines a map
\begin{equation}\mathbf{p}:F\rightarrow\{1,\cdot\cdot\cdot,d\}.\end{equation}
Associated to $\mathbf{p}$ there is a collection of holomorphic maps $\bm{\phi}=(\phi_f)_{f\in F}$ such that each $\phi_f:S\rightarrow\mathbb{R}\times[0,1]$ tends to an isomorphism on the compactification $\overline{S}\rightarrow\mathbb{D}$, such that
\begin{equation}\phi_f(z_0)=-\infty,\phi_f(z_{p_f})=+\infty.\end{equation}
The quadruple $(S,\bm{\varepsilon},\gamma,\bm{\phi})$ is called a \textit{weighted popsicle}, where $\bm{\varepsilon}$ is a set of strip-like ends, and $\gamma$ is a carefully chosen sub-closed 1-form whose definition involves the specifications of the weights $\mathbf{w}=\{w_0,\cdot\cdot\cdot,w_d\}$, where $w_0=\sum_{k=1}^dw_k+|F|$.\bigskip

Fix a set $\mathbf{x}=\{x_0,\cdot\cdot\cdot,x_d\}$ of Hamiltonian chords of $X_H$ with weights $\mathbf{w}$, denote by $\mathcal{M}^{d+1,\mathbf{p},\mathbf{w}}(\mathbf{x})$ the moduli space of solutions $u:S\rightarrow M$ of (\ref{eq:CR}). In order to define the $A_\infty$ structure on $\mathcal{W}_\lambda(M)$, one needs to choose a family $\mathcal{I}_{S,\bm{\phi},\mathbf{w}}(M)\subset\mathcal{J}(M)$ of domain dependent admissible almost complex structures which are compatible with the strip-like ends, see $\cite{as}$. Namely $\mathcal{I}_{S,\bm{\phi},\mathbf{w}}(M)$ varies smoothly over the moduli space of stable weighted popsicles $\mathcal{M}^{d+1,\mathbf{p},\mathbf{w}}$, and compatible with the gluing. However, to achieve transversality of the moduli spaces $\mathcal{M}^{d+1,\mathbf{p},\mathbf{w}}(\mathbf{x})$, a set of infinitesimal deformations
\begin{equation}\mathcal{K}_{S,\bm{\phi},\mathbf{w}}(M)\subset T\mathcal{J}(M),\end{equation}
with superexponential decay along the strip-like ends $\bm{\varepsilon}$ of $S$ must be introduced, where $T\mathcal{J}(M)$ denotes the tangent bundle over the infinite-dimensional manifold $\mathcal{J}(M)$. Note that although we are using tame almost complex structures instead of compatible ones, it is easy to see the fact that $\omega_M(\cdot,K\cdot)$ is no longer symmetric with $K\in\mathcal{K}_{S,\bm{\phi},\mathbf{w}}(M)$ does not affect the arguments in $\cite{as}$. Exponentiating the elements in $\mathcal{K}_{S,\bm{\phi},\mathbf{w}}(M)$ to get actual deformations of $\mathcal{I}_{S,\bm{\phi},\mathbf{w}}(M)$, we get a family of admissible almost complex structures $\mathcal{J}_{S,\bm{\phi},\mathbf{w}}(M)$, such that
\begin{equation}\mathcal{I}_{S,\bm{\phi},\mathbf{w}}(M)\subset\mathcal{J}_{S,\bm{\phi},\mathbf{w}}(M)\subset\mathcal{J}(M).\end{equation}
The upshot is: with a generic choice of almost complex structures in $\mathcal{J}_{S,\bm{\phi},\mathbf{w}}(M)$, the moduli space $\mathcal{M}^{d+1,\mathbf{p},\mathbf{w}}(\mathbf{x})$ is a smooth manifold with expected dimension.

Since we are working over a Novikov field $\mathbb{K}$, the corresponding a priori energy estimate is tautological, and Gromov compactness holds as in the usual case. Counting isolated solutions in $\mathcal{M}^{d+1,\mathbf{p},\mathbf{w}}(\mathbf{x})$ defines a map
\begin{equation}\mu^{d,\mathbf{p},\mathbf{w}}:\mathit{CF}^\ast(L_{d-1},L_d;w_dH)[\mathbf{q}]\otimes\cdot\cdot\cdot\otimes\mathit{CF}^\ast(L_0,L_1;w_1H)[\mathbf{q}]\rightarrow\mathit{CF}^\ast(L_0,L_d;w_0H)[\mathbf{q}]\end{equation}
for $d\geq2$. For details see $\cite{rs}$.

Taking the weighted sum over all the possible $\mathbf{p}$ and $\mathbf{w}$, we obtain the $A_\infty$ structure maps
\begin{equation}\mu^d_\mathcal{W}:\mathit{CW}^\ast(L_{d-1},L_d)\otimes\cdot\cdot\cdot\otimes\mathit{CW}^\ast(L_0,L_1)\rightarrow\mathit{CW}^\ast(L_0,L_d)[2-d]\end{equation}
of $\mathcal{W}_\lambda(M)$.\bigskip

On the object level, the Fukaya category of compact Lagrangians $\mathcal{F}(M)$ consists of all the closed admissible Lagrangian submanifolds in $M$. As before, $\mathcal{F}(M)$ is a disjoint union of the full subcategories $\mathcal{F}_\lambda(M)$ whose objects are closed Lagrangian submanifolds $L\subset M$ with $m_0(L)=\lambda$. The morphism between two objects $L_0,L_1$ of $\mathcal{F}_\lambda(M)$ is given by the usual Floer complex $\mathit{CF}^\ast(L_0,L_1)$ generated by chords of $X_H$ with $H$ a compactly supported Hamiltonian.

Note that the $A_\infty$ operations $\mu_\mathcal{F}^d$ on $\mathcal{F}_\lambda(M)$ are well-defined by Lemma \ref{lemma:lag-max}, without referring to popsicles and weights. However, it is an easy consequence of Lemma \ref{lemma:m0} that
\begin{equation}\mathit{HW}^\ast(L_0,L_1)\cong\mathit{HF}^\ast(L_0,L_1)\end{equation}
for $L_0,L_1$ any two objects of $\mathcal{F}_\lambda(M)$. In fact, one can even construct an $A_\infty$ functor
\begin{equation}\mathcal{A}:\mathcal{F}(M)\rightarrow\mathcal{W}(M)\end{equation}
which is cohomologically full and faithful, by allowing $w=0$ in the definition of the wrapped Floer complex (\ref{eq:wrap-cpx}). When $M^\mathrm{in}$ is a strong filling, this is the \textit{acceleration functor} defined in $\cite{rs}$.

\subsection{Open-closed maps}\label{section:OC}

As in the last subsection, our standing assumption is that $M^\mathrm{in}$ is a stable symplectic filling whose completion $M$ is monotone. One way to relate the Hamiltonian and Lagrangian flavors of Floer theory is to use the open-closed or closed-open string maps. Let $S$ be a Riemann surface with both boundary and interior punctures (resp. interior marked points), the definitions of these maps involve the study of the moduli spaces of $J$-holomorphic maps $u:S\rightarrow M$ with Lagrangian boundary conditions and asymptotic to Hamiltonian chords and orbits (resp. hitting locally finite cycles). In particular, a combination of Lemmas \ref{lemma:pair-of-pants} and \ref{lemma:lag-max} ensures that the required maximum principle holds for defining these maps.\bigskip

We first consider the case of compact Lagrangian submanifolds. Fix a disc $S$ with $d+1$ boundary punctures and an interior marked point $\ast$, which is an output and can be fixed to be the origin. Denote by $\mathit{CC}_\ast\left(\mathcal{F}_\lambda(M),\mathcal{F}_\lambda(M)\right)$ the Hochschild chain complex of $\mathcal{F}_\lambda(M)$. On the chain level, the degree $d$ open-closed map
\begin{equation}\mathit{OC}^d:\mathit{CC}_d\left(\mathcal{F}_\lambda(M),\mathcal{F}_\lambda(M)\right)\rightarrow\mathit{QC}^{d+n}(M)\end{equation}
is defined by counting the solutions $u:S\rightarrow M$ of (\ref{eq:CR}) which pass through some fixed choice of locally finite cycle $c\in\mathit{QC}_\ast^\mathit{BM}(M)$ at the interior marked point $\ast$, where $\mathit{QC}_\ast^\mathit{BM}(M)$ denotes the Borel-Moore model of the quantum chain complex of the non-compact manifold $M$. We remark that the transversality of the moduli spaces involved in this definition relies on the fact that $M$ is monotone, see Section 5.4 of $\cite{rs}$ for details. Summing over $d$ gives us a chain map, which induces on the cohomology level the open-closed string map
\begin{equation}\label{eq:OCQ}\mathit{OC}:\mathit{HH}_\ast\left(\mathcal{F}_\lambda(M),\mathcal{F}_\lambda(M)\right)\rightarrow\mathit{QH}^{\ast+n}(M).\end{equation}
For the closed-open map, consider again $S=\mathbb{D}\setminus\{z_0,\cdot\cdot\cdot,z_d\}$, but now the interior marked point $\ast\in S$ is an input, and the puncture between the boundary components $\partial S_0$ and $\partial S_n$ is an output. By counting the rigid solutions $u:S\rightarrow M$ which satisfy (\ref{eq:CR}) and an additional intersection condition at $\ast$, we get a chain map
\begin{equation}\mathit{CO}^d:\mathit{QC}^d(M)\rightarrow\mathit{CC}^d\left(\mathcal{F}_\lambda(M),\mathcal{F}_\lambda(M)\right),\end{equation}
where the right hand side is the degree $d$ Hochschild cochain complex. Summing over $d$ and passing to cohomologies yields the closed-open string map
\begin{equation}\mathit{CO}:\mathit{QH}^\ast(M)\rightarrow\mathit{HH}^\ast\left(\mathcal{F}_\lambda(M),\mathcal{F}_\lambda(M)\right).\end{equation}

The general case of (possibly non-compact) admissible Lagrangians is more complicated. Since we need to work on the chain level, the construction involves popsicles with additional interior punctures. As in the case of $A_\infty$ operations, counting the solutions of (\ref{eq:CR}) which are asymptotic to the Hamiltonian 1-orbit of $w_0X_H$ at $\ast\in S$, with specified weights $\mathbf{w}=\{w_0,\cdot\cdot\cdot,w_{d+1}\}$ and the sub-closed 1-form $\gamma$, defines a map
\begin{equation}\mathit{OC}^{d,\mathbf{p},\mathbf{w}}:\mathit{CF}^\ast(L_d,L_0;w_{d+1}H)[\mathbf{q}]\otimes\cdot\cdot\cdot\otimes\mathit{CF}^\ast(L_0,L_1;w_1H)[\mathbf{q}]\rightarrow\mathit{CF}^\ast(w_0H)[\mathbf{q}],\end{equation}
where the Lagrangian submanifolds involved are assumed to satisfy $m_0(L_i)=\lambda$ for some fixed $\lambda\in\mathbb{K}$. Refer to $\cite{rs}$ for details of this construction. Summing up the $\mathit{OC}^{d,\mathbf{p},\mathbf{w}}$'s as $\mathbf{p},\mathbf{w}$ vary yields the map
\begin{equation}\mathit{OC}^d:\mathit{CC}_d\left(\mathcal{W}_\lambda(M),\mathcal{W}_\lambda(M)\right)\rightarrow\mathit{SC}^{d+n}(M).\end{equation}
This is a chain map, so we get from this the open-closed string map
\begin{equation}\mathit{OC}:\mathit{HH}_\ast\left(\mathcal{W}_\lambda(M),\mathcal{W}_\lambda(M)\right)\rightarrow\mathit{SH}^{\ast+n}(M).\end{equation}
The construction of the closed-open map is similar, except that $\ast$ will be an input, while the puncture separates $\partial_0S$ and $\partial_nS$ is considered to be an output.\bigskip

Denote by $\mathit{Spec}\left(\star c_1(M)\right)\subset\mathbb{K}$ the set of eigenvalues of the quantum multiplication by $c_1(M)$. Recall that we have a decomposition
\begin{equation}\mathit{QH}^\ast(M)=\bigoplus_{\lambda\in\mathit{Spec}\left(\star c_1(M)\right)}\mathit{QH}^\ast(M)_\lambda.\end{equation}
Via the ring homomorphism (PSS map)
\begin{equation}c^\ast:\mathit{QH}^\ast(M)\rightarrow\mathit{SH}^\ast(M)\end{equation}
obtained by composing a sequence of continuation maps, $\mathit{SH}^\ast(M)$ can be realized as a $\mathit{QH}^\ast(M)$-module. This defines subalgebras $\mathit{SH}^\ast(M)_\lambda\subset\mathit{SH}^\ast(M)$, namely the $\lambda$-generalized eigensummand of the multiplication by $c^\ast\left(c_1(M)\right)$. Note that by definition, for monotone Lefschetz domains (Definition \ref{definition:Lefschetz}), we actually have
\begin{equation}\mathit{SH}^\ast(M)=\bigoplus_{\lambda\in\mathit{Spec}\left(\star c_1(M)\right)}\mathit{SH}^\ast(M)_\lambda,\mathit{SH}^\ast(M)_0=0,\end{equation}
and the homomorphism $c^\ast$ above is simply the localization of $\mathit{QH}^\ast(M)$ at $c_1(M)$.

The following result relates the generalized eigenvalues $\lambda$ and the $m_0$-values of the Lagrangians via the open-closed maps. Since its proof has nothing to do with the fact that the symplectic filling $M^\mathrm{in}$ may not be strong, the argument of $\cite{rs}$ extends to our case without any modification.
\begin{proposition}[Ritter-Smith $\cite{rs}$]\label{proposition:eigen}
The images of the open-closed string maps
\begin{equation}\mathit{OC}:\mathit{HH}_\ast\left(\mathcal{F}_\lambda(M),\mathcal{F}_\lambda(M)\right)\rightarrow\mathit{QH}^{\ast+n}(M),\mathit{OC}:\mathit{HH}_\ast\left(\mathcal{W}_\lambda(M),\mathcal{W}_\lambda(M)\right)\rightarrow\mathit{SH}^{\ast+n}(M)\end{equation}
lie in the generalized eigensummands $\mathit{QH}^\ast(M)_\lambda$ and $\mathit{SH}^\ast(M)_\lambda$ respectively. Similar statement holds for the closed-open string maps.
\end{proposition}
Using this fact, and the construction of the acceleration functor $\mathcal{A}$, we have the following commutative diagram:
\begin{equation}\label{eq:accel}
\xymatrix{
\mathit{HH}_\ast(\mathcal{F}_\lambda(M),\mathcal{F}_\lambda(M)) \ar[d]_{\mathit{OC}} \ar[r]^{\mathit{HH}_\ast(\mathcal{A})}
 &\mathit{HH}_\ast(\mathcal{W}_\lambda(M),\mathcal{W}_\lambda(M)) \ar[d]^{\mathit{OC}}\\
\mathit{QH}^{\ast+n}(M)_\lambda \ar[r]^{c^\ast}
 & \mathit{SH}^{\ast+n}(M)_\lambda}
\end{equation}
which appears in $\cite{rs}$ as the \textit{acceleration diagram} when $M^\mathrm{in}$ is a strong filling.

We note that (\ref{eq:accel}) is a commutative diagram of $\mathit{QH}^\ast(M)$-modules, with the $\mathit{QH}^\ast(M)$-module structures for the top line given as follows. Given a locally finite cycle $c\in\mathit{QC}_\ast^\mathit{BM}(M)$, we can define an endomorphism $\phi_c$ of the diagonal $\left(\mathcal{F}_\lambda(M),\mathcal{F}_\lambda(M)\right)$-bimodule $\left(\mathcal{D}_\lambda,\mu_\mathcal{D}^{p|1|q}\right)$ as follows. Let $S$ be a disc $\mathbb{D}$ with $d_1+1+d_2$ boundary punctures removed and an interior marked point $\ast$. Among these $d_1+1+d_2$ boundary punctures, we fix a module input at $1\in\partial\mathbb{D}$ and a module output at $-1\in\partial\mathbb{D}$. Now the components of $\partial S $ in the upper-half plane $\mathbb{H}$ are labeled by Lagrangian submanifolds $L_0,\cdot\cdot\cdot,L_{d_1}$, while the boundary components $\partial S\setminus\mathbb{H}$ are associated with Lagrangian labels $L_0',\cdot\cdot\cdot,L_{d_2}'$. All these Lagrangian submanifolds involved are objects of $\mathcal{F}_\lambda(M)$. The definition of \begin{eqnarray}
\phi_c^{d_1|1|d_2}:\mathit{CF}^\ast\left(L_{d_1},L_{d_1-1}\right)\otimes\cdot\cdot\cdot\otimes\mathit{CF}^\ast\left(L_1,L_0\right)\otimes\mathit{CF}^\ast\left(L_0,L_0'\right)\\ \nonumber
\otimes\mathit{CF}^\ast\left(L_0',L_1'\right)\otimes\cdot\cdot\cdot\otimes\mathit{CF}^\ast\left(L_{d_2-1}',L_{d_2}\right)\rightarrow\mathit{CF}^\ast\left(L_{d_1},L_{d_2}'\right)
\end{eqnarray}
is similar to the $A_\infty$ bimodule structure maps $\mu^{d_1|1|d_2}_\mathcal{D}=\pm\mu_\mathcal{F}^{d_1+d_2+1}$ of $\mathcal{D}_\lambda$, namely the Hamiltonian chords with ends on $L_0,\cdot\cdot\cdot,L_{d_1}$ act on the left, and the Hamiltonian chords with ends on $L_0',\cdot\cdot\cdot,L_{d_2}'$ act on the right. However, in this case the solutions $u:S\rightarrow M$ of (\ref{eq:CR}) are required to satisfy an additional intersection condition at $\ast$ specified by the locally finite cycle $c$. Summing up the $\phi_c^{d_1|1|d_2}$'s we get the endomorphism $\phi_c\in\mathit{End}\left(\mathcal{D}_\lambda\right)$.

By Theorem 8.1 of $\cite{rs}$, the unital $\mathbb{K}$-algebra homomrphism
\begin{equation}\mathit{QH}^\ast(M)\rightarrow H^\ast\left(\mathit{End}\left(\mathcal{D}_\lambda\right)\right)\cong\mathit{End}\left(\mathit{HH}_\ast\left(\mathcal{D}_\lambda,\mathcal{D}_\lambda\right)\right)\end{equation}
defined by $c\mapsto\phi_c$ endows $\mathit{HH}_\ast\left(\mathcal{F}_\lambda(M),\mathcal{F}_\lambda(M)\right)$ with a $\mathit{QH}^\ast(M)$-module structure.\\
Similarly, $\mathit{HH}_\ast\left(\mathcal{W}_\lambda(M),\mathcal{W}_\lambda(M)\right)$ also admits the structure of a $\mathit{QH}^\ast(M)$-module.
\begin{proposition}[Ritter-Smith $\cite{rs}$]\label{proposition:acc-mod}
The maps in the acceleration diagram (\ref{eq:accel}) are $\mathit{QH}^\ast(M)$-module homomorphisms.
\end{proposition}
For most of our applications, the most relevant are the degree 0 parts of the open-closed and closed open maps
\begin{equation}\label{eq:0part}\mathit{OC}^0:\mathit{HF}^\ast(L,L)\rightarrow\mathit{QH}^{\ast+n}(M)_\lambda,\mathit{CO}^0:\mathit{QH}^\ast(M)_\lambda\rightarrow\mathit{HF}^\ast(L,L),\end{equation}
where $L\subset M$ is a closed admissible Lagrangian submanifold with $m_0(L)=\lambda$. Notice that by definition, the right hand side of $\mathit{CO}^0$ we get $\mathit{Hom}\left(\mathbb{K},\mathit{CF}^\ast(L,L)\right)$. Since the Hochschild differential in this case can be identified with the Floer differential $\mu^1$ on $\mathit{CF}^\ast(L,L)$, by evaluating at $1\in\mathbb{K}$ one gets a map $\mathit{CO}^0$ with image in $\mathit{HF}^\ast(L,L)$. Similarly we have degree 0 part of the open-closed maps for possibly non-compact admissible Lagrangians, just replace the $\mathit{QH}^\ast(M)_\lambda$ in (\ref{eq:0part}) with $\mathit{SH}^\ast(M)_\lambda$.

Finally we need the following important lemma, which computes explicitly the image of $c_1(M)\in\mathit{QH}^2(M)$ under the closed-open map $\mathit{CO}^0$. Its proof is simply an adaptation of the argument in Lemma 9.1 of $\cite{rs}$ to our setting.
\begin{lemma}[Auroux-Kontsevich-Seidel]\label{lemma:aks}
Let $\mathit{CO}^0$ be the degree 0 open-closed string map defined in (\ref{eq:0part}), then
\begin{equation}\mathit{CO}^0\left(c_1(M)\right)=m_0(L)[L].\end{equation}
\end{lemma}
\bigskip
Our proof of Theorem \ref{theorem:main} relies on the following non-triviality result of the open-closed string map, which is a mild generalization of Theorem 12.19 in $\cite{rs}$.
\begin{proposition}\label{proposition:OC}
Suppose the monotone symplectic manifold $M$ satisfies in addition the item (i) of Definition \ref{definition:Lefschetz}, and let $L\subset M$ be a closed admissible Lagrangian submanifold with $m_0(L)\neq0$. If $pt\in C_\ast(L)$ defines a cocycle in $\mathit{HF}^\ast(L,L)$, then
\begin{equation}\mathit{OC}^0\left([pt]\right)\neq0\in\mathit{QH}^\ast(M).\end{equation}
\end{proposition}
\begin{proof}
If the symplectic filling $M^\mathrm{in}$ of $(V,\xi)$ is strong, the proposition follows from Theorem 12.19 in $\cite{rs}$. Otherwise $c_1(M)$ itself may not be exact at infinity, but we can mimick their argument with $c_1(M)$ being replaced by its quantum power
\begin{equation}c_1(M)^{\star j}=c_1(M)\star\cdot\cdot\cdot\star c_1(M)\in\mathit{QH}^{2j}(M).\end{equation}
We want to show that the class $c_1(M)^{\star j}$ can be represented by a compactly supported differential form. To see this, use the assumption that $c_1(M)^{\star j}\in H^{2j}(M;\mathbb{K})$. Since we care only about the behavior of $c_1(M)^{\star j}$ on the cylindrical end, we can restrict it to $M\setminus M^\mathrm{in}$. By our assumption, any representative of $c_1(M)^{\star j}$ defines a cohomology class in $H^{2j}(M\setminus M^\mathrm{in};\mathbb{K})$. From the obvious identification between $H^{2j}(M\setminus M^\mathrm{in};\mathbb{K})$ and $H^{2j}(V;\mathbb{K})$ and the assumption that $H^{2j}(V;\mathbb{Q})=0$, we see that this cohomology class must be trivial.

This allows us to choose a compact locally finite cycle $c\in H_\ast^{\mathit{BM}}(M)$ which represents $c_1(M)^{\star j}$, where by $H_\ast^{\mathit{BM}}(M)$ we mean the Borel-Moore homology. By Lemma \ref{lemma:aks} and the fact that $\mathit{CO}^0:\mathit{QH}^\ast(M)\rightarrow\mathit{HF}^\ast(L,L)$ is a unital algebra homomorphism, we get
\begin{equation}\mathit{CO}^0\left(c_1(M)^{\star j}\right)=m_0(L)^j[L],\end{equation}
which shows that $c$ is not a boundary in $H_\ast^{\mathit{BM}}(M)$. By definition, $m_0(L)^j$ is then the coefficient before $\mathit{PD}(c^\vee)$ in the disc counting $\mathit{OC}^0\left([pt]\right)$, where by $c^\vee\in H_\ast(M;\mathbb{K})$ we mean the dual of $c$ under intersection pairing and $\mathit{PD}$ is the Poincar\'{e} dual. More precisely,
\begin{equation}\mathit{OC}^0\left([pt]\right)=m_0(L)^j\mathit{PD}(c^\vee)+\textrm{other terms}.\end{equation}
Since the other terms on the right hand side counting Maslov index 2 discs correspond to cycles in $H_\ast(M;\mathbb{K})$ which are linear independent with $c^\vee$, they cannot cancel the term $m_0(L)^j\mathit{PD}(c^\vee)$. The disc countings with Maslov index not equal to 2 will have different powers in the Novikov paramter $q$, so they will not cancel the first term either. Now the statement follows from our assumption that $m_0(L)\neq0$.
\end{proof}

\subsection{Proof of Theorem \ref{theorem:main}}\label{section:pf-main}

We have now introduced all the algebraic tools needed in the proof of Theorem \ref{theorem:main}. Our argument follows the proof of Theorem 1.3 sketched in Section (1d) of $\cite{ps}$ closely, for which the following version of Cardy relation plays a key role.
\begin{proposition}[Cardy relation]\label{proposition:cardy}
	Let $L_1,L_2\subset M$ be two objects which belong to the same eigensumaand of the monotone Fukaya category $\mathcal{F}(M)$, and $[a_i]\in\mathit{HF}^\ast(L_i,L_i)$ for $i=1,2$. Then
	\begin{equation}\label{eq:cardy}\mathit{OC}^0\left([a_2]\right)\star\mathit{OC}^0\left([a_1]\right)=(-1)^{n(n+1)/2}\mathrm{Str}\left([a]\mapsto(-1)^{|a|\cdot|a_2|}[a_2]\cdot[a]\cdot[a_1]\right),\end{equation}
	where the left hand side is the quantum intersection product on $\mathit{QH}^\ast(M)$, while the right hand side is the supertrace of the endomorphism $\mathit{HF}^\ast(L_1,L_2)$ given by composition with $[a_1]$ and $[a_2]$.
\end{proposition}
The above Cardy relation is stated in its more general form in Section 5.2 of $\cite{ps1}$, but the special case above is enough for our applications here. Note also that the quantum product on the left hand side of (\ref{eq:cardy}) is taken to be the ordinary intersection product on $H^\ast(M;\mathbb{K})$ in $\cite{ps1}$ since the work is done in the exact category.

The proof of the above result relies on a degeneration and gluing analysis of the moduli space of annuli with two boundary marked points, lying respectively on the two boundary components. See for example, Section 11.4 of $\cite{rs}$ for a detailed analysis. The same argument extends to the current setting by Lemma \ref{lemma:lag-max}.\bigskip

Assume that $M^\mathrm{in}$ is a monotone Lefschetz domain. By assumption $\mathit{SH}^\ast(M)$ is semisimple, so we have the decomposition
\begin{equation}\label{eq:decomp}\mathit{SH}^0(M)=\bigoplus_{i\in I}\mathbb{K}v_i,\end{equation}
where $I$ is some finite index set and $(v_i)_{i\in I}$ is a collection of idempotents in $\mathit{SH}^\ast(M)$ which are orthogonal to each other. Here we are using the $\mathbb{Z}/2$-grading on $\mathit{SH}^\ast(M)$, so $\mathit{SH}^0(M)$ is a commutative ring. By Lemma \ref{lemma:pair-of-pants}, there is a well-defined pair-of-pants product on $\mathit{SH}^\ast(M)$, which restricts to one on the even degree part $\mathit{SH}^0(M)$. By assumption, the pair-of-pants product is non-trivial on each summand $\mathbb{K}v_i$.\bigskip

For any closed admissible Lagrangian submanifold $L\subset M$, consider the open-closed string map $\mathit{OC}^0$ composed with the unital algebra homomorphism $c^\ast$ induced by the localization of $\mathit{QH}^\ast(M)$ at $c_1(M)$:
\begin{equation}\label{eq:OCSH}
\mathit{HF}^\ast(L,L)\xrightarrow{\mathit{OC}^0}\mathit{QH}^{\ast+n}(M)\xrightarrow{c^\ast}\mathit{SH}^{\ast+n}(M).
\end{equation}
We can decompose the even part of the quantum cohomology $\mathit{QH}^0(M)$ and hence $\mathit{SH}^0(M)$ into different generalized eigenspaces with respect to $\star c_1(M)$ and working in a fixed eigensummand of the Fukaya category $\mathcal{F}_\lambda(M)$. By Proposition \ref{proposition:eigen}, the image of the map $c^\ast\circ\mathit{OC}^0$ then lies in the summand $\mathit{SH}^{\ast+n}(M)_\lambda$.\bigskip

Let $L$ be a closed admissible Lagrangian submanifold with $m_0(L)=\lambda\neq0$, consider the restriction of the map (\ref{eq:OCSH}) to the degree $n$ part of Floer cohomology
\begin{equation}
c^\ast\circ\mathit{OC}^0:\mathit{HF}^n(L,L)\rightarrow\mathit{SH}^0(M)_\lambda\subset\bigoplus_{i\in I}\mathbb{K}v_i.
\end{equation}
By Proposition \ref{proposition:acc-mod}, both of the maps $\mathit{OC}^0$ and $c^\ast$ respect the $\mathit{QH}^\ast(M)$-module structure, it follows that the image of $c^\ast\circ\mathit{OC}^0$ in $\mathit{SH}^0(M)_\lambda$ consists of a subset of the summands in (\ref{eq:decomp}). Now Proposition \ref{proposition:OC} together with the assumption that $\mathit{SH}^\ast(M)_\lambda\cong\mathit{QH}^\ast(M)_\lambda$ for any $\lambda\neq0$ (condition (ii) of Definition \ref{definition:Lefschetz}) shows that this subset is non-empty if $[\mathit{pt}]\in\mathit{HF}^n(L,L)$ is non-trivial. Now let $L_i,L_j\subset M$ be two closed monotone Lagrangian submanifolds with $m_0(L_i)=m_0(L_j)=\lambda$ which are disjoinable by Hamiltonian isotopy, then $\mathit{HF}^\ast(L_i,L_j)$ is well-defined and vanishes. By Proposition \ref{proposition:cardy}, the images of $\mathit{HF}^n(L_i,L_i)$ and $\mathit{HF}^n(L_j,L_j)$ under $\mathit{OC}^0$ must be mutually orthogonal with respect to the quantum product on $\mathit{QH}^0(M)_\lambda$. This is also true for the images of $\mathit{HF}^n(L_i,L_i)$ and $\mathit{HF}^n(L_j,L_j)$ in $\mathit{SH}^0(M)_\lambda$ under the composition $c^\ast\circ\mathit{OC}^0$ since $c^\ast$ is a $\mathbb{K}$-algebra homomorphism. By the non-vanishing of the pair-of-pants product on each summand $\mathbb{K}v_i\subset\mathit{SH}^0(M)_\lambda$, we see that the images of $\mathit{HF}^n(L_i,L_i)$ and $\mathit{HF}^n(L_j,L_j)$ under $c^\ast\circ\mathit{OC}^0$ must lie in different summands of $\mathit{SH}^0(M)_\lambda$, therefore the number of disjoinable Lagrangians in $\mathcal{F}(M)_\lambda$ is bounded by $\dim_\mathbb{K}\mathit{SH}^0(M)_\lambda$. Collecting all the eigensummands together finishes the proof.

\subsection{Contact toric manifolds}\label{section:con-tor}
We digress a little bit and discuss briefly the vanishing condition $H^{2j}(V;\mathbb{Q})=0$, where $1\leq j\leq n-1$ is a fixed integer. This plays a role when applying Theorems \ref{theorem:main} and \ref{theorem:generation} to concrete examples. See Section \ref{section:Lefschetz}.\bigskip

The concept of a \textit{contact toric manifold} is introduced by Lerman $\cite{el}$. For a $(2n-1)$-dimensional closed contact manifold $(V,\xi)$, this means that $V$ carries an effective $T^n$-action which preserves the contact structure $\xi$.

Consider the positive half of the symplectization of $(V,\xi)$. Concretely this is defined by
\begin{equation}S\xi:=\left(\mathbb{R}_+^\ast\times V,dr\wedge\theta_V+rd\theta_V\right),\end{equation}
where $r\in\mathbb{R}_+^\ast$ and $\theta_V$ is a $T^n$-invariant contact form. One can lift the $T^n$-action on $V$ to $T^\ast V$, since $S\xi\subset T^\ast V$ is preserved under this action, we see that $S\xi$ is a symplectic toric manifold. Denote by
\begin{equation}\mu_{S\xi}:S\xi\rightarrow(\mathfrak{t}^\vee)^n\cong\mathbb{R}^n\end{equation}
the toric moment map. The \textit{moment cone} associated to $(V,\xi)$ is defined to be the set
\begin{equation}C_V:=\mu_{S\xi}(S\xi)\cup\{0\}.\end{equation}
\begin{definition}[$\cite{el}$]
A rational polyhedral cone
\begin{equation}C=\left\{x\in\mathbb{R}^n|\langle x,v_i\rangle\geq0,\forall i\in I\right\},\end{equation}
where $I$ is a finite index set, and $v_i\in\mathbb{Z}^n$ is said to be good if
\begin{itemize}
\item every codimension $m$ face is the intersection of $m$ facets of $C$;
\item for every subset $J\subset I$, $\mathbb{Z}\langle (v_j)_{j\in J}\rangle$ is a direct summand of $\mathbb{Z}^n$ with rank $|J|$.
\end{itemize}
\end{definition}
A contact toric manifold $(V,\xi)$ is called \textit{good} if $\dim_\mathbb{R}(V)>3$ and $C_V$ is a strictly convex good cone. Note that when $\dim_\mathbb{R}(V)=3$ and the $T^2$-action on $V$ is not free, then it is known that $V$ is diffeomorphic to a lens space $\cite{el}$.

Using a linear transformation in $\mathit{SL}(n,\mathbb{Z})$ we can place $C_V\setminus\{0\}$ in the upper half space $\mathbb{R}^{n-1}\times\mathbb{R}_+^\ast$. To ensure that $V$ is a circle bundle over a smooth toric manifold, one needs to impose the additional requirement that the intersection between $C_V$ and the hyperplane $\mathbf{H}:=\mathbb{R}^{n-1}\times\{1\}$ is a \textit{Delzant polytope}. With these assumptions, the following result follows essentially by applying the Gysin sequence.
\begin{proposition}[Luo $\cite{sl}$]
For a good contact toric manifold $(V,\xi)$ such that $C_V\cap\mathbf{H}$ is Delzant,
\begin{equation}H^{2j}(V;\mathbb{Q})=0,\lceil n/2\rceil\leq j\leq n-1.\end{equation}
\end{proposition}
This contains the following example as a special case.
\begin{lemma}\label{lemma:boundary}
Let $(V,\xi_{BW})$ be the total space of the unit sphere bundle associated to $\mathcal{O}(-1)^{\oplus n_1}\rightarrow\mathbb{CP}^{n_2}$ equipped with its standard contact structure. Then $(V,\xi_{BW})$ is a good contact toric manifold such that $C_V\cap\mathbf{H}$ is a Delzant polytope.
\end{lemma}
\begin{proof}
Since the total space of $\mathcal{O}(-1,-1)\rightarrow\mathbb{CP}^{n_1-1}\times\mathbb{CP}^{n_2}$ can be regarded as the blow-up of $\mathcal{O}(-1)^{\oplus n_1}\rightarrow\mathbb{CP}^{n_2}$ along the subvariety $\mathbb{CP}^{n_1-1}\subset\mathbb{CP}^{n_2}$ in the zero section, we have an identification between the total spaces of $\mathcal{O}(-1,-1)\rightarrow\mathbb{CP}^{n_1-1}\times\mathbb{CP}^{n_2}$ and $\mathcal{O}(-1)^{\oplus n_1}\rightarrow\mathbb{CP}^{n_2}$ away from their zero sections. It follows that the unit disc bundle of $\mathcal{O}(-1,-1)\rightarrow\mathbb{CP}^{n_1-1}\times\mathbb{CP}^{n_2}$ gives a strong symplectic filling of the unit sphere bundle associated to $\mathcal{O}(-1)^{\oplus n_1}\rightarrow\mathbb{CP}^{n_2}$, from which it is easy to see that $C_V\cap\mathbf{H}$ is the moment polytope associated to $\mathbb{CP}^{n_1-1}\times\mathbb{CP}^{n_2}$.
\end{proof}
Note that the same is true for the more general case of negative vector bundles $\mathcal{O}(-m)^{\oplus n_1}\rightarrow\mathbb{CP}^{n_2}$. Since their ideal contact boundaries $(V,\xi_{BW})$ are circle bundles and can be strongly filled by the unit disc bundle of $\mathcal{O}(-1,-m)\rightarrow\mathbb{CP}^{n_1-1}\times\mathbb{CP}^{n_2}$, so $C_V\cap\mathbf{H}$ is still the moment polytope of $\mathbb{CP}^{n_1-1}\times\mathbb{CP}^{n_2}$.

\section{Semisimplicity}\label{section:semisimplicity}
This section studies another major issue of this paper, namely symplectic manifolds with semisimple symplectic cohomologies. We introduce the surgery of boundary connected sums between stable symplectic fillings, and then prove that symplectic cohomology is well-behaved under such a surgery. Based on this we further investigate some special cases of reverse $\mathit{MMP}$ transitions and show that they can be used to construct new examples of symplectic manifolds whose symplectic cohomologies are semisimple.

\subsection{Handle attachment}\label{section:handle}
The idea of attaching handles to Weinstein domains dates back to $\cite{aw}$. In $\cite{gnw}$, the authors observed that such a surgery can be generalized to the case of weak symplectic fillings. See also $\cite{gz}$. For our purposes, we recall here only the construction for 1-handles.
\bigskip

Start with the more general case when $(M^\mathrm{in},\omega_M)$ is a weak symplectic filling of a $(2n-1)$-dimensional contact manifold $(V,\xi)$. For any pair of distinct points $v_+,v_-\in V$, there exist neighborhoods $U_\pm\subset M^\mathrm{in}$ containing $v_\pm$ such that one can deform $\omega_M$ in a collar neighborhood of $M^\mathrm{in}$ to $\omega_M'$ so that the new symplectic filling $(M^\mathrm{in},\omega_M')$ becomes locally strong in $U_\pm$. More precisely, there are Liouville vector fields $Z_\pm$ on $U_\pm$ so that the restrictions of their duals are the contact form
\begin{equation}\theta_V|U_\pm\cap V=\iota_{Z_\pm}\omega_M'|U_\pm\cap V.\end{equation}
The existence of such a deformation is ensured by Lemma 2.10 of $\cite{mnw}$.

To construct a symplectic structure on the 1-handle $H_1$, equip $\mathbb{R}^{2n}$ with coordinates
\begin{equation}\left(x^-,y^-,x_1^+,y_1^+,\cdot\cdot\cdot,x_{n-1}^+,y_{n-1}^+\right).\end{equation}
Take $H_1(\delta)=[-1,1]\times\mathbb{B}(\delta)^{2n-1}\subset\mathbb{R}^{2n}$, where $\mathbb{B}(\delta)$ denotes a ball of radius $\delta$. We use $y^-$ to represent the coordinate on the interval $[-1,1]$. Denote respectively by $\partial_-H_1(\delta)$ and $\partial_+H_1(\delta)$ the following boundary components of $H_1(\delta)$:
\begin{equation}\partial_-H_1(\delta)=\left(\{-1\}\cup\{1\}\right)\times\mathbb{B}(\delta)^{2n-1},\partial_+H_1(\delta)=[-1,1]\times\partial\mathbb{B}(\delta)^{2n-1}.\end{equation}
The two point set $\{\pm1\}\times\{0\}$ is called the \textit{core} of $\partial_-H_1(\delta)$, where $0\in\mathbb{B}(\delta)^{2n-1}$ is the center of the ball, while the core of $\partial_+H_1(\delta)$ is the $2n-2$ sphere $\{0\}\times\partial\mathbb{B}(\delta)^{2n-1}$.

Following Section 3 of $\cite{aw}$, we equip $\mathbb{R}^{2n}$ with the symplectic form
\begin{equation}\omega_{\mathbb{R}^{2n}}=2dx^-\wedge dy^-+4\sum_{i=1}^{n-1}dx_i^+\wedge dy_i^+.\end{equation}
Note that the coefficient before $2dx^-\wedge dy^-$ above is reduced, which follows the convention of Section 3.1, $\cite{gnw}$. $\omega_{\mathbb{R}^{2n}}$ has a primitive
\begin{equation}\theta_{\mathbb{R}^{2n}}=4x^-dy^-+2y^-dx^-+2\sum_{i=1}^{n-1}\left(x_i^+dy_i^+-y_i^+dx_i^+\right),\end{equation}
and the corresponding Liouville vector field is given by
\begin{equation}Z_{\mathbb{R}^{2n}}=2x^-\frac{\partial}{\partial x^-}-y^-\frac{\partial}{\partial y^-}+\frac{1}{2}\sum_{i=1}^{n-1}\left(x_i^+\frac{\partial}{\partial x_i^+}+y_i^+\frac{\partial}{\partial y_i^+}\right).\end{equation}
The vector field $Z_{\mathbb{R}^{2n}}$ points outwards along $\partial_+H_1(\delta)$ so that it will serve as part of the contact boundary for the new weak filling $M^\mathrm{in}_1$ after attaching the handle. It is inward-pointing along $\partial_-H_1(\delta)$, so we can glue this concave boundary component to part of the pseudo-convex boundary of the original weak filling $M^\mathrm{in}$.

For $\delta>0$ small enough, $\left(\partial_-H_1(\delta),\theta_{\mathbb{R}^{2n}}\right)$ is contact isomorphic to small neighborhoods $U_\pm(\delta)\subset U_\pm$ of $v_\pm$. It follows that there is a contact form $\theta_V$ for $\xi$ so that we can glue $\partial_-H_1(\delta)$ to $U_\pm(\delta)$ to obtain a symplectic manifold with corners
\begin{equation}(M^\mathrm{in},\omega_M)\cup_{\partial_-H_1(\delta)}\left(H_1(\delta),\omega_{\mathbb{R}^{2n}}\right),\end{equation}
which weakly fills
\begin{equation}\label{eq:newbun}\left(V\setminus U_\pm(\delta),\xi\right)\cup\left(\partial_+H_1(\delta),\xi_\partial\right),\end{equation}
where $\xi_\partial$ is the contact structure induced by $\theta_{\mathbb{R}^{2n}}$. After rounding off the corners we get the desired weak symplectic filling $M_1^\mathrm{in}$ of the contact boundary $(V_1,\xi_1)$. This completes the construction.\bigskip

Since we are interested in the behaviors of symplectic cohomologies under handle attachment, we need to show that the above surgery can be done within the category of stable fillings, rather than just weak ones, so that $\mathit{SH}^\ast(M^\mathrm{in},\omega_M')$ is well-defined for semi-positive $M$. The following is a slight modification of Corollary 2.12 in $\cite{mnw}$.
\begin{proposition}\label{proposition:stabledef}
There exists a deformation $\omega_M'$ of $\omega_M$ so that $\omega_V'=d\theta_V$ when restricted to $V\cap U_\pm$ and $(\omega_V',\theta_V)$ form a stable Hamiltonian structure on $V$.
\end{proposition}
\begin{proof}
Without loss of generality, we can assume that $[\omega_V]\in H^2(V;\mathbb{Q})$ represents a rational cohomology class. By Proposition 2.18 of $\cite{cv}$, we can find a closed 2-form $\omega_V'$ in $[\omega_V]$ so that $(\omega_V',\theta_V)$ form a stable Hamiltonian structure on $V$. Fix a large constant $C\gg0$, recall that
\begin{equation}\omega_V'=d\theta_V+\frac{1}{C} t_F,\end{equation}
where $t_F$ is a Thom form associated to some tubular neighborhood $D_S$ of a codimention 2 contact submanifold $S\subset V$ representing a multiple of the class $\mathit{PD}\left([\omega_V]\right)$, and $F$ is a function compactly supported in the fiber of $D_S\rightarrow S$, so that
\begin{equation}\mathrm{supp}(t_F)\subset\mathrm{supp}(F).\end{equation}
If $v_+$ or $v_-$ lies on $S$, we can perturb $S$ a little bit to another contact submanifold $S'\subset V$ so that $S'\cap\{v_+,v_-\}=\emptyset$ and $[S]=[S']$ in $H_{2n-3}(V;\mathbb{Z})$. Since we can shrink the $\mathrm{supp}(F)$ to make it sufficiently small, it follows that there are small neighborhoods $U_\pm$ of $v_\pm$ so that $U_\pm\cap\mathrm{supp}(t_F)=\emptyset$. Now $\omega_V'$ is the 2-form we want, since it is cohomologous to $\omega_V$, Lemma 2.10 of $\cite{mnw}$ implies the existence of such a deformation.
\end{proof}
\begin{corollary}\label{corollary:stable}
Let $(M^\mathrm{in},\omega_M)$ be a stable filling of $(V,\xi)$, then $M_1^\mathrm{in}$ is a stable filling of $(V_1,\xi_1)$.
\end{corollary}
\begin{proof}
By Proposition \ref{proposition:stabledef}, $(\omega_V',\theta_V)$ is a stable Hamiltonian structure on $V$. On the other hand, the symplectic structure on the handle $H_1(\delta)$ strongly fills the boundary component $\partial_+H_1(\delta)$. The result now follows from (\ref{eq:newbun}).
\end{proof}
From now on we shall always assume that the deformation of $\omega_M$ we made in the construction of $M_1^\mathrm{in}$ satisfies the requirement of Proposition \ref{proposition:stabledef}. This ensures that $\mathit{SH}^\ast(M_1^\mathrm{in})$ is well-defined.

\subsection{Proof of Theorem \ref{theorem:flip}}\label{section:attachment}
From now on assume $(M^\mathrm{in},\omega_M)$ is a semi-positive stable filling, and $(M^\mathrm{in},\omega_M')$ is the deformation provided by Proposition \ref{proposition:stabledef}. Before attaching the handle, we need the following lemma which shows that the deformation of the symplectic structure we done on the collar of $M^\mathrm{in}$ does not affect its symplectic cohomology.
\begin{lemma}\label{lemma:triviality}
There is an isomorphism of $\mathbb{K}$-algebras
\begin{equation}\mathit{SH}^\ast(M,\omega_M)\cong\mathit{SH}^\ast(M,\omega_M').\end{equation}
\end{lemma}
\begin{proof}
Note first that since $\omega_M'=\omega_M+\beta$ is a deformation of $\omega_M$ in some collar neighborhood of $M^\mathrm{in}$, $M$ equipped with $\omega_M'$ is also semi-positive, in particular $\mathit{SH}^\ast(M,\omega_M')$ is well-defined. The argument is similar to Remark 35 of $\cite{ar3}$. More precisely, one can shrink slightly the domain $M^\mathrm{in}$ using the radial vector field $\frac{\partial}{\partial r}$ on the collar to a subdomain $U^\mathrm{in}\subset M^\mathrm{in}$ so that $\beta=0$ on $U^\mathrm{in}$. Moreover, we can replace every Hamiltonian $H:M\rightarrow\mathbb{R}$ involved in the definition of $\mathit{SH}^\ast(M)$ with another Hamiltonian $H_U$ so that $H_U$ has the same slope as $H$ at infinity and reaches such a slope in the collar of $U^\mathrm{in}$. Using these Hamiltonians, the Floer solutions involved in the definition of $\mathit{HF}^\ast\left(H_U\right)$ are all contained in $U^\mathrm{in}$, where $\beta=0$. In view of Lemma \ref{lemma:pair-of-pants}, the same is true for the Floer trajectories defining the pair-of-pants product on $\mathit{SH}^\ast(M)$, so the isomorphism also preserves the ring structure.
\end{proof}
Attaching the handle $H_1(\delta)$ to $M^\mathrm{in}$ gives us the new stable filling $M_1^\mathrm{in}$. It's easy to see the completion $M_1$ of $M_1^\mathrm{in}$ is also semi-positive, so $\mathit{SH}^\ast(M_1)$ is well-defined. We now imitate the argument of $\cite{kc}$ to prove Theorem \ref{theorem:flip}.\bigskip

One can regard $\left(M^\mathrm{in},\omega_M'\right)$ as a subdomain of $M_1^\mathrm{in}$ by enlarging the latter one a little bit using the vector field $\frac{\partial}{\partial r}$. Consider a cofinal family of Hamiltonians $K_m:M^\mathrm{in}\rightarrow\mathbb{R}$ such that for any fixed $m$, $K_m<0$ and is $C^2$-small in the interior of $M^\mathrm{in}$ and is linear with slope $a_m>0$ on the collar neighborhood $(1,1+\varepsilon]\times V$, with $a_m\notin\mathcal{P}_M$ for any $m$. For convenience, we take $K_m$ to be a small perturbation of the piecewise smooth function
\begin{equation}K_m=\left\{\begin{array}{ll}a_mr, & \textrm{on }\{r\geq1\}\times V, \\ 0, & \mathrm{otherwise}.\end{array}\right.\end{equation}
Note that we have tacitly extended the domain $M^\mathrm{in}$ to $M^\mathrm{in}\cup(1,1+\varepsilon]\times V$ with $\varepsilon>0$ sufficiently small so that it is still a subdomain of $M_1^\mathrm{in}$. For convenience, we do not change the notation after such an extension.

Assuming $\delta>0$ is taken to be sufficiently small, one can extend the Hamiltonian $K_m$ over the handle $H_1(\delta)$ to obtain a Hamiltonian $H_m:M_1^\mathrm{in}\rightarrow\mathbb{R}$ such that $H_m$ has slope $b_m>a_m$ in a collar neighborhood of $\partial M_1^\mathrm{in}$, and the only 1-periodic orbit of $X_{H_i}$ on $H_1(\delta)$ is a critical point $x_m$ in its interior. With this definition, $(H_m)$ forms a cofinal family of Hamiltonians on $M_1^\mathrm{in}$. Note that we always assume $\left\{b_m\right\}\cap\mathcal{P}_{M_1}=\emptyset$. One can arrange so that when taking $\delta>0$ to be very small, $\partial_+H_1(\delta)$ is always transverse to the Liouville vector fields $Z_\pm$ existing locally, so the parameter $\delta$ does not affect the completion $M_1$ of $M_1^\mathrm{in}$. For more details of this step, see $\cite{kc}$ or $\cite{mm}$.\bigskip

Recall that the action functional in the non-exact case is given by
\begin{equation}A_H(u,x)=\int_{\mathbb{D}^2}u^\ast\omega_M-\int_{S^1}H\left(t,x(t)\right)dt,\end{equation}
where $(u,x)\in\widetilde{\mathcal{L}_0}M$ is defined in Section \ref{section:seidel}, and $H:S^1\times M\rightarrow\mathbb{R}$ is a time-dependent Hamiltonian function.

For any critical point $x$ of $K_m$ in the interior of $M^\mathrm{in}$, we have
\begin{equation}\int_{S^1}K_m\left(t,x(t)\right)dt=K_m(x)<0,\end{equation}
which implies that $A_{K_m}(u_x,x)>0$. Similarly, for the case of a non-constant 1-periodic orbit $x$ of $X_{K_m}$ in the boundary $V=\partial M^\mathrm{in}$, we also have $A_{K_m}(u,x)>0$ for any choice of $u:\mathbb{D}^2\rightarrow M^\mathrm{in}$ since $\int_{\mathbb{D}^2}u^\ast\omega_M'>0$ and the value of $K_m$ can be made sufficiently small near $\partial M^\mathrm{in}$.

This proves that the only generator of the Floer complex $\mathit{CF}^\ast(H_m)$ with negative action is the critical point $x_m\in H_1(\delta)$ (with the constant disc $u_{x_m}$ in the lift $\widetilde{\mathcal{L}_0}M$ of $\mathcal{L}_0M$). Since the Floer differential decreases the symplectic action, $x_m$ generates a subcomplex $\mathit{CF}^\ast(M_1,M;H_m)\subset\mathit{CF}^\ast(H_m)$ of the Floer complex. One can take the direct limit of its cohomology
\begin{equation}\mathit{SH}^\ast(M_1,M):={\varinjlim}_m\mathit{HF}^\ast(M_1,M;H_m).\end{equation}
This is called the \textit{relative symplectic cohomology} in $\cite{kc}$, and is identical to the symplectic cohomology associated to the \textit{stable symplectic cobordism} $M_1^\mathrm{in}\setminus M^\mathrm{in}$ defined in $\cite{co}$.\bigskip

To show that $\mathit{SH}^\ast(M)\cong\mathit{SH}^\ast(M_1)$, we first prove that $\mathit{SH}^\ast(M_1,M)=0$. Note that although in general none of $c_1(M)$ and $c_1(M_1)$ will vanish, the difference $M_1^\mathrm{in}\setminus M^\mathrm{in}$ is a stable cobordism whose completion $M_1\setminus M$ satisfies $c_1(M_1\setminus M)=0$ and is in fact exact, so one can use the Conley-Zehnder index $\mathrm{ind}_\mathit{CZ}(x_m)$ to equip $\mathit{SH}^\ast(M_1,M)$ with a $\mathbb{Z}$-grading. When $m\rightarrow\infty$, the slope of $H_m$ increases and so does $\mathrm{ind}_\mathit{CZ}(x_m)$, from which one sees that the unique generator $x_m$ in the complex $\mathit{CF}^\ast(M_1,M;H_m)$ does not live in the direct limit, namely $\mathit{SH}^\ast(M_1,M)=0$.\bigskip

It remains to show that any Floer trajectory $u:S^1\times\mathbb{R}\rightarrow M_1^\mathrm{in}$ defining the Floer differential on Hamiltonian orbits of $X_{H_m}$ with positive action must be contained in $M^\mathrm{in}\subsetneq M_1^\mathrm{in}$. In the exact case, this follows from Lemma 7.2 of $\cite{as}$. In our case, one can prove this by applying Lemma $\ref{lemma:pair-of-pants}$ to the stable filling $\left(M^\mathrm{in},\omega_M'\right)$. Since all the Hamiltonian orbits of $X_{H_m}$ with $A_{H_m}(u,x)>0$ are contained in $M^\mathrm{in}$, Lemma $\ref{lemma:pair-of-pants}$ shows that the value of the function $r\circ u:M_1\rightarrow\mathbb{R}_+$ cannot be larger than $1+\varepsilon$. A similar application of Lemma $\ref{lemma:pair-of-pants}$ to $J$-holomorphic maps from pair-of-pants surfaces $u:P\rightarrow M_1^\mathrm{in}$ shows that the product structure on Hamiltonian orbits with positive actions is unaffected under handle attachment. Combining with the vanishing of $\mathit{SH}^\ast(M_1,M)$, this shows that there is a $\mathbb{K}$-algebra isomorphism $\mathit{SH}^\ast(M)\cong\mathit{SH}^\ast(M_1)$.\bigskip

Note that although one needs to perturb the Hamiltonian $H$ so that it becomes $t$-dependent in order for the orbits of $X_H$ to be non-degenerate, without loss of generality one can assume that $H$ stays the same near $\partial H_1(\delta)$, so it does not affect the validity of Lemma \ref{lemma:pair-of-pants}.\bigskip

Now let $M^\mathrm{in}$ and $(M')^\mathrm{in}$ be stable fillings of $(V,\xi)$ and $(V',\xi')$ respectively, and whose completions $M$ and $M'$ are semi-positive. Choose points $v\in V$ and $v'\in V'$, by applying the modifications of the symplectic structures on the collar neighborhoods of $M^\mathrm{in}$ and $(M')^\mathrm{in}$ as described in Section \ref{section:handle}, the boundary connected sum $M^\mathrm{in}\#_\partial(M')^\mathrm{in}$ can be defined by connecting the neighborhoods $U(\delta)\subset M^\mathrm{in}$ and $U'(\delta)\subset(M')^\mathrm{in}$ containing $v$ and $v'$ respectively using a 1-handle $H_1(\delta)\subset\mathbb{R}^{2n}$. Denote by $M\#_\partial M'$ the completion of $M^\mathrm{in}\#_\partial(M')^\mathrm{in}$, the above argument together with Lemma \ref{lemma:triviality} shows that
\begin{equation}\mathit{SH}^\ast(M\#_\partial M')\cong\mathit{SH}^\ast(M)\oplus\mathit{SH}^\ast(M')\end{equation}
as rings.

\subsection{Seidel representation}\label{section:seidel}
We establish here a mild extension of Ritter's generalization of Seidel representation $\cite{ps3}$ on convex symplectic manifolds. This will be used in the next subsection to do computations for the quantum and symplectic cohomologies in certain non-convex cases.\bigskip

In this subsection, $M$ will be the completion of a stable filling $M^\mathrm{in}$ of the contact boundary $(V,\xi)$ which carries an additional structure, namely a Hamiltonian circle action $g$ compatible with the Reeb flow at infinity. For the purpose of discussing Seidel representation, we shall further assume that $M$ is \textit{strongly semi-positive}, by which we will mean
\begin{equation}
\label{eq:ssemi-positive}2-n\leq c_1(M)([u])<0\Rightarrow \omega_{M}\left([u]\right)\leq0.
\end{equation}
This stronger assumption (called \textit{weak+ monotonicity} by Ritter) is imposed so that the transversality argument in Section 5.5 of $\cite{ar}$ holds.\bigskip

Denote by $\mathit{Ham}_\ell(M,\omega_M)$ the space of Hamiltonians $H:M\rightarrow\mathbb{R}$ of the form
\begin{equation}H(v,r)=f(v)r\end{equation}
for $r\gg0$, where $v\in V$ and $f:V\rightarrow\mathbb{R}$ is invariant under the Reeb flow. Similarly one can define $\mathit{Ham}_{\ell\geq0}(M,\omega_M)$ and $\mathit{Ham}_{\ell>0}(M,\omega_M)$ by requiring that $f(v)\geq0$ and $f(v)>0$. These Hamiltonians are referred to as \textit{weakly admissible} in $\cite{mm}$. The motivation of this enlargement (rather than requiring that $f(v)$ is a constant, as in the usual definition of $\mathit{SH}^\ast(M)$) is that the Hamiltonians induced by rotations about toric divisors considered below (see Section \ref{section:computation}) are in general not linear at infinity. However, by an extension of the maximum principle proved by Ritter in Appendix C of $\cite{ar1}$, these Hamiltonians can be used to compute $\mathit{SH}^\ast(M)$ when $M^\mathrm{in}$ is a strong filling. Unfortunately, his argument does not work in the general case of stable fillings as the circle action $g$ may not preserve the contact form $\theta_V$ due to the non-exactness of $\omega_V$. Because of this, we present an alternative argument below, which works also for stable fillings.
\begin{lemma}\label{lemma:ext-maximum}
Let $(H_m)\subset\mathit{Ham}_{\ell\geq0}(M,\omega_M)$ be a cofinal sequence of weakly admissible Hamiltonians of the form $f_m(v)r$ at infinity such that $f_m:V\rightarrow\mathbb{R}$ involved in the definitions are invariant under the Reeb flow, and $\lim_{m\rightarrow\infty}\min_{v\in V}f_m(v)=\infty$. Under the additional assumption that
\begin{equation}\label{eq:max-min}\max_{v\in V}f_m(v)\leq\min_{v\in V}f_{m+1}(v),\end{equation}
the direct limit ${\varinjlim}_m\mathit{HF}^\ast(H_m)$ can be defined and is isomorphic to the symplectic cohomology group $\mathit{SH}^\ast(M)$.
\end{lemma}
\begin{proof}
Fix a weakly admissible Hamiltonian $H\in\mathit{Ham}_{\ell\geq0}(M,\omega_M)$, since $f$ is invariant under the Reeb flow, the projections of the equation $(du-X_H\otimes\gamma)^{0,1}=0$ to the directions $\langle R\rangle$ and $\langle\partial_r\rangle$ are well-defined, in the sense that the splitting $TM=\xi\oplus\langle R\rangle\oplus\langle\partial_r\rangle$ is preserved by the flow of $X_H$. Thus one can argue similarly as in Lemma \ref{lemma:pair-of-pants} to show that maximum principle holds for the solutions $u:S\rightarrow M$. Note that one needs to require that $f(v)\geq0$ if $\gamma$ is not closed, since the right hand side of (\ref{eq:MP}) now takes the form $-\frac{f(v)d\gamma}{ds\wedge dt}$.

Choose a homotopy $(f_s)$ which interpolates $f(v)$ and $\min_{v\in V}f(v)$ such that each $f_s(v)$ is invariant under the Reeb flow. By projecting to $\langle R\rangle$ and $\langle\partial_r\rangle$ we get the equation
\begin{equation}\Delta\rho+h_s''\partial_s\rho+\partial_sf_s\geq0.\end{equation}
Therefore in order to ensure that the continuation map $\mathit{HF}^\ast(K_\mathrm{min})\rightarrow\mathit{HF}^\ast(H(v,r))$ is well-defined, where $K_\mathrm{min}$ is the Hamiltonian which has the form $\left(\min_{v\in V}f(v)\right)r$ for $r\gg0$, it suffices to require $\partial_sf_s\leq0$.

Similarly, under the condition that $\partial_sf_s\leq0$, one can find a homotopy $(f_s)$ invariant under the Reeb flow which interpolates $\max_{v\in V}f(v)$ and $f(v)$, so there is a continuation map $\mathit{HF}^\ast(H(v,r))\rightarrow\mathit{HF}^\ast(K_\mathrm{max})$, where $K_\mathrm{max}$ is of the form $\left(\max_{v\in V}f(v)\right)r$ at infinity. From this it is easy to see that under the condition $(\ref{eq:max-min})$, there exists a sequence of Hamiltonians $(K_m)$ which are linear for $r\gg0$ with increasing slopes such that all the continuation maps $\mathit{HF}^\ast(K_m)\rightarrow\mathit{HF}^\ast(H_m)$ and $\mathit{HF}^\ast(H_m)\rightarrow\mathit{HF}^\ast(K_{m+1})$ are well-defined. In particular, the continuation maps $\mathit{HF}^\ast(H_m)\rightarrow\mathit{HF}^\ast(H_{m+1})$ are well-defined. Passing to direct limits, we get ${\varinjlim}_m\mathit{HF}^\ast(H_m)\cong\mathit{SH}^\ast(M)$.
\end{proof}
Moreover, we need the following, whose proof is an adaptation of the argument of Theorem A.15 in the appendix of $\cite{ar2}$.
\begin{proposition}\label{proposition:TQFT}
The above isomorphism
\begin{equation}{\varinjlim}_m\mathit{HF}^\ast(H_m)\cong\mathit{SH}^\ast(M)\end{equation}
preserves product structures on both sides.
\end{proposition}
\begin{proof}
Let $K$ be a Hamiltonian on $M$ linear at infinity, and $H_s:M\rightarrow\mathbb{R}$ be an $s$-dependent Hamiltonian of the form $h_s=f_s(v)r$ when $r\gg0$, just as in the proof of Lemma \ref{lemma:ext-maximum}. To prove the statement, it is enough to show that the following diagram commutes, where $w_1+w_\infty=w_0$.
\begin{equation}\label{eq:pair-of-pants}
\xymatrix{
\mathit{HF}^\ast(w_1K)\otimes\mathit{HF}^\ast(w_\infty K) \ar[d]_{\varphi_1\otimes\varphi_\infty} \ar[r]^-{\psi_P}
 &\mathit{HF}^\ast(w_0K) \\
\mathit{HF}^\ast(w_1'H_s)\otimes\mathit{HF}^\ast(w_\infty'H_s) \ar[r]^-{\psi_P'}
 & \mathit{HF}^\ast(w_0'H_s) \ar[u]_{\kappa_0} }
\end{equation}
In the above, the vertical arrows are continuation maps, while the map $\psi_P'$ is defined in the same way as $\psi_P$ using the non-linear Hamiltonian $H_s$. By Lemma \ref{lemma:ext-maximum}, the continuation maps $\varphi_1$, $\varphi_\infty$ and $\kappa_0$ are well-defined under the conditions $w_1K'\ll w_1'h_s'$, $w_\infty K'\ll w_\infty'h_s'$ and $w_0'h_s'\ll w_0K'$ when $r\gg0$. Suppose the weights and homotopies have been arranged in this way, then passing to direct limits one sees that the resulting commutative diagram gives precisely what we want.

From now on the argument is completely identical to Theorem A.15 of $\cite{ar2}$. Namely consider the glued surface defining the operation $\kappa_0\circ\psi_P'\circ(\varphi_1\otimes\varphi_\infty)$ by attaching certain continuation cylinders corresponding to $\kappa_0,\varphi_1$ and $\varphi_\infty$ to the surface $P$ defining the pair-of-pants product $\psi_P'$, one can choose a 1-parameter interpolation
\begin{equation}(P_\lambda,\gamma_\lambda,K_\lambda,J_\lambda)_{0\leq\lambda\leq1}\end{equation}
which connects this glued surface together with the corresponding auxiliary data to that of $(P,\gamma,K,J)$ defining $\psi_P$. The fact that any solutions $u:P_\lambda\rightarrow M$ of the equation
\begin{equation}\left(du-X_{K_\lambda}\otimes\gamma_\lambda\right)^{0,1}=0\end{equation}
satisfies the maximum principle can be proved by combining Lemma \ref{lemma:pair-of-pants} and Lemma \ref{lemma:ext-maximum}. This gives us a 1-parameter family of moduli spaces $\mathcal{M}_\lambda(x_1,x_\infty;x_0)$, from which one can extract a chain homotopy from $\kappa_0\circ\psi_P'\circ(\varphi_1\otimes\varphi_\infty)$ to $\psi_P$ on the chain level, which yields the commutative diagram (\ref{eq:pair-of-pants}) after passing to cohomologies.
\end{proof}
With the above facts established, the whole machinery of $\cite{ar}$ and $\cite{ar1}$ can be generalized to our case. It seems appropriate to collect here some important facts which will be useful later.\bigskip

Given $g\in\pi_1\left(\mathit{Ham}_\ell(M,\omega_M)\right)$, we can lift the action of $g$ on $\mathcal{L}_0M$ to $\widetilde{\mathcal{L}_0}M$ to obtain a map $\tilde{g}:\widetilde{\mathcal{L}_0}M\rightarrow\widetilde{\mathcal{L}_0}M$. Any two lifts of $g$ differ by an element in the deck transformation group $\pi_2(M)/\pi_2(M)_0$. This defines an extension of $\pi_1\left(\mathit{Ham}_\ell(M,\omega_M)\right)$:
\begin{equation}1\rightarrow\pi_2(M)/\pi_2(M)_0\rightarrow\widetilde{\pi_1}\left(\mathit{Ham}_\ell(M,\omega_M)\right)\rightarrow\pi_1\left(\mathit{Ham}_\ell(M,\omega_M)\right)\rightarrow1.\end{equation}
For any lift $\tilde{g}$ of $g$, it acts on the generators of the Floer complex $\mathit{CF}^\ast(H)$ by $x\mapsto\tilde{g}^{-1}\cdot x$, which results in an isomorphism
\begin{equation}\mathcal{R}_{\tilde{g}}:\mathit{HF}^\ast(H)\rightarrow\mathit{HF}^{\ast+2I(\tilde{g})}(g^\ast H),\end{equation}
where the index $I(\tilde{g})$ is defined in the usual way by trivializing $u^\ast TM$ along the boundary of $\mathbb{D}^2$, and computing the degree of the loop induced by $\tilde{g}$ inside $\mathit{Sp}(2n,\mathbb{R})$. Taking direct limit with respect to the slope of $H$ at infinity, we obtain a $\mathbb{K}$-algebra automorphism of $\mathit{SH}^\ast(M)$, whose inverse is induced from $\tilde{g}^{-1}$.

To adapt Seidel's original approach $\cite{ps3}$ to the current setting, assume further that the time-dependent Hamiltonian $H_g:S^1\times M\rightarrow\mathbb{R}$ associated to $g$ belongs to $\mathit{Ham}_{\ell\geq0}(M,\omega_M)$. This assumption ensures that the action of $g$ on the original Hamiltonian $H$ has the effect of slowing down the flow of $X_H$ at infinity, which enables us to construct a continuation map
\begin{equation}\label{eq:continuation}\kappa_H:\mathit{HF}^\ast(g^\ast H)\rightarrow\mathit{HF}^\ast(H).\end{equation}
Denote by $H_\varepsilon$ a Hamiltonian with very small slope $\varepsilon>0$ at infinity. Composing with $\mathcal{R}_{\tilde{g}}$ and taking $H=H_\varepsilon$ in (\ref{eq:continuation}), we get a $\mathbb{K}$-algebra homomorphism
\begin{equation}\mathcal{S}_{\tilde{g}}:\mathit{QH}^\ast(M)\rightarrow\mathit{QH}^{\ast+2I(\tilde{g})}(M),\end{equation}
where we used the identification $\mathit{HF}^\ast(H_\varepsilon)\cong\mathit{QH}^\ast(M)$. Since $H_{g^{-1}}$ does not necessarily lie in $\mathit{Ham}_{\ell\geq0}(M,\omega_M)$, $\mathcal{S}_{\tilde{g}}(1)$ may not be invertible.

The following result is due to Ritter when $M^\mathrm{in}$ is a strong filling, see $\cite{ar,ar1}$.
\begin{theorem}[Seidel representation]\label{theorem:seidel}
There is a group homomorphism
\begin{equation}\mathcal{R}:\widetilde{\pi_1}\left(\mathit{Ham}_\ell(M,\omega_M)\right)\rightarrow\mathit{SH}^\ast(M)^\times\end{equation}
given by $\tilde{g}\mapsto\mathcal{R}_{\tilde{g}}(1)$, where $\mathit{SH}^\ast(M)^\times$ denotes the invertible elements with respect to the product structure on $\mathit{SH}^\ast(M)$.\\
On the other hand, the homomorphism
\begin{equation}\mathcal{S}:\widetilde{\pi_1}\left(\mathit{Ham}_{\ell\geq0}(M,\omega_M)\right)\rightarrow\mathit{QH}^\ast(M)\end{equation}
is only well-defined for Hamiltonians with non-negative slopes when $r\gg0$.
\end{theorem}
Theorem \ref{theorem:seidel} enables us to identify the continuation maps $\mathit{HF}^\ast(H_{m-1})\rightarrow\mathit{HF}^\ast(H_m)$ for a sequence of weakly admissible Hamiltonians $(H_m)$ compatible with the Reeb flow with the isomorphisms $\mathcal{R}_{\tilde{g}}^{-m}\circ\kappa_{H_\varepsilon}\circ\mathcal{R}_{\tilde{g}}^m$, such an explicit realizations of continuation maps makes the computation of the symplectic cohomologies possible. The following is an easy consequence whose proof reduces essentially to linear algebra.
\begin{corollary}[Ritter $\cite{ar,ar1}$]\label{corollary:localization}
For any loop $g\in\pi_1\left(\mathit{Ham}_{\ell>0}(M,\omega_M)\right)$,
\begin{equation}\mathit{SH}^\ast(M)\cong\mathit{QH}^\ast(M)/\ker\mathcal{S}_{\tilde{g}}^d,\end{equation}
as algebras over $\mathbb{K}$, where $d$ is a positive integer taken to be sufficiently large.
\end{corollary}
It remains to determine the image $\mathcal{S}_{\tilde{g}}(1)$ in $\mathit{QH}^\ast(M)$. Let $F(g)$ be the fixed locus of the Hamiltonian $S^1$-action $g$, then $dg$ acts on the tangent space $T_xM$ with $x\in F(g)$ as a unitary matrix, which preserves the quotient space $T_xM/T_xF(g)\cong\mathbb{C}^{n-f}$, where $f=\dim_\mathbb{C}F(g)$.
\begin{lemma}[Ritter $\cite{ar,ar1}$]\label{lemma:fix}
Let $\tilde{g}$ be the lift of $g$ which maps the pair $(u_x,x)$ to itself, where $x\in F(g)$ and $u_x$ is the constant disc associated to $x$. If $dg\in\mathit{Hom}_\mathbb{C}(\mathbb{C}^{n-f},\mathbb{C}^{n-f})$ has degree 1, then
\begin{equation}\label{eq:Sg}
\mathcal{S}_{\tilde{g}}(1)=\mathit{PD}\left[F(g)\right]+\textrm{higher order } q \textrm{ terms}\in\mathit{QH}^{2n-2f}(M).
\end{equation}
Moreover, if $f=1$, then the higher order $q$ terms in (\ref{eq:Sg}) vanish.
\end{lemma}

\subsection{Local computations}\label{section:computation}

This subsection is mainly devoted to the computation of $\mathit{QH}^\ast(E_+)$ and $\mathit{SH}^\ast(E_+)$ when $E_+$ is the total space of the vector bundle $\mathcal{O}(-1)^{\oplus n_1}\rightarrow\mathbb{CP}^{n_2}$, where $1\leq n_1\leq n_2$. Note that these symplectic manifolds are just the exceptional pieces arising from blow-ups of points and standard reverse flips. For $n_1=1$, this has been done in $\cite{ar}$. Our computations will recover this as a special case.\bigskip

For a monotone negative vector bundle $\mathcal{V}\rightarrow B$ over a closed monotone symplectic manifold, $\cite{ar}$ considers the Seidel representation associated to the Hamiltonian $S^1$-action $h$ given by rotation along the fibers, and proves that $\mathcal{S}_{\tilde{h}}(1)\in\mathit{QH}^\ast(\mathcal{V})$ is given by the pullback of the top Chern class of $\mathcal{V}$. Since in our case the bundle splits and the total space $E_+$ is toric, we can consider at the same time the Hamiltonian $S^1$-actions given by rotations about the toric divisors of $E_+$ respectively, and obtain a toric description of the element $\mathcal{S}_{\tilde{h}}(1)\in\mathit{QH}^{2n_1}(E_+)$, which makes the explicit computation of $\mathit{SH}^\ast(E_+)$ possible.

Following Section 11 of $\cite{ar}$, we equip the total space of $p:E_+\rightarrow\mathbb{CP}^{n_2}$ with the symplectic form
\begin{equation}\omega_{E_+}=p^\ast\omega_\mathit{FS}+\pi\cdot\omega_{\mathbb{C}^{n_1}},\end{equation}
where $\omega_\mathit{FS}$ is the Fubini-Study form on $\mathbb{CP}^{n_2}$ and $\omega_{\mathbb{C}^{n_1}}$ means here a 2-form on $E_+$ which restricts to the standard area form on fibers and vanishes on the base. For cohomological reasons one sees that $\left(E_+,\omega_{E_+}\right)$ is not conical at infinity when $n_1\geq2$. In particular, it is not admissible in the sense of $\cite{ar1}$.

We remark that $\mathit{QH}^\ast(E_+)$ and $\mathit{SH}^\ast(E_+)$ are well-defined. This is proved for monotone negative vector bundles in $\cite{ar}$. Alternatively, one can proceed by noticing that $\omega_{E_+}$ restricted to the unit sphere bundle $\partial E_+^\mathrm{in}$, together with a suitable choice of the contact form $\theta_{BW}$ for the contact structure $\xi_{BW}$ form a stable Hamiltonian structure on $\partial E_+^\mathrm{in}$. So the general argument for stable symplectic fillings presented in Section \ref{section:ham} solves the problem.\bigskip

The standard reference for the geometry of toric vector bundles is Section 7.3 of $\cite{cls}$. Let $D_1,\cdot\cdot\cdot,D_{n_2+1}$ be the toric divisors of $\mathbb{CP}^{n_2}$. Denote by $p_i,1\leq i\leq n_1$ the bundle projection
\begin{equation}\mathcal{O}(-1)\oplus\cdot\cdot\cdot\oplus\widehat{\mathcal{O}(-1)}\oplus\cdot\cdot\cdot\oplus\mathcal{O}(-1)\rightarrow\mathbb{CP}^{n_2}\end{equation}
which misses the $i$-th summand. Now the toric divisors of $E_+$ are given by
\begin{equation}p^{-1}(D_1),\cdot\cdot\cdot,p^{-1}(D_{n_2+1}),p_1^{-1}(\mathbb{CP}^{n_2}),\cdot\cdot\cdot,p_{n_1}^{-1}(\mathbb{CP}^{n_2}).\end{equation}
Recall that the toric divisors are in 1-1 correspondence to homogenous coordinates on $E_+$, and we denote these coordinates by
\begin{equation}x_1,\cdot\cdot\cdot,x_{n_2+1};y_1,\cdot\cdot\cdot,y_{n_1}.\end{equation}
For each tuple $(\tau_1,\cdot\cdot\cdot,\tau_{n_2+1};\zeta_1,\cdot\cdot\cdot,\zeta_{n_1})\in\mathbb{Z}^{n_2+1}\times\mathbb{Z}^{n_1}$, we can associate to it an Hamiltonian $S^1$-action on $E_+$ given by
\begin{equation}x_1\mapsto e^{2\pi i\tau_1t}x_1,\cdot\cdot\cdot,x_{n_2+1}\mapsto e^{2\pi i\tau_{n_2+1}t}x_{n_2+1},y_1\mapsto e^{2\pi i\zeta_1t}y_1,\cdot\cdot\cdot,y_{n_1}\mapsto e^{2\pi i\zeta_{n_1}t}y_{n_1}.\end{equation}
Picking the standard basis of $\mathbb{Z}^{n_2+1}\times\mathbb{Z}^{n_1}$ we get rotations
\begin{equation}\label{eq:action}g_1,\cdot\cdot\cdot,g_{n_2+1},h_1,\cdot\cdot\cdot,h_{n_1}\in\pi_1\left(\mathit{Ham}(E_+,\omega_{E_+})\right)\end{equation}
about the toric divisors of $E_+$. Note that $\prod_{i=1}^{n_2+1}g_i=1$. Define $h:=\prod_{j=1}^{n_1}h_j$.
\bigskip

The next result serves as a replacement for the last two properties in the definition of an admissible non-compact toric manifold introduced in $\cite{ar1}$. Its proof is similar to Theorem 4.5 of $\cite{ar1}$.
\begin{lemma}
For any fixed $t\in S^1$, Denote by $H_1,\cdot\cdot\cdot,H_{n_2+1},K_1,\cdot\cdot\cdot,K_{n_1}$ the Hamiltonian functions which define the Hamiltonian $S^1$-actions in (\ref{eq:action}), and by $K:=\sum_{j=1}^{n_1}K_j$ the Hamiltonian function which define the $S^1$-action $h$. Then $H_1,\cdot\cdot\cdot,H_{n_2+1}\in\mathit{Ham}_{\ell\geq0}(E_+,\omega_{E_+})$ and $K\in\mathit{Ham}_{\ell>0}(E_+,\omega_{E_+})$.
\end{lemma}
\begin{proof}
It is easy to see that the last $n_1$ componets of the toric moment map $\mu_{E_+}:E_+\rightarrow\mathbb{R}^n$ are given in homogenous coordinates by
\begin{equation}\left(x_1,\cdot\cdot\cdot,x_{n_2+1},y_1,\cdot\cdot\cdot,y_{n_1}\right)\rightarrow\left(|y_1|^2/2,\cdot\cdot\cdot,|y_{n_1}|^2/2\right).\end{equation}
Denote by $\mathbb{CP}^{n_2}$ and $\mathbb{CP}^{n_1-1}\times\mathbb{CP}^{n_2}$ the zero sections of $E_+$ and $\mathcal{O}_{\mathbb{P}^{n_1-1}\times\mathbb{P}^{n_2}}(-1,-1)$ respectively. It follows from Section 11.2 of $\cite{ar}$ that there is an identification between $E_+\setminus\mathbb{CP}^{n_2}$ and $\mathcal{O}(-1,-1)\setminus\left(\mathbb{CP}^{n_1-1}\times\mathbb{CP}^{n_2}\right)$ which preserves the radial coordinate. Denote by $|w|$ the radial coordinate on the negative line bundle $\mathcal{O}(-1,-1)\rightarrow\mathbb{CP}^{n_1-1}\times\mathbb{CP}^{n_2}$, it then follows that
\begin{equation}\label{eq:rel-r}
\sum_{j=1}^{n_1}|y_j|^2=|w|^2.
\end{equation}
On the other hand, it is proved in Lemma 45 of $\cite{ar}$ that $r=\frac{1}{1+\pi}\left(1+|w|^2/2\right)$. Combined with (\ref{eq:rel-r}) we have
\begin{equation}\label{eq:radical}r=\frac{1}{1+\pi}\left(1+\sum_{j=1}^{n_1}|y|_j^2/2\right).\end{equation}
But for points outside the locus $\mu_{E_+}^{-1}(0)$, we have $K_j=|y_j|^2/2$. Using (\ref{eq:radical}) and the fact that $K:=\sum_{j=1}^{n_1}K_j$, we see that $K=(1+\pi)r-1$, which has strictly positive slope at infinity.

All the Hamiltonians $H_i$ arise from the base $\mathbb{CP}^{n_2}$, and the fact that they have non-negative slope at infinity essentially follows from the non-negativity of the corresponding Hamiltonians on the base. More precisely, we can lift the Hamiltonian vector fields $X_{H_i}$ to $\mathbb{C}^{n+1}$, so that they become the standard angular vector fields $\frac{\partial}{\partial\theta_i}$. Since the bundle map $p:E_+\rightarrow\mathbb{CP}^{n_2}$ corresponds precisely to the projection to the $x$-coordinates, one deduces easily that $dp\cdot\frac{\partial}{\partial\theta_i}=X_{B_i}$, where $B_i:\mathbb{CP}^{n_2}\rightarrow\mathbb{R}_+$ are the corresponding Hamiltonians on the base. Combining with the expression of the symplectic form $\omega_{E_+}$ we deduce
\begin{equation}
\begin{split}
dH_i&=\omega_{E_+}\left(\cdot,\frac{\partial}{\partial\theta_i}\right) \\
&=(1+\pi r^2)p^\ast\omega_{\mathbb{C}^n}\left(\cdot,\frac{\partial}{\partial\theta_i}\right)+\pi\left(d(r^2)\wedge\theta\right)\left(\cdot,\frac{\partial}{\partial\theta_i}\right) \\
&=C_1rd\left(p^\ast B_i\right)+C_2\omega_{\mathbb{C}^{n_1}}\left(\frac{\partial}{\partial\theta_i}\right),
\end{split}
\end{equation}
where $C_1,C_2>0$ are some constants, and $(r,\theta)$ is the polar coordinate on the fiber of the line bundle $\mathcal{O}(-1,-1)\rightarrow\mathbb{CP}^{n_1-1}\times\mathbb{CP}^{n_2}$. Integrating on both sides we get $H_i=C_1rp^\ast B_i$. Notice that $p^\ast B_i$ is a function on $(\partial E_+^\mathrm{in},\xi_{BW})$ which is invariant under the Reeb flow and $B_i\geq0$, the claim follows.
\end{proof}

It is clear from out proof that $K_j$ does not lie in $\mathit{Ham}_{\ell>0}(E_+,\omega_{E_+})$ for any $j$ as long as $n_1>1$. It is therefore necessary to consider the Hamiltonian $S^1$-action $h$ instead of any individual $h_j$.
\bigskip

Now our computation reduces essentially to standard toric geometry and the general theory of Seidel representation recalled in Section \ref{section:seidel}. By Theorem 3.1, we then have well-defined elements
\begin{equation}\mathcal{S}_{\tilde{g}_i}(1),\mathcal{S}_{\tilde{h}}(1)\in\mathit{QH}^\ast(E_+),\mathcal{R}_{\tilde{g}_i}(1),\mathcal{R}_{\tilde{h}}(1)\in\mathit{SH}^\ast(E_+)^\times.\end{equation}
One can normalize the choices of lifts $\tilde{g}_i$ and $\tilde{h}$ of $g_i$ and $h$ so that the constant pairing $(u_x,x)\in\widetilde{\mathcal{L}_0}E_+$ maps to itself under the actions of $\tilde{g}_i$ and $\tilde{h}$. By Lemma \ref{lemma:fix},
\begin{equation}\mathcal{S}_{\tilde{g}_i}(1)=\mathit{PD}\left[p^{-1}(D_i)\right],\mathcal{S}_{\tilde{h}}(1)=\mathit{PD}\left[\mathbb{CP}^{n_2}\right]+\textrm{higher order } q \textrm{ terms}.\end{equation}
For our particular choice of $h\in\pi_1\left(\mathit{Ham}_{\ell>0}(E_+,\omega_{E_+})\right)$, it follows Theorem 72 of $\cite{ar}$ that the higher order $q$ terms in the expression of $\mathcal{S}_{\tilde{h}}(1)$ actually vanish, therefore $\mathcal{S}_{\tilde{h}}(1)\in\mathit{QH}^{2n_1}(E_+)$ represents the pullback of top Chern class of $E_+$ via the bundle map $p:E_+\rightarrow\mathbb{CP}^{n_2}$.
\bigskip

The rest of the computation can be completed by applying standard toric techniques, see for example $\cite{mt}$.
\begin{proposition}\label{proposition:computation}
Let $E_+$ be the total space of $\mathcal{O}(-1)^{\oplus n_1}\rightarrow\mathbb{CP}^{n_2}$, then
\begin{equation}\mathit{QH}^\ast(E_+)\cong\mathbb{K}[x]/\left(x^{n_2+1}-(-1)^{n_1}q^{n_2+1-n_1}x^{n_1}\right)\end{equation}
and
\begin{equation}\mathit{SH}^\ast(E_+)\cong\mathbb{K}[x]/\left(x^{n_2+1-n_1}-(-1)^{n_1}q^{n_2+1-n_1}\right).\end{equation}
\end{proposition}
\begin{proof}
It follows from Section 7.3 of $\cite{cls}$ that the rays of the fan of $E_+$ are given by
\begin{eqnarray}
& e_1=(b_1,1,\cdots,0),\cdots,e_{n_1}=(b_{n_1},0\cdots,1), \nonumber \\
& e_{n_1+1}=(b_{n_1+1},0,\cdots,0),\cdots,e_{n_2+1}=(b_{n_2+1},0,\cdot\cdot\cdot,0), \\
& f_1=(0,1,\cdot\cdot\cdot,0),\cdot\cdot\cdot,f_{n_1}=(0,0,\cdot\cdot\cdot,1)\in\mathbb{Z}^{n_1+n_2+1}, \nonumber
\end{eqnarray}
where $b_1,\cdot\cdot\cdot,b_{n_2+1}$ are rays of the fan of $\mathbb{CP}^{n_2}$. So the new linear relations are given by
\begin{equation}x_i+y_1=0,\cdot\cdot\cdot,x_i+y_{n_1}=0,i=1,\cdots,n_2+1.\end{equation}
It's easy to see that in our case the primitive set of $E_+$ is precisely the primitive set of the base $\mathbb{CP}^{n_2}$, so we have the relation \begin{equation}\sum_{i=1}^{n_2+1}e_i=\sum_{j=1}^{n_1}f_j\end{equation}
among the rays of the fan. This gives rise to the quantum Stanley-Reisner relation
\begin{equation}\prod_{i=1}^{n_2+1}x_i=q^{n_2+1-n_1}\prod_{j=1}^{n_1}y_j.\end{equation}
Using the original linear relations $x_i=x$ for all $1\leq i\leq n_2+1$ for the base $\mathbb{CP}^{n_2}$ to simplify the above relations, we get the presentation of $\mathit{QH}^\ast(E_+)$.

By Corollary \ref{corollary:localization}, $\mathit{SH}^\ast(E_+)$ is the localization of $\mathit{QH}^\ast(E_+)$ at $c_{n_1}(E_+)$, i.e. at $x^{n_1}$ by our computations of $\mathit{QH}^\ast(E_+)$, from which the result follows.
\end{proof}

\begin{corollary}\label{corollary:Lefschetz}
Let $j=\lceil n/2\rceil$, then $c_1(E_+)^{\star j}\in H^{2j}(E_+;\mathbb{K})$. In particular, $E_+^\mathrm{in}$ is a Lefschetz domain at level $\lceil n/2\rceil$ in the sense of Definition \ref{definition:Lefschetz}.
\end{corollary}
\begin{proof}
$c_1(E_+)$ is represented by a non-zero multiple of $x$ in the above presentation of $\mathit{QH}^\ast(E_+)$, so our computation in Proposition \ref{proposition:computation} shows that $\mathit{SH}^\ast(E_+)$ is in fact a localization of $\mathit{QH}^\ast(E_+)$ at $c_1(E_+)$. $c_1(E_+)^{\star j}$, as well as $x^j$, generates $H^{2j}(E_+;\mathbb{K})$. Note that $H^{2j}(E_+;\mathbb{K})\neq0$ follows from the assumptions that $j=\lceil n/2\rceil$ and $n_1\leq n_2$. The fact that $E_+^\mathrm{in}$ is a Lefschetz domain at level $\lceil n/2\rceil$ then follows from Lemma \ref{lemma:boundary}.
\end{proof}
Exactly the same method works for the negative vector bundles $\mathcal{O}(-m)^{\oplus n_1}\rightarrow\mathbb{CP}^{n_2}$, where $m\geq1$, $mn_1\leq n_2$ and $n_1+n_2=n$. The result is
\begin{equation}\mathit{QH}^\ast\left(\mathcal{O}_{\mathbb{P}^{n_2}}(-m)^{\oplus n_1}\right)\cong\mathbb{K}[x]/\left(x^{n_2+1}-q^{n_2+1-mn_1}(-m)^{mn_1}x^{mn_1}\right),\end{equation}
and
\begin{equation}\mathit{SH}^\ast\left(\mathcal{O}_{\mathbb{P}^{n_2}}(-m)^{\oplus n_1}\right)\cong\mathbb{K}[x]/\left(x^{n_2+1-mn_1}-q^{n_2+1-mn_1}(-m)^{mn_1}\right).\end{equation}
These computations imply that one can pick any $j$ with $\lceil\frac{n}{2}\rceil\leq j\leq\lceil\frac{mn}{m+1}\rceil$ to turn the total space of the unit ball bundle $\mathcal{O}(-m)^{\oplus n_1}_{\leq1}\rightarrow\mathbb{CP}^{n_2}$ into a Lefschetz domain.

In particular, for the negative vector bundles $\mathcal{O}(-m)^{\oplus n/(m+1)}\rightarrow\mathbb{CP}^{mn/(m+1)}$ appeared in Corollary \ref{corollary:wide}, their symplectic cohomologies are 1-dimensional, and are located in the even degree.

\subsection{Symplectomorphisms of contact type}
If one thinks on the topological level, taking $M'$ to be $E_+$ in Theorem \ref{theorem:flip} should yield some interesting special cases of Conjecture \ref{conjecture:mmp}. However, to do this geometrically, it still remains to modify the symplectic structures on the completion of $M^\mathrm{in}\#_\partial E_+^\mathrm{in}$, so that they coincide with the ones obtained by blow-ups or reverse flips. To study the behavior of $\mathit{SH}^\ast(M)$ under certain deformations of symplectic forms, it is useful to extend the invariance property of $\mathit{SH}^\ast(M)$ under a particular class of symplectomorphisms to the stable filling case.\bigskip

The class of symplectomorphisms we are interested in is called \textit{symplectomorphisms of contact type at infinity}. Roughly speaking, it consists of the symplectomorphisms which preserve the contact structures on the boundaries. It is defined in $\cite{ps2}$ for Liouville manifolds and can be generalized to our case as follows.
\begin{definition}
Let $M^\mathrm{in}$ and $(M')^\mathrm{in}$ be stable symplectic fillings of the contact manifolds $(V,\xi)$ and $(V',\xi')$. A symplectomorphism $\phi:M\rightarrow M'$ is of contact type at infinity if on a collar neighborhood $\phi$ is of the form
\begin{equation}\label{eq:contact}\phi(r,y)=\left(r-f(v),\eta(v)\right),\end{equation}
where $f:V\rightarrow\mathbb{R}$ is a smooth function, and $\eta:V\rightarrow V'$ is a contactomorphism such that $\eta^\ast\theta_{V'}=e^{f(v)}\theta_V$.
\end{definition}
With this concept, Theorem 8 of $\cite{ar3}$ can be generalized as follows.
\begin{proposition}\label{proposition:contact-type}
Let $\phi:M\rightarrow M'$ be a symplectomorphism of contact type at infinity, assume $M$ and $M'$ are semi-positive, then $\mathit{SH}^\ast(M)\cong\mathit{SH}^\ast(M')$ as $\mathbb{K}$-algebras.
\end{proposition}
\begin{proof}
The key point is to have a version of the maximum principle similar to Lemma 7 of $\cite{ar3}$ so that it guarantees that the required continuation maps are well-defined. The argument is similar to the parametrized version of Lemma \ref{lemma:pair-of-pants}, with the additional complexity that $f$ is $s$-dependent. Choose an interpolation $f_s$ from $f:V\rightarrow\mathbb{R}$ to 0, which is supported on a closed interval $I\subset\mathbb{R}$. Write again $u=(v,\rho)$ for a solution to the parametrized Floer equation $(du-X_{H_s}\otimes dt)^{0,1}=0$, where $X_{H_s}$ is the Hamiltonian vector field associated to $H_s$, which has the form $h\left(r-f_s(v)\right)$, where $h$ is linear for $r\gg0$. Note that
\begin{equation}\rho=\left(r-f_s(v)\right)\circ u\end{equation}
in our case. What is different from the argument in $\cite{ar3}$ is that we shall use an admissible almost complex structure $J\in\mathcal{J}(M)$, so in particular $dr\circ J=-\theta_V$ for $r\gg0$. To indicate the dependence of $h$ on $s$, we denote it as $h_s$ in the computations below. As in the proof of Lemma \ref{lemma:pair-of-pants}, we can project the Floer equation to the $\langle R\rangle$ and $\langle\partial_r\rangle$ directions to get
\begin{equation}\label{eq:CR-equation}\left\{\begin{array}{l}\theta_V\left(\partial_sv(s,t)\right)+\partial_t\rho=0 \\ \partial_s\left(\rho+f_s\circ v(s,t)\right)-\theta_V\left(\partial_tv(s,t)\right)+h_s'(\rho)=0\end{array}\right.\end{equation}
where we have used the fact that $f_s$ is independent of $t$. Since $J$ restricted to $\xi$ is tamed by $d\theta_V$, we have
\begin{equation}\partial_s\left(\theta_V\left(\partial_tv(s,t)\right)\right)-\partial_t\left(\theta_V\left(\partial_sv(s,t)\right)\right)\geq0,\end{equation}
Substitute the expressions of $\theta_V\left(\partial_sv(s,t)\right)$ and $\theta_V\left(\partial_tv(s,t)\right)$ in the above inequality, we get an inequality satisfied by $\Delta\rho$:
\begin{equation}\Delta\rho+\partial_s^2f_s+h_s''\partial_s\rho+\partial_sh_s'\geq0.\end{equation}
Since the term $h_s''\partial_s\rho$ above only involves first order derivative in $\rho$, to ensure that the maximum principle for elliptic operators applies, we only require $(f_s)_{s\in I}$ to satisfy
\begin{equation}\partial_sh_s'\leq-\partial_s^2f_s.\end{equation}
One can assume that the homotopy satisfies $\partial_s^2f_s\leq C$ for some large constant $C\gg0$, so that the above condition becomes $\partial_sh_s'\leq-C$. Denote by $h_-$ the value of $h_s$ for $s\ll0$, and by $h_+$ the value of $h_s$ for $s\gg0$. After integrating on $I$ one sees that it suffices to require $h_-'\gg h_+'$. The homotopy $(f_s)$ from 0 to $f$ can be built similarly, provided that $h_-'\ll h_+'$.

The above argument enables us to build continuation maps $\mathit{HF}^\ast(H_+)\rightarrow\mathit{HF}^\ast(H_-)$ when $h_-'\gg h_+'$ for $r\gg0$. Now the argument is identical to that of Theorem 8 of $\cite{ar3}$, namely one picks a sequence of Hamiltonians $(K_m)$ linear at infinity which alternates the sequence of non-linear Hamiltonians $(H_m)$ of the form $h_m\left(r-f_{s_m}(v)\right)$ at infinity, so that the continuation maps $\mathit{HF}^\ast(K_m)\rightarrow\mathit{HF}^\ast(H_m)$ and $\mathit{HF}^\ast(H_m)\rightarrow\mathit{HF}^\ast(K_{m+1})$ can be defined. Suppose that the slopes of $(H_m)$ and $(K_m)$ tend to infinity when $m\rightarrow\infty$, passing to direct limits one sees that ${\varinjlim}_m\mathit{HF}^\ast(H_m)\cong\mathit{SH}^\ast(M)$. This finishes the proof by realizing the obvious isomorphism ${\varinjlim}_m\mathit{HF}^\ast(H_m)\cong\mathit{SH}^\ast(M')$.

For the product structures, one can argue similarly as in Proposition \ref{proposition:TQFT}, with $H_s$ replaced by a Hamiltonian of the form $h_s=h\left(r-f_s(v)\right)$ when $r\gg0$, and Lemma \ref{lemma:ext-maximum} replaced by Proposition \ref{proposition:contact-type}.
\end{proof}
A byproduct of the proposition above is that the \textit{growth rate}
\begin{equation}\mathit{\Gamma}(M)\in\{-\infty\}\cup\mathbb{R}_+\cup\{\infty\}\end{equation}
of $\mathit{SH}^\ast(M)$ introduced in $\cite{ps2}$ is well-defined under the assumption that $M^\mathrm{in}$ is a stable filling whose completion $M$ is semi-positive. It follows from our definition that for every Lefschetz domain $M^\mathrm{in}$, $\mathit{\Gamma}(M)=0$.

\subsection{Blow-ups and flips}\label{section:symp-birational}
We now collect all the tools and results established in the last five subsections to prove Corollaries \ref{corollary:blow-up} and \ref{corollary:flip}.\bigskip

We start by recalling the concept of a symplectic flip, which is considered in $\cite{cw}$ for closed symplectic manifolds. The definition in $\cite{cw}$ applies straightforwardly to non-compact symplectic manifolds, in particular it makes sense to consider symplectic flips in our set up.
\begin{definition}[$\cite{cw}$, Definition 2.7]\label{definition:flip}
We say that two symplectic manifolds $W_-$ is obtained from $W_+$ are related by a simple symplectic flip if there exists a symplectic vector space $\widetilde{W}\cong\mathbb{C}^{n+1}$ with a Hamiltonian $S^1$-action whose moment map is $\mu_{\widetilde{W}}:\widetilde{W}\rightarrow\mathbb{R}$ such that $W_\pm=\mu_{\widetilde{W}}(\pm\varepsilon)/S^1$ are symplectic reductions at values $\pm\varepsilon$ for some small $\varepsilon>0$, and the following are satisfied:
\begin{itemize}
\item[(i)]denote by $\lambda_1,\cdots,\lambda_{n+1}$ the weights of the Hamiltonian $S^1$-action, then $\sum_{i=1}^{n+1}\lambda_i>0$;
\item[(ii)]$\lambda_i\neq0$ for every $i$;
\item[(iii)]at least two weights of $\{\lambda_i\}$ are positive, and at least two weights are negative.
\end{itemize}
Write $\widetilde{W}_\pm$ for the sum of the positive resp. negative weight spaces in $\widetilde{W}$, so that $\widetilde{W}=\widetilde{W}_-\oplus\widetilde{W}_+$. The semi-stable loci are
\begin{equation}
\widetilde{W}^{\mathrm{ss},-}=\left(\widetilde{W}_-\setminus\{0\}\right)\times\widetilde{W}_+,\widetilde{W}^{\mathrm{ss},+}=\widetilde{W}_+\times\left(\widetilde{W}_+\setminus\{0\}\right).
\end{equation}
In particular, $W_+$ is obtained from $W_-$ by replacing a weighted projective space of dimension $n_--1$ with a weighted projective space of dimension $n_+-1$, where $n_\pm$ are the number of positive resp. negative weights.

More generally, let $M_-$ and $M_+$ be non-empty symplectic manifolds. We say that $M_+$ is obtained from $M_-$ by a symplectic flip is there exist open covers
\begin{equation}
M_\pm=U_\pm\cup W_\pm
\end{equation}
such that the following hold:
\begin{itemize}
	\item[(i)] $U_+$ is diffeomorphic to $U_-$ and there is a family of symplectic forms $\omega_{U,t}\in\Omega^2(U_\pm)$, where $t\in[-\varepsilon,\varepsilon]$, together with symplectic embeddings $i_\pm:U_\pm\rightarrow M_\pm$, so that $i_\pm^\ast\omega_{M_\pm}=\omega_{U,\pm\varepsilon}$.
	\item[(ii)] $W_\pm$ are related by a simple flip.
	\item[(iii)] Under the canonical identification $H^2(M_-;\mathbb{R})\cong H^2(M_+;\mathbb{R})$ which maps $c_1(M_-)$ to $c_1(M_+)$, we have
	\begin{equation}
	\left[\omega_{M_-}\right]-\left[\omega_{M_+}\right]=2\varepsilon' c_1(M_-)
	\end{equation}
	for some $\varepsilon'>0$.
\end{itemize}
\end{definition}
In the above, $\widetilde{W}$ is the local model used to perform the variation of GIT construction, which results in the (generally non-isomorphic) symplectic manifolds $W_+$ and $W_-$. $U\subset M_\pm$ is the complement of $W_\pm$, where the symplectic structure does not change up to isotopy.

A particular simple case of interest is when all the weights of the Hamiltonian $S^1$-action on $\widetilde{W}$ are equal to $\pm1$. In this case the center of the flip is necessarily trivial, meaning that there is a symplectic $\mathbb{CP}^{n_2}$ embedded in $M_+$. Since the weights are $\pm1$, a tubular neighborhood of $\mathbb{CP}^{n_2}\subset M_+$ is symplectomorphic to $\mathcal{O}(-1)^{\oplus n_1}\rightarrow\mathbb{CP}^{n_2}$. Under the simple flip $M_+\dashrightarrow M_-$, the $\mathbb{CP}^{n_2}$ is replaced by a $\mathbb{CP}^{n_1-1}\subset M_-$ whose tubular neighborhood is identified with $\mathcal{O}(-1)^{\oplus n_2+1}\rightarrow\mathbb{CP}^{n_1-1}$. Note that here $n_2$ is the number of positive weights of our $S^1$-action.\bigskip

We now proceed to the proofs of Corollaries \ref{corollary:blow-up} and \ref{corollary:flip}. By allowing $n_1=1$ in $E_+$ we can treat the cases of blow-ups and reverse simple flips simultaneously. We first show that in both cases the manifold $M_+$ admits the structure of the completion of a stable symplectic filling, so in particular $\mathit{SH}^\ast(M_+)$ is well-defined.
\begin{lemma}\label{lemma:stab}
There exists a compact submanifold with boundary $M_+^\mathrm{in}\subset M_+$ whose boundary $V_+:=\partial M_+^\mathrm{in}$ carries a stable Hamiltonian structure $\left(\omega_{V_+},\theta_{V_+}\right)$, so that $M_+$ is the completion of $M_+^\mathrm{in}$, where $\omega_{V_+}:=\omega_{M_+}|V_+$ is the restriction of the symplectic form, and $\theta_{V_+}$ is a contact form on $V_+$.
\end{lemma}
\begin{proof}
Let $M_+=\mathit{Bl}_x(M_-)$, with $x\in M_-\setminus M_-^\mathrm{in}$. Outside of a compact subset of $M_+$, the symplectic form $\omega_{M_+}=\pi^\ast\omega_{M_-}$, where $\pi:M_+\dashrightarrow M_-$ is the blow-down map. Recall that $M_-$ is the completion of a stable filling of the contact boundary $(V_-,\xi_-)$, without loss of generality one can assume that $\omega_{M_+}=\pi^\ast\omega_{M_-}$ holds on $\pi^{-1}(V_-)$. Define $V_+:=\pi^{-1}(V_-)$ and $\theta_{V_+}:=\pi^\ast\theta_{V_-}$, then it is clear that $\theta_{V_+}$ is a contact form on $V_+$. Since $\left(\omega_{V_-},\theta_{V_-}\right)$ is a stable Hamiltonian structure, it is easy to see the same is true for $\left(\omega_{V_+},\theta_{V_+}\right)$. Such a geometric setup clearly makes $M_+$ the completion of $M_+^\mathrm{in}$.

In the case of a reverse simple flip, a similar argument holds since the flip $M_+\dashrightarrow M_-$ can be regarded as first blowing up the symplectic submanifold $\mathbb{CP}^{n_2}\subset M_+$, and then blowing down along the $\mathbb{CP}^{n_1-1}$ direction.
\end{proof}
We shall use the convenient notation $\left(M_\#,\omega_{M_\#}\right)$ for the completions of the boundary connected sums $M_-^\mathrm{in}\#_\partial E_+^\mathrm{in}$ and $U^\mathrm{in}\#_\partial E_+^\mathrm{in}$ respectively in the blow-up and reverse flip cases. Denote by $\left(V_\#,\xi_\#\right)$ the boundary of $\left(M_\#,\omega_{M_\#}\right)$.
\begin{lemma}
There is a contactomorphism $\eta:(V_+,\xi_+)\xrightarrow{\cong}\left(V_\#,\xi_\#\right)$.
\end{lemma}
\begin{proof}
By our proof of Lemma \ref{lemma:stab}, there is a contactomorphism between $(V_+,\xi_+)$ and $(V_-,\xi_-)$. In particular, $\left(\partial E_+^\mathrm{in},\xi_{BW}\right)$ and $\left(\partial E_-^\mathrm{in},\xi_{BW}\right)$ are contactomorphic.

In the blow-up case, the contact structure on $\left(V_\#,\xi_\#\right)$ comes from the contact connected sum $(V_-,\xi_-)\#\left(S^{2n-1},\xi_{std}\right)$, where $\xi_{std}$ is the standard contact structure on $S^{2n-1}$, which is contactomorphic to $(V_-,\xi_-)$.

In the reverse flip case, $(V_-,\xi_-)$ is the contact connected sum $(V,\xi)\#\left(\partial E_-^\mathrm{in},\xi_{BW}\right)$, while $\left(V_\#,\xi_\#\right)$ is the contact connected sum $(V,\xi)\#\left(\partial E_+^\mathrm{in},\xi_{BW}\right)$. Since $\left(\partial E_+^\mathrm{in},\xi_{BW}\right)$ and $\left(\partial E_-^\mathrm{in},\xi_{BW}\right)$ are contactomorphic, so are $(V_-,\xi_-)$ and $\left(V_\#,\xi_\#\right)$.
\end{proof}
\bigskip
\begin{proof}[Proof of Corollaries \ref{corollary:blow-up} and \ref{corollary:flip}]
By the above lemma, we can change the contact hypersurface from $(V_+,\xi_+)$ to $\left(V_\#,\xi_\#\right)$ in $\left(M_+,\omega_{M_+}\right)$ when deforming the symplectic form from $\omega_{M_+}$ to $\omega_{M_\#}$. Under this change of contact hypersurface, $\left(V_\#,\xi_\#\right)\subset\left(M_+,\omega_{M_+}\right)$ is still stably filled and the cylindrical end $[1,\infty)\times\partial M_+$ can be identified with $[1,\infty)\times\left(V_\#,\xi_\#\right)$. There is a symplectommorphism $\phi:\left(M_+,\omega_{M_+}\right)\rightarrow\left(M_+,\omega_{M_+}'\right)$ which realizes the above process, and $\omega_{M_+}$ is cohomologous to $\omega_{M_+}'$.  By Proposition \ref{proposition:contact-type}, we have $\mathit{SH}^\ast\left(M_+,\omega_{M_+}\right)\cong\mathit{SH}^\ast\left(M_+,\omega_{M_+}'\right)$ as $\mathbb{K}$-algebras.

On the other hand, the symplectic manifold $\left(M_\#,\omega_{M_\#}\right)$ can be regarded as $\left(M_+,\omega_{M_+}'\right)$ with the symplectic structure $\omega_{M_+}'$ deformed in $M_+^\mathrm{in}$ to $\omega_{M_\#}$. Since the symplectic form $\omega_{M_+}$ is cohomologous to $\omega_{M_\#}$ by our choice, and $\omega_{M_+}$ is cohomologous to $\omega_{M_+}'$, the difference $\omega_{M_+}'-\omega_{M_\#}$ is a compactly supported exact form. Moser's lemma then produces a symplectomorphism
\begin{equation}\varphi:\left(M_\#,\omega_{M_\#}\right)\rightarrow\left(M_+,\omega_{M_+}'\right),\end{equation}
and therefore an identification of the Floer complexes $\mathit{CF}^\ast(H)$ and $\mathit{CF}^\ast(\varphi^\ast H)$. Since $\varphi$ is an identity outside a compact subset of $M_+$, it follows that after taking direct limits we get an isomorphism
\begin{equation}\mathit{SH}^\ast\left(M_\#,\omega_{M_\#}\right)\cong\mathit{SH}^\ast\left(M_+,\omega_{M_+}'\right).\end{equation}
which preserves the $\mathbb{K}$-algebra structures.

By Theorem \ref{theorem:flip} proved in Section \ref{section:attachment},
\begin{equation}\mathit{SH}^\ast(M_\#)\cong\mathit{SH}^\ast(Y)\oplus\mathit{SH}^\ast(E_+),\end{equation}
where by $Y$ we mean $M_-$ in the case of blow-ups, and $U$ in the case of reverse flips. By the computations in Section \ref{section:computation}, $\mathit{SH}^\ast(E_+)\cong\mathbb{K}^{n_2+1-n_1}$ is semisimple. Combined with our assumption that $\mathit{SH}^\ast(Y)=0$ or $\mathit{SH}^\ast(Y)$ is semisimple, the same is true for $\mathit{SH}^\ast(M_\#)$. Since we have proved
\begin{equation}\mathit{SH}^\ast\left(M_+,\omega_{M_+}\right)\cong\mathit{SH}^\ast\left(M_+,\omega_{M_+}'\right)\cong\mathit{SH}^\ast\left(M_\#,\omega_{M_\#}\right)\end{equation}
as $\mathbb{K}$-algebras, $\mathit{SH}^\ast(M_+)$ is semisimple.
\end{proof}

\subsection{Remark on the general case}
Although Corollaries \ref{corollary:blow-up} and \ref{corollary:flip} are only stated for blow-ups and flips with trivial centers, it is not hard to see the method presented here works when the center $Z$ of the blow-up or flip is a smooth toric Fano variety (with the algebra $\mathit{QH}^\ast(Z)/\ker c_1(Z)$ being semisimple).

As a simple example which can be extracted essentially from known computations, consider the case when $M_-$ is the total space of the negative vector bundle
\begin{equation}\mathcal{O}(-1)^{\oplus(n+1)/2}\rightarrow\mathbb{CP}^{(n-1)/2}.\end{equation}
Since $c_1(M_-)=0$, it follows from $\cite{yg}$ or $\cite{ar}$ that $\mathit{SH}^\ast(M_-)=0$. By blowing up along the zero section $\mathbb{CP}^{(n-1)/2}\subset M_-$ we obtain another symplectic manifold $M_+$, which can be identified with the total space of the line bundle $\mathcal{O}(-1,-1)\rightarrow\mathbb{CP}^{(n-1)/2}\times\mathbb{CP}^{(n-1)/2}$. Note that $\mathit{QH}^\ast(\mathbb{CP}^{(n-1)/2})$, the quantum cohomology of the blowing-up center, is semisimple $\cite{km}$. Using the method of $\cite{ar1}$, it can be shown that $\mathit{SH}^\ast(M_+)\cong\mathbb{K}^{(n^2+2n-7)/4}$ is semisimple (the simplest case $n=3$ is covered by Corollary 4.16 of $\cite{ar1}$, and the general case follows similarly).
\bigskip

A possible approach to prove Conjecture \ref{conjecture:mmp} for more general blow-ups and flips (which are not necessarily performed on the cylindrical end) would be to try to generalize the Mayer-Vietoris sequence for symplectic cohomologies established in $\cite{co}$ to the non-exact case. More precisely, let $M^\mathrm{in}$ be a Liouville domain and $U_1^\mathrm{in},U_2^\mathrm{in}\subset M^\mathrm{in}$ Liouville cobordisms such that $M^\mathrm{in}=U_1^\mathrm{in}\cup U_2^\mathrm{in}$. Denote by $A^\mathrm{in}$ the Liouville cobordism $U_1^\mathrm{in}\cap U_2^\mathrm{in}$, it follows from $\cite{co}$ that there is an exact triangle
\begin{equation}
\xymatrix{
  \mathit{SH}^\ast(M) \ar[rr]^{ }
                &  &    \mathit{SH}^\ast(U_1)\oplus\mathit{SH}^\ast(U_2) \ar[dl]^{ }    \\
                & \mathit{SH}^\ast(A) \ar[ul]^{[-1]}                }
\end{equation}
Assuming the above long exact sequence holds for stable symplectic fillings, we illustrate here briefly how to adapt the it to study the symplectic cohomology in the special case when $M_-\dashrightarrow M_+$ is a blow-up or reverse flip with trivial center $Z=\{\mathit{pt}\}$. Take $U_1^\mathrm{in}$ above to be the Lefschetz domain $E_+^\mathrm{in}$, and $U_2^\mathrm{in}$ to be the stable symplectic cobordism $M_+^\mathrm{in}\setminus\mathcal{O}(-1)_{\leq1+\varepsilon}^{\oplus n_1}$, where $\mathcal{O}(-1)_{\leq1+\varepsilon}^{\oplus n_1}$ is a slight enlargement of the unit ball bundle $E_+^\mathrm{in}$. It follows that $A^\mathrm{in}=[0,\varepsilon]\times\partial E_+^\mathrm{in}$ is a trivial cobordism, therefore one would have $\mathit{SH}^\ast(A)=0$ in view of Albers-Kang $\cite{ak}$.

In order to compute $\mathit{SH}^\ast(U_2)$, it would be convenient to change to an alternative symplectic filling of the boundary component $\partial E_+^\mathrm{in}\subsetneq\partial U_2^\mathrm{in}$ of the symplectic cobordism $U_2^\mathrm{in}$. For this we have an obvious choice, which is $E_-^\mathrm{in}$. As mentioned in the introduction, one should then expect an isomorphism
\begin{equation}\mathit{SH}^\ast(M_-)\cong\mathit{SH}^\ast(U_2)\end{equation}
of symplectic cohomology groups. Note that although in general the symplectic cohomology for symplectic cobordisms depends on the choice of a filling, certain invariance results can be proved when the contact boundary is dynamically convex, see Section 9.5 of $\cite{co}$.

The above procedure should be enough to determine $\mathit{SH}^\ast(M_+)$ additively. Recovering its ring structure requires more detailed analysis. Note that if the center $Z\subset M_-$ is not toric, then the computation of $\mathit{SH}^\ast(U_1)$ would be substantially more difficult.

\section{Applications}
We collect in this section some interesting consequences of Theorems \ref{theorem:main} and \ref{theorem:flip}. In particular, Theorem \ref{theorem:generation} will be proved in Section \ref{section:generation}.

\subsection{Lefschetz manifolds}\label{section:Lefschetz}
The Lefschetz condition (Definition \ref{definition:Lefschetz}) imposed on completions of stable symplectic fillings is very strong. The condition (i) has been discussed briefly in Section \ref{section:con-tor}, and (ii) is particularly hard to check without explicit computations of $\mathit{SH}^\ast(M)$. However, Theorem \ref{theorem:flip} and its corollaries still enable us to get some new examples in terms known ones.

The following is a simple consequence of Theorem \ref{theorem:flip}.
\begin{proposition}\label{proposition:Lefschetz-sum}
Suppose $M^\mathrm{in}$ and $(M')^\mathrm{in}$ are two Lefschetz domains at the same level $j$, then their boundary connected sum $M^\mathrm{in}\#_\partial(M')^\mathrm{in}$ is still a Lefschetz domain at the same level.
\end{proposition}
\begin{proof}
It's easy to see from the Mayer-Vietoris sequence that additively we have
\begin{equation}\label{eq:adqh}\mathit{H}^\ast(M\#_\partial M';\mathbb{K})\cong\frac{\mathit{H}^\ast(M;\mathbb{K})\oplus\mathit{H}^\ast(M';\mathbb{K})}{H^\ast(\mathit{pt};\mathbb{K})}.\end{equation}

For the product structure, suppose $u:S^2\rightarrow M\#_\partial M'$ is any $J$-holomorphic sphere contributing to the quantum product on $\mathit{H}^\ast(M\#_\partial M';\mathbb{K})$. It is clear that its image $\mathit{im}(u)$ must intersect the handle $H_1(\delta)$ non-trivially. More precisely, we actually have $\mathit{im}(u)\cap\partial_-H_1(\delta)\neq\emptyset$. Since any such smooth map $u$ can be regarded as a Floer trajectory which is asymptotic to orbits of a Hamiltonian function with very small slope on the cylindrical end of $M\#_\partial M'$, we can apply Lemma 7.2 of $\cite{as}$ to the restriction of $u$ to $u^{-1}\left(\mathit{im}(u)\cap M^\mathrm{in}\right)$ to exclude its existence. This is possible because the symplectic structure on $M\#_\partial M'$ restricted to a small neighborhood of $H_1(\delta)$ is exact, and the boundary components $\partial_-H_1(\delta)$ are concave with respect to the local Liouville vector fields $Z_\pm$. Note that using an almost complex structure $J\in\mathcal{J}(M)$ which is not necessarily compatible with $\omega_M$ does not affect the proof there. From this we see that the quantum product structures on both summands in (\ref{eq:adqh}) are unchanged under the boundary connected sum, and for $a_1\in H^\ast(M;\mathbb{K})$ and $a_2\in\mathit{H}^\ast(M';\mathbb{K})$, we have $a_1\star a_2=0$. This completely determines the ring structure of $\mathit{QH}^\ast(M\#_\partial M')$.

With this realization, $c_1(M\#_\partial M')\in\mathit{QH}^\ast(M\#_\partial M')$ is given by
\begin{equation}\left(c_1(M),c_1(M')\right)\in\mathit{H}^\ast(M;\mathbb{K})\oplus\mathit{H}^\ast(M';\mathbb{K}).\end{equation}
It then follows from Theorem \ref{theorem:flip} that $\mathit{SH}^\ast(M\#_\partial M')$ is the localization of $\mathit{QH}^\ast(M\#_\partial M')$ at $c_1(M\#_\partial M')$.

On the other hand, it is obvious that $H^{2j}(V\# V';\mathbb{Q})$ vanishes because of the vanishing of $H^{2j}(V;\mathbb{Q})$ and $H^{2j}(V';\mathbb{Q})$. Since $c_1(M)^{\star j}\in H^{2j}(M;\mathbb{K})$ and $c_1(M')^{\star j}\in H^{2j}(M';\mathbb{K})$, it follows that $c_1(M\#_\partial M')^{\star j}\in H^{2j}(M\#_\partial M';\mathbb{K})$.
\end{proof}
By our proof of Corollaries \ref{corollary:blow-up} and \ref{corollary:flip} in Section \ref{section:symp-birational}, we can translate the above result in terms of blow-ups of points and reverse simple flips.
\begin{corollary}\label{corollary:Lefschetz-blp}
If $M$ is the completion of a Lefschetz domain, then so is its blow-up $\mathit{Bl}_x(M)$ at any point $x\in M\setminus M^\mathrm{in}$.
\end{corollary}
\begin{proof}
As we have already observed in the introduction, for the total space of $\mathcal{O}(-1)_{\leq1}\rightarrow\mathbb{CP}^{n-1}$, we can make it into a Lefschetz domain by taking $j$ to be any integer with $1\leq j\leq n-1$. Replacing $\mathit{QH}^\ast\left(\mathit{Bl}_x(M)\right)$ by the Hamiltonian Floer cohomology $\mathit{HF}^\ast(H_\varepsilon)$ for a Hamiltonian function $H_\varepsilon:\mathit{Bl}_x(M)\rightarrow\mathbb{R}$ with sufficiently small slope $\varepsilon>0$ at infinity. Note that we can choose $\varepsilon$ to be so small that under a symplectomorphism $\phi$ of contact type at infinity, the pullback $\phi^\ast H_\varepsilon$ still defines the quantum cohomology, namely
\begin{equation}\mathit{HF}^\ast(\phi^\ast H_\varepsilon)\cong\mathit{QH}^\ast\left(\mathit{Bl}_x(M)\right).\end{equation}
In fact, $\phi^\ast H_\varepsilon$ has the form $\varepsilon\left(r-f(v)\right)$ for $r\gg0$. Since $f(v)$ is bounded, we may assume $f(v)<\frac{r}{2}$, and replace $\varepsilon$ by $\frac{\varepsilon}{2}$ whenever necessary. Finally, apply Moser's lemma, we deduce that \begin{equation}\mathit{QH}^\ast\left(\mathit{Bl}_x(M)\right)\cong\mathit{QH}^\ast\left(M\#_\partial\mathcal{O}_{\mathbb{P}^{n-1}}(-1)\right)\end{equation} as algebras over $\mathbb{K}$. The corollary then follows from Proposition \ref{proposition:Lefschetz-sum}.
\end{proof}
This in particular shows that the manifold $\mathit{Bl}_S(\mathbb{C}^n)$, obtained by blowing up $\mathbb{C}^n$ at a finite set of distinct points $S\subset\mathbb{C}^n$, is the completion of a Lefschetz domain. One can also take any split vector bundle $\mathcal{O}(-m)^{\oplus n_1}\rightarrow\mathbb{CP}^{n_2}$, with $mn_1\leq n_2+1$ and do the blow-ups away from the zero section.

Analogously, we have the following:
\begin{corollary}
Let $(U^\mathrm{in},\omega_U)$ be a Lefschetz domain at level $j=\lceil n/2\rceil$, and denote by $U\#_\partial E_-$ the completion of the boundary connected sum $U^\mathrm{in}\#_\partial E_-^\mathrm{in}$. Under the reverse simple flip $U\#_\partial E_-\dashrightarrow M$ along $\mathbb{CP}^{n_1-1}\subset E_-^\mathrm{in}$, $M^\mathrm{in}$ is also a Lefschetz domain.
\end{corollary}
\begin{proof}
This is similar to the proof of Corollary \ref{corollary:Lefschetz-blp}, except that in this case condition (i) in Definition \ref{definition:Lefschetz} may not be satisfied for $E_+$ with any choice of $j$ with $1\leq j\leq n-1$. However, it follows from Corollary \ref{corollary:Lefschetz} that $E_+^\mathrm{in}$ a Lefschetz domain at level $\lceil n/2\rceil$.
\end{proof}
Note that the Lefschetz domain $U^\mathrm{in}$ in the above can be taken to be the total space of $\mathcal{O}(-m)_{\leq1}^{\oplus n_1}\rightarrow\mathbb{CP}^{n_2}$, where $mn_1\leq n_2+1$, by the computations in Section \ref{section:computation}. We can actually go beyond this case a little bit. For instance, take the blow-up of $E_-$ at any point $x$ away from the zero section $\mathbb{CP}^{n_1-1}$, and then perform a reverse simple flip $\mathit{Bl}_x(E_-)\dashrightarrow M$ along $\mathbb{CP}^{n_1-1}\subset E_-$. The resulting manifold $M$ is also Lefschetz.\bigskip

Let $U_1^\mathrm{in}$ and $U_2^\mathrm{in}$ be Lefschetz domains. An embedding $\iota:U_1^\mathrm{in}\hookrightarrow U_2^\mathrm{in}$ is a \textit{Lefschetz embedding} if it is a symplectomorphism of contact type onto its image. Following $\cite{ps2}$, one can introduce the following notion.
\begin{definition}
Let $M$ be a non-compact symplectic manifold without boundary. $M$ is a Lefschetz manifold if there exists a sequence of Lefschetz domains $\left\{U_k^\mathrm{in}\right\}$ and embeddings $\iota_k:U_k^\mathrm{in}\hookrightarrow M$ which are symplectomorphisms of contact type onto its image, such that the images $\left\{\iota_k\left(U_k^\mathrm{in}\right)\right\}$ exhaust $M$.
\end{definition}
Clearly, every completion of a Lefschetz domain is a Lefschetz manifold. On the other hand, there are certainly Lefschetz manifolds which do not come from completion. For example, one can take $M=\mathit{Bl}_{\mathbb{Z}^{2n}}(\mathbb{C}^n)$ to be the blow-up of $\mathbb{C}^n$ at every point in $\mathbb{Z}^{2n}\subset\mathbb{C}^n$.

With some effort, the Viterbo functoriality (\cite{cvi}) can be generalized to the current case. Namely for every Lefschetz embedding $\iota:U_1^\mathrm{in}\hookrightarrow U_2^\mathrm{in}$ between Lefschetz domains, there is a pullback map
\begin{equation}\mathcal{T}_\iota:\mathit{SH}^\ast(U_2)\rightarrow\mathit{SH}^\ast(U_1).\end{equation}
Using this one can associate to every Lefschetz manifold $M$ its symplectic cohomology, by taking the inverse limit
\begin{equation}\mathit{SH}^\ast(M):={\varprojlim}_k\mathit{SH}^\ast(U_k).\end{equation}
As an example, it follows from Proposition \ref{proposition:sh-blp} below that
\begin{equation}\mathit{SH}^\ast\left(\mathit{Bl}_{\mathbb{Z}^{2n}}(\mathbb{C}^n)\right)\cong\prod_{\alpha\in\mathbb{Z}^{2n}}\mathbb{K}[x_\alpha]/\left(x_\alpha^n+q^n\right),\end{equation}
which shows that the symplectic cohomology of a Lefschetz manifold need not be finite dimensional.

\subsection{Split-generation}\label{section:generation}
Let $M$ be a monotone symplectic manifold obtained by completing the stable filling of the contact manifold $(V,\xi)$. Using the algebraic tools introduced in Section \ref{section:structure}, one can prove the following:
\begin{theorem}[Ritter-Smith $\cite{rs}$]\label{theorem:split}
Let $\mathcal{B}\subset\mathcal{F}_\lambda(M)$ be an full $A_\infty$ subcategory. If the open-closed map
\begin{equation}\mathit{OC}:\mathit{HH}_\ast(\mathcal{B},\mathcal{B})\rightarrow\mathit{QH}^\ast(M)_\lambda\end{equation}
hits an invertible element, then $\mathcal{B}$ split-generates $D^\pi\mathcal{F}_\lambda(M)$. Similarly, a full $A_\infty$ subcategory $\mathcal{B}\subset\mathcal{W}_\lambda(M)$ split-generates $D^\pi\mathcal{W}_\lambda(M)$ if
\begin{equation}\mathit{OC}:\mathit{HH}_\ast(\mathcal{B},\mathcal{B})\rightarrow\mathit{SH}^\ast(M)_\lambda\end{equation}
hits an invertible element of $\mathit{SH}^\ast(M)_\lambda$.
\end{theorem}
This is proved in $\cite{rs}$ when $M^\mathrm{in}$ is a strong filling, and extends easily to its form stated above once the relevant maximum principles have been established. See Sections \ref{section:ham} and \ref{section:fuk}. The key ingredient of the proof is a breaking analysis of the moduli space of marked annuli, namely the one involved in the Cardy relation (Proposition \ref{proposition:cardy}).\bigskip

For the purpose of proving Theorem \ref{theorem:generation}, we need to extend the Fukaya categories $\mathcal{W}(M)$ and $\mathcal{F}(M)$ by allowing as their objects the Lagrangian branes $(L,\xi_L)$, where $L\subset M$ is an admissible or closed admissible Lagrangian submanifold, and $\xi_L\in\mathit{Hom}\left(\pi_1(L),U_\mathbb{K}\right)$ determines a unitary rank 1 local system with the fiber at the point $l\in L$ given by $\underline{\mathbb{K}}_l^L\cong\mathbb{K}$. Under this extension, we treat the underlying Lagrangian submanifold $L$ as a Lagrangian brane equipped with a trivial local system.

One can generalize the definition (\ref{eq:obs}) of $\mathfrak{m}_0$ to the admissible Lagrangian brane $(L,\xi_L)$ as follows:
\begin{equation}\mathfrak{m}_0(L,\xi_L)=\sum_{\beta\in\pi_2(M,L)}q^{\omega_M(\beta)}\xi_L(\partial\beta)\mathit{ev}_\ast\left[\mathcal{M}_1(L,\beta)\right]=m_0(L,\xi_L)[L],\end{equation}
where we have identified $\xi_L$ with a cohomology class of $H^1(L;U_\mathbb{K})$, so $\xi_L(\partial\beta)\in U_\mathbb{K}$.

Fix two Lagrangian branes $(L_0,\xi_{L_0})$ and $(L_1,\xi_{L_1})$ with $m_0(L_0,\xi_{L_0})=m_0(L_1,\xi_{L_1})$, the morphism spaces between them are given by the Floer complex
\begin{equation}\mathit{CF}^\ast\big((L_0,\xi_{L_0}),(L_1,\xi_{L_1});wH\big):=\bigoplus_{x\in\mathit{CF}^\ast(L_0,L_1;wH)}\mathit{Hom}_\mathbb{K}\left(\underline{\mathbb{K}}_{x(0)}^{L_0},\underline{\mathbb{K}}_{x(1)}^{L_1}\right),\end{equation}
where the direct sum ranges over all Hamiltonian chords of $X_H$ with weight $w\in\mathbb{Z}$. The term $q^{\omega_M([u])}x$ in the original Floer differential $\partial y$ is now deformed to be
\begin{equation}q^{\omega_M([u])}h_{\ell_1}^{L_1}\circ o_y\circ h_{\ell_0}^{L_0}\in\mathit{Hom}_\mathbb{K}\left(\underline{\mathbb{K}}_{x(0)}^{L_0},\underline{\mathbb{K}}_{x(1)}^{L_1}\right)\end{equation}
in the Floer differential of $o_y\in\mathit{Hom}_\mathbb{K}\left(\underline{\mathbb{K}}_{y(0)}^{L_0},\underline{\mathbb{K}}_{y(1)}^{L_1}\right)$, where $\ell_0:[0,1]\rightarrow L_0$, $\ell_1:[0,1]\rightarrow L_1$ are paths swept out by $u(\partial\mathbb{D})$ along the oriented arcs connecting the boundary punctures -1 to 1 , and 1 to -1 respectively. $h_{\ell_i}^{L_i}:\underline{\mathbb{K}}_{\ell_i(0)}^{L_i}\rightarrow\underline{\mathbb{K}}_{\ell_i(1)}^{L_i}$ are the corresponding parallel transport maps. With this definition, it is easy to see the original Fukaya categories embed fully faithfully into the corresponding extensions. For convenience we shall denote the extended Fukaya categories still by their original notations $\mathcal{W}(M)$ and $\mathcal{F}(M)$.

The definitions of open-closed string maps can also be generalized to the current case so that the generation criterion stated above still holds for the extended version of Fukaya categories. In particular, the zeroth order open-closed maps
\begin{equation}\mathit{OC}^0:\mathit{HF}^\ast\big((L,\xi_L),(L,\xi_L)\big)\rightarrow\mathit{QH}^{\ast+n}(M)\end{equation}
and
\begin{equation}\mathit{OC}^0:\mathit{HW}^\ast\big((L,\xi_L),(L,\xi_L)\big)\rightarrow\mathit{SH}^{\ast+n}(M)\end{equation}
are $\mathbb{K}$-algebra homomorphisms.\bigskip

\begin{proof}[Proof of Theorem \ref{theorem:generation}]
From now on we restrict ourselves to the special case when $M^\mathrm{in}$ is a monotone Lefschetz domain. With all the preliminaries at hand, Theorem \ref{theorem:generation} is a simple corollary of the generation criterion. Notice that the argument presented in Section \ref{section:pf-main} of Theorem \ref{theorem:main} still holds after replacing the collection of monotone Lagrangian submanifolds $(L_1,\cdot\cdot\cdot,L_r)$ with a collection of Lagrangian branes $\big((L_1,\xi_{L_1}),\cdot\cdot\cdot,(L_r,\xi_{L_r})\big)$. The same method as in Proposition \ref{proposition:OC} shows that $\mathit{OC}^0\left([pt]\right)\neq0$ in $\mathit{QH}^0(M)$, for $pt\in L$. By assumption, $(L_i,\xi_{L_i})$ and $(L_j,\xi_{L_j})$ lie in different eigensummands of $\mathcal{F}(M)$ whenever $L_i\cap L_j\neq\emptyset$, so by fixing a particular eigensummand $\mathcal{F}_\lambda(M)$ with $\lambda\neq0$, we only need to deal with disjoinable Lagrangians $L_i,L_j$ with $\mathit{HF}^\ast\big((L_i,\xi_{L_i}),(L_j,\xi_{L_j})\big)=0$. This reduces the problem to the original setting of Theorem \ref{theorem:main}.

Denote by $\mathcal{B}$ the full $A_\infty$ subcategory of $\mathcal{F}(M)$ formed by the Lagrangian branes $(L_1,\xi_{L_1}),\cdot\cdot\cdot,(L_r,\xi_{L_r})$ with $m_0\left(L_i,\xi_{L_i}\right)\neq0$. By our assumption, the even degree part of the non-zero eigensummand $\mathit{QH}^0(M)_{\neq0}$ is commutative and semisimple, therefore isomorphic to $\bigoplus_{i\in I}\mathbb{K}v_i$, with $I$ a finite index set and $(v_i)_{i\in I}\subset\mathit{QH}^\ast(M)$ a collection of idempotents. Since $r=\dim_\mathbb{K}\mathit{SH}^0(M)$ and $M$ is Lefschetz, the image of $\mathit{HH}_\ast(\mathcal{B},\mathcal{B})$ under the open closed map $\mathit{OC}$ contains an element of the form
\begin{equation}\left(a_1,\cdot\cdot\cdot,a_{|I|}\right)\in\mathit{QH}^0(M)_{\neq0}\cong\bigoplus_{i\in I}\mathbb{K}v_i,\end{equation}
where $a_i\neq0$ for every $i$. This defines an invertible element when restricted to every eigensummand $\mathit{QH}^\ast(M)_\lambda$ with $\lambda\neq0$. Denote by $\mathcal{B}_\lambda\subset\mathcal{B}$ the full $A_\infty$ subcategory which belongs to the $A_\infty$ category $\mathcal{F}_\lambda(M)$, it follows from Theorem \ref{theorem:split} that $\mathcal{B}_\lambda$ split-generates $D^\pi\mathcal{F}_\lambda(M)$.

By the acceleration diagram (\ref{eq:accel}), we further have $\mathit{OC}\left(\mathit{HH}_\ast(\mathcal{B}_\lambda,\mathcal{B}_\lambda)\right)$ hits an invertible element of $\mathit{SH}^\ast(M)_\lambda$, which proves the split-generation of $D^\pi\mathcal{W}_\lambda(M)$ by $\mathcal{B}_\lambda$.
\end{proof}

\subsection{Lagrangian tori}\label{section:example}
Theorem \ref{theorem:generation} proved in the last subsection is useful in finding explicit generators of Fukaya categories for monotone Lefschetz domains. As an application, we discuss three examples here, and show that certain eigensummands of their derived wrapped Fukaya categories are split-generated by Lagrangian tori (equipped with local systems).

\paragraph{Negative vector bundles.} Let $M$ be the total space of the negative vector bundle $\mathcal{O}(-m)^{\oplus n_1}\rightarrow\mathbb{CP}^{n_2}$, where $m\geq1$, $mn_1\leq n_2$ and $n_1+n_2=n$. Since $M$ is toric, we may apply standard toric techniques to find generators of the non-zero eigensummands of the Fukaya categories $\mathcal{F}(M)$ and $\mathcal{W}(M)$. Since the situation here is very similar to that of Section 12 of $\cite{rs}$, our discussions below will be quite sketchy.

We first remark that the monotone symplectic form $\omega_M$ is taken so that its restriction to the base is $mn_1\cdot\omega_\mathit{FS}$, where $\omega_\mathit{FS}$ is the Fubini-Study metric on $\mathbb{CP}^{n_2}$. From this we deduce the monotonicity constant for $M$:
\begin{equation}\lambda_M=\frac{1+n_2-mn_1}{mn_1}.\end{equation}

Let $\mu_M:M\rightarrow\Delta_M\subset\mathbb{R}^n$ be the moment map with respect to the standard $T^n$-action on $M$. The moment polytope $\Delta_M$ is given by
\begin{equation}\Delta_M=\left\{y\in\mathbb{R}^n\left|y_1\geq0,\cdot\cdot\cdot,y_n\geq0,\sum_{j=1}^{n_2}y_j-my_{n_2+1}\leq m,\cdot\cdot\cdot,\sum_{j=1}^{n_2}y_j-my_n\leq m\right.\right\}.\end{equation}
From this one can write down the superpotential $W:(\mathbb{K}^\ast)^n\rightarrow\mathbb{K}$,
\begin{equation}\label{eq:potential}W(z)=\sum_{j=1}^nz_j+q^{n_2+1-mn_1}z_1^{-1}\cdot\cdot\cdot z_{n_2}^{-1}z_{n_2+1}^m\cdot\cdot\cdot z_n^m.\end{equation}
One can check that for every solution $x=q(-m)^{1/\lambda_M}$ of the equation
\begin{equation}\label{eq:ambi}x^{n_2+1-mn_1}=q^{n_2+1-mn_1}(-m)^{mn_1},\end{equation}
there is a critical point
\begin{equation}z_c=(x,\cdot\cdot\cdot,x,-mx,\cdot\cdot\cdot,-mx),\end{equation}
whose critical value is
\begin{equation}W(z_c)=(n_2-mn+1)x.\end{equation}
A consequence of these computations is the following closed string homological mirror symmetry statement, which generalizes Corollary 12.11 of $\cite{rs}$.
\begin{proposition}
For $M$ the total space of $\mathcal{O}(-m)^{\oplus n_1}\rightarrow\mathbb{CP}^{n_2}$, with $mn_1\leq n_2$, we have the ring isomorphism
\begin{equation}\mathit{SH}^\ast(M)\cong\mathit{Jac}(W).\end{equation}
\end{proposition}
\begin{proof}
Recall that the Jacobi ring $\mathit{Jac}(W)$ is define by
\begin{equation}\mathit{Jac}(W)\cong\mathbb{K}\left[z_1^\pm,\cdot\cdot\cdot,z_n^\pm\right]/\left(\partial_{z_1}W,\cdot\cdot\cdot,\partial_{z_n}W\right).\end{equation}
This can be explicitly computed out using the expression (\ref{eq:potential}), and the result follows by comparing $\mathit{Jac}(W)$ with our computations of $\mathit{SH}^\ast(M)$ in Section \ref{section:computation}.
\end{proof}
We now turn our attention to Lagrangian submanifolds of $M$. Analogous to the closed toric Fano case, we show that there is a non-displaceable Lagrangian torus in $M$, which is a fiber of $\mu_M$.
\begin{proposition}\label{proposition:toric}
There is a monotone Lagrangian torus $L\subset M$, which is a $T^{n_1}$-bundle over the Clifford torus $T_\mathrm{Cl}^{n_2}\subset\mathbb{CP}^{n_2}$ such that $\mathit{HF}^\ast(L,L)\neq0$ and $pt\in C_\ast(L)$ defines a Floer cocycle. Moreover, it can be equipped with $n_2+1-mn_1$ different local systems $\xi_L^1,\cdot\cdot\cdot,\xi_L^{n_2+1-mn_1}\in\mathit{Hom}\left(\pi_1(L),U_\mathbb{K}\right)$, such that
\begin{equation}\mathit{HF}^\ast\big((L,\xi_L^i),(L,\xi_L^i)\big)\neq0,i=1,\cdot\cdot\cdot,n_2+1-mn_1.\end{equation}
\end{proposition}
\begin{proof}
Making the change of variable $y=q^{mn_1\lambda_M}z_{n_2+1}^m\cdot\cdot\cdot z_n^m$ in the expression of the superpotential (\ref{eq:potential}), we get
\begin{equation}W(z)=W_{\mathbb{P}}(y)+z_{n_2+1}+\cdot\cdot\cdot+z_n,\end{equation}
where $W_{\mathbb{P}}(y)$ is the superpotential for the Landau-Ginzburg model of $\mathbb{CP}^{n_2}$, which has the form
\begin{equation}W_{\mathbb{P}}(y)=z_1+\cdot\cdot\cdot+z_{n_2}+\frac{y}{z_1\cdot\cdot\cdot z_{n_2}}.\end{equation}
Recall that $W_{\mathbb{P}}(y)$ has $n_2+1$ critical points with non-zero critical values. Given any critical point $z(y)$ of $W_{\mathbb{P}}(y)$, a detailed computation shows that there is a critical point $z_c=\left(z(y),z_{n_2+1},\cdot\cdot\cdot,z_n\right)$ of $W$. Under the valuation map $\mathrm{val}_q:(\mathbb{K}^\ast)^n\rightarrow\mathbb{R}^n$, the critical point $z_c$ is mapped to a point $\mathrm{val}_q(z_c)\in\Delta_M$, and the fiber of $\mu_M$ over this point is exactly our monotone Lagrangian torus $L$.

Geometrically, consider the Hamiltonian $T^{n_1}$-action given by rotations along the fibers of the vector bundle $\mathcal{O}(-m)^{\oplus n_1}\rightarrow\mathbb{CP}^{n_2}$, which gives rise to a moment map $\nu_M:M\rightarrow\mathbb{R}^{n_1}$. The monotone Lagrangian torus $L\subset M$ is then given by the image of the Clifford torus $T_\mathrm{Cl}^{n_2}\subset\mathbb{CP}^{n_2}$ under the monotone Lagrangian correspondence
\begin{equation}\nu_M^{-1}(\lambda_L)\subset\left(\mathbb{CP}^{n_2},-\omega_\mathit{FS}\right)\times(M,\omega_M)\end{equation}
for some $\lambda_L\in\mathbb{R}^{n_1}$ corresponding to the last $n_1$ coordinates of the critical point $z_c$. Note that $\nu_M^{-1}(\lambda_L)$ is topologically a $T^{n_1}$-bundle over $\mathbb{CP}^{n_2}$.

It follows from the general result of Cho $\cite{cc}$ on the Floer cohomologies of toric fibers that $\mathit{HF}^\ast(L,L)\neq0$ and $pt\in C_\ast(L)$ defines a cocycle in $\mathit{HF}^\ast(L,L)$. The choices of local systems correspond to the ambiguity in choosing a solution of (\ref{eq:ambi}). In fact, $m_0(L,\xi_L^i)$ equals the critical value $W(z_c)$ of the corresponding critical point $z_c$ of $W$.
\end{proof}
Since $\dim_\mathbb{K}\mathit{SH}^\ast(M)=n_2+1-mn_1$ by the computations in Section \ref{section:computation}, as a Corollary of Theorem \ref{theorem:generation} we have:
\begin{theorem}\label{theorem:gen-nv}
Let $M$ be the total space of $\mathcal{O}(-m)^{\oplus n_1}\rightarrow\mathbb{CP}^{n_2}$, where $mn_1\leq n_2$. The non-zero eigensummands $\bigsqcup_{\lambda\neq0}D^\pi\mathcal{F}_\lambda(M)$ of the split-closed derived Fukaya category and that of the split-closed derived wrapped Fukaya category $\bigsqcup_{\lambda\neq0}D^\pi\mathcal{W}_\lambda(M)$ of $M$ are split-generated by the Lagrangian branes $(L,\xi_L^i)$ for $i=1,\cdot\cdot\cdot,n_2+1-mn_1$.
\end{theorem}

\paragraph{Blow-ups of symplectic vector spaces.} The material here is a generalization of Section 5.1 of $\cite{ll}$. Consider the symplectic blow-up of $\mathbb{C}^n$ at a finite set of points $S=\{q_1,\cdot\cdot\cdot,q_m\}$, as we have noticed in Section \ref{section:Lefschetz}, $M=\mathit{Bl}_S(\mathbb{C}^n)$ is a Lefschetz manifold. Its symplectic cohomology can be computed using Theorem \ref{theorem:flip} and known results.
\begin{proposition}\label{proposition:sh-blp}
Let $M=\mathit{Bl}_S(\mathbb{C}^n)$, then
\begin{equation}\mathit{SH}^\ast(M)\cong\bigoplus_{i=1}^m\mathbb{K}[x_i]/\left(x_i^n+q^n\cdot n\right).\end{equation}
\end{proposition}
\begin{proof}
Let $M_\#$ be the boundary connected sum of $m$ copies of the unit ball bundle associated to $\mathcal{O}(-1)\rightarrow\mathbb{CP}^{n-1}$. By Theorem \ref{theorem:flip} and the computation of $\mathit{SH}^\ast\left(\mathcal{O}_{\mathbb{P}^{n-1}}(-1)\right)$ in $\cite{ar}$, we have $\mathit{SH}^\ast(M_\#)\cong\bigoplus_{i=1}^m\mathbb{K}[x_i]/\left(x_i^n+q^n\right)$ is semisimple. On the other hand, by our proof of Corollaries \ref{corollary:blow-up} and \ref{corollary:flip} in Section \ref{section:symp-birational}, $\mathit{SH}^\ast(M)\cong\mathit{SH}^\ast(M_\#)$.
\end{proof}
For the purpose of dealing with Fukaya categories, we shall blow up with equal amounts at every point of $S$, so that $M$ is monotone. To simplify our exposition, we shall impose the following two simplifying assumptions:
\begin{itemize}
\item[(i)] $S\subset\mathbb{C}$ lies on a common complex plane.
\item[(ii)] $\mathrm{dist}_\mathbb{C}(q_i,q_j)\gg0$ for any $i\neq j$.
\end{itemize}
Note that (i) in the above gives us a Morse-Bott fibration
\begin{equation}\label{eq:mbf}\pi:M\rightarrow\mathbb{C}.\end{equation}
This is obtained by starting with the trivial projection $\mathbb{C}^n\rightarrow\mathbb{C}$ to the complex plane containing $S$, and then attaching an exceptional $\mathbb{CP}^{n-1}$ to the fiber over the points in $S$. As a consequence, the critical loci of $\pi$ form a disjoint union of $m$ copies of $\mathbb{CP}^{n-2}$. We remark that once there is such a Morse-Bott fibration $\pi$ on $M$, a generalization of the method developed in $\cite{am,mm}$ for studying symplectic cohomologies of Lefschetz fibrations is possible, which reduces the computation of $\mathit{SH}^\ast(M)$ on the additive level to the Morse complexes of the critical loci of $\pi$.

(ii) is needed to ensure that $M$ can be equipped with a symplectic form $\omega_M$ so that every exceptional divisor $\mathbb{CP}^{n-1}\subset M$ has area $\pi$.\bigskip

Since $(M,\omega_M)$ is the completion of a monotone Lefschetz domain, Theorem \ref{theorem:main} applies to $M$. Namely for $(L_1,\cdot\cdot\cdot,L_r)$ a set of wide monotone Lagrangian submanifolds of $M$ which are disjoinable by Hamiltonian isotopies, we have
\begin{equation}\label{eq:bd-blp}r\leq m(n-1).\end{equation}

We shall find a set of monotone Lagrangian branes in $M$ which satisfies the assumptions in Theorem \ref{theorem:generation} and realizes the upper bound of (\ref{eq:bd-blp}). To do this, take the unit ball bundles $E_+^\mathrm{in}\subset M$ of every exceptional divisor, their boundaries $\partial E_+^\mathrm{in}$ are circle bundles over $\mathbb{CP}^{n-1}$. Taking the circle bundle over the Clifford torus $T_\mathrm{Cl}^{n-1}\subset\mathbb{CP}^{n-1}$ gives us a Lagrangian torus $L_i\subset M$, so all together we get $m$ disjoint Lagrangian tori $L_1,\cdot\cdot\cdot,L_m\subset M$.

Under the Morse-Bott fibration $\pi:M\rightarrow\mathbb{C}$, these Lagrangian tori project to circles $\gamma_1,\cdot\cdot\cdot,\gamma_m\subset\mathbb{C}$ with the same radii. Taking the Clifford torus $T_\mathrm{Cl}^{n-2}\subset\mathbb{CP}^{n-2}$ in every connected component of the critical loci of $\pi$, the corresponding \textit{relative vanishing cycles} (cf. $\cite{ss1}$) of $\pi$ are given by the standard product tori $T_\mathrm{Cl}^{n-2}\times S^1\subset\mathbb{C}^{n-1}$, so $L_i$ can be regarded as a matching torus of $\pi$.
\begin{lemma}\label{lemma:dis-blp}
The Lagrangian tori $L_1,\cdot\cdot\cdot,L_m$ are monotone, and $\mathit{HF}^\ast(L_i,L_i)\neq0$ with $pt\in C_\ast(L_i)$ defining a Floer cocycle for every $i$. Moreover, one can equip each $L_i$ with $n-1$ different local systems $\xi_{L_i}^1,\cdot\cdot\cdot,\xi_{L_i}^{n-1}$, such that
\begin{equation}\mathit{HF}^\ast\big((L_i,\xi_{L_i}^j),(L_i,\xi_{L_i}^j)\big)\neq0,1\leq j\leq n-1.\end{equation}
\end{lemma}
\begin{proof}
The fact that $L_i$ is monotone follows conceptually from the fact that $L_i$ is the image of $T_\mathrm{Cl}^{n-1}\subset\mathbb{CP}^{n-1}$ under the monotone Lagrangian correspondence given by the boundary of the unit sphere bundle $\partial E_+^\mathrm{in}\subset M$. More concretely, let $u:(\mathbb{D},\partial\mathbb{D})\rightarrow(M,L_i)$ be a $J$-holomorphic disc bounded by $L_i$, where we take $J$ to be the standard complex structure on $M$. Applying the maximum principle to $\pi\circ u$ one sees that the image of $u$ lies inside $\pi^{-1}(D_i)$, where $D_i\subset\mathbb{C}$ is the disc bounded by $\gamma_i$. This enables us to identify the moduli spaces $\mathcal{M}_1^M(L_i,\beta_i)$ with the moduli space $\mathcal{M}_1^{\mathcal{O}(-1)}(L,\beta)$, where the superscripts indicate the symplectic manifolds that contain the relevant Lagrangian submanifolds, $\beta_i\in\pi_2(M,L_i)$, $\beta\in\pi_2\left(\mathcal{O}_{\mathbb{P}^{n-1}}(-1),L\right)$, and $L\subset\mathcal{O}_{\mathbb{P}^{n-1}}(-1)$ is the non-displaceable monotone Lagrangian torus corresponding to the critical points of the mirror superpotential, see Proposition \ref{proposition:toric}. Note that the moduli space $\mathcal{M}_1^{\mathcal{O}(-1)}(L,\beta)$ is regular with respect to the standard complex structure $J$, since $\mathcal{O}_{\mathbb{P}^{n-1}}(-1)$ is toric. This identification makes use of the Morse-Bott fibration $\pi:\mathcal{O}_{\mathbb{P}^{n-1}}(-1)\rightarrow\mathbb{C}$, which is obtained as a specialization of (\ref{eq:mbf}). Because of this, the monotonicity of $L_i$ follows from that of $L$. This also enables us to identify the pearl complexes of $L_i$ and $L$ computing the Floer cohomologies $\cite{bc1}$, so the rest of the lemma follows from Proposition \ref{proposition:toric}.
\end{proof}
When $n=2$, this in particular shows that the bound obtained in Theorem \ref{theorem:main} for the number of disjoinable monotone wide Lagrangian submanifolds is sharp for $\mathit{Bl}_S(\mathbb{C}^2)$. Combining Lemma \ref{lemma:dis-blp} with Proposition \ref{proposition:sh-blp} and Theorem \ref{theorem:generation} we get the following:
\begin{theorem}\label{theorem:blscn}
Let $M$ be the monotone Lefschetz manifold $\mathit{Bl}_S(\mathbb{C}^n)$. The non-zero eigensummands of the derived Fukaya categories $\bigsqcup_{\lambda\neq0}D^\pi\mathcal{F}_\lambda(M)$ and $\bigsqcup_{\lambda\neq0}D^\pi\mathcal{W}_\lambda(M)$ are split-generated by the Lagrangian branes $(L_1,\xi_{L_1}^1),\cdot\cdot\cdot,(L_1,\xi_{L_1}^{n-1}),\cdot\cdot\cdot,(L_m,\xi_{L_m}^1),\cdot\cdot\cdot,(L_m,\xi_{L_m}^{n-1})$.
\end{theorem}
\paragraph{A reverse flip.} As in Section \ref{section:Lefschetz}, denote by $M$ the $2n$-dimensional Lefschetz manifold obtained by the reverse simple flip
\begin{equation}\mathit{Bl}_S(E_-)\dashrightarrow M,\end{equation}
where $S\subset E_-\setminus E_-^\mathrm{in}$ is a finite set of points with $|S|=m$ on the cylindrical end of $E_-$, such that for every point $p\in S$ there is a symplectic embedding $\mathbb{B}_p(\sqrt{2\pi})\hookrightarrow E_-$, where $\mathbb{B}_p(\sqrt{2\pi})$ is a ball with radius $\sqrt{2\pi}$ centered at $p$, and $\mathbb{B}_p(\sqrt{2\pi})\cap\mathbb{B}_q(\sqrt{2\pi})=\emptyset$ for any two different points $p,q\in S$. It's then clear that $M$ can be equipped with a monotone symplectic form $\omega_M$ so that the symplectic area of any rational curve coming from blow-up is $\pi$. Note that this then forces $\omega_M$ restricted to $\mathbb{CP}^{n_2}\subset M$ equals $\frac{n-1}{n_2-n_1+1}\omega_\mathit{FS}$. In particular, as the completion of a stable filling, the monotone Fukaya category $\mathcal{F}(M)$ and the monotone wrapped Fukaya category $\mathcal{W}(M)$ of $M$ are well-defined. By Theorem \ref{theorem:flip}, we know that $\mathit{SH}^\ast(M)$ is semisimple and
\begin{equation}\dim_\mathbb{K}\mathit{SH}^\ast(M)=n_2+1-n_1+m(n-1).\end{equation}
From the above discussions, it is easy to deduce the following:
\begin{theorem}
There are $m+1$ monotone Lagrangian tori $L_0;L_1,\cdot\cdot\cdot,L_m\subset M$ such that there exist unitary local systems
\begin{equation}\xi_{L_0}^1,\cdot\cdot\cdot,\xi_{L_0}^{n_2+1-n_1}:\pi_1(L_0)\rightarrow U_\mathbb{K};\xi_{L_i}^1,\cdot\cdot\cdot,\xi_{L_i}^{n-1}:\pi_1(L_i
)\rightarrow U_\mathbb{K},i=1,\cdot\cdot\cdot,m\end{equation}
such that the Lagrangian branes $(L_0,\xi_{L_0}^1),\cdot\cdot\cdot,(L_0,\xi_{L_0}^{n_2+1-n_1});(L_1,\xi_{L_1}^1),\cdot\cdot\cdot,(L_1,\xi_{L_1}^{n-1}),\cdot\cdot\cdot,(L_m,\xi_{L_m}^1),\cdot\cdot\cdot,(L_m,\xi_{L_m}^{n-1})$ split-generate the non-zero eigensummands of $D^\pi\mathcal{F}(M)$ and $D^\pi\mathcal{W}(M)$.
\end{theorem}
\begin{proof}
Instead of relying on the Morse-Bott fibration, we can use Lemma \ref{lemma:m0} to localize the analysis of Maslov 2 holomorphic discs to domains isomorphic to $\mathcal{O}(-1)_{\leq1}$ or $\mathcal{O}(-1)^{\oplus n_1}_{\leq1}$. This ensures that the Lagrangian tori $L_1,\cdot\cdot\cdot,L_m$ created by blowing up at $S$ are monotone and non-displaceable. On the other hand, although we have changed the monotonicity constant to $n-1$, it is not hard to check the argument above for negative vector bundles still holds, and can be applied to the Lagrangian torus $L_0$ created by the reverse flip. The result then follows from Theorems \ref{theorem:gen-nv} and \ref{theorem:blscn}.
\end{proof}

\subsection{Mirror symmetry}
The monotone Lefschetz manifold $M=\mathit{Bl}_S(\mathbb{C}^2)$ appeared above is of particular interest since the mirror construction due to Abouzaid-Auroux-Katzarkov $\cite{aak}$ can be applied to it to get its Mirror Landau-Ginzburg model $(M^\vee,W)$.\bigskip

We briefly recall here the geometry of the mirror $(M^\vee,W)$, the details can be found in Section 9 of $\cite{aak}$. Denote by $\overline{M}^\vee$ the resolution of the $(A_{m-1})$ singularity $\{xy=z^m\}\subset\mathbb{K}^3$. Note that as a toric Calabi-Yau surface, $\overline{M}^\vee$ can be covered by $m$ affine coordinate charts $\overline{U}_i\cong\mathbb{K}^2$ with coordinates $(x_i,y_i)$ for $1\leq i\leq m$, and $z$ defines a regular function on $\overline{M}^\vee$ which restricts to $q^\varepsilon x_i y_i-q^\varepsilon$ on each $\overline{U}_i$, where $q^\varepsilon\in\mathbb{K}$ is some constant. $M^\vee$ is the complement of the surface $\{z=1\}\subset\overline{M}^\vee$.

Denote by $x$ the regular function on $\overline{M}^\vee$ defined by $x_1$, or its restriction on $M^\vee$, then the superpotential $W:M^\vee\rightarrow\mathbb{K}$ is given by $x+z$. Denote by $D^\pi_\mathit{sing}\left(W^{-1}(-q^\varepsilon)\right)$ the split-closure of the triangulated category of singularities defined in $\cite{do}$, where $-q^\varepsilon$ is the unique critical value of $W$. The following statement concerning homological mirror symmetry is easy to prove.

\begin{proposition}
Let $\mathcal{W}_\lambda(M)$ denote the unique non-zero summand of the monotone wrapped Fukaya category of $M$. There is an equivalence between triangulated categories
\begin{equation}D^\pi\mathcal{W}_\lambda(M)\cong D^\pi_\mathit{sing}\left(W^{-1}(-q^\varepsilon)\right).\end{equation}
\end{proposition}
\begin{proof}
The wrapped Fukaya category $\mathcal{W}_\lambda(M)$ has been computed in the last subsection. In this case, $D^\pi\mathcal{W}_\lambda(M)$ is split-generated by the Lagrangian tori $L_1,\cdot\cdot\cdot,L_m$ in Theorem \ref{theorem:blscn}, and
\begin{equation}\mathit{HW}^\ast(L_i,L_i)\cong\mathit{HF}^\ast(L_i,L_i)\cong\mathrm{Cl}_2\end{equation}
are isomorphic to the Clifford algebra.

On the other hand, $M^\vee$ is covered by $m$ affine charts $U_i\cong\mathbb{K}^2\setminus\{x_i y_i=1\}$, and the superpotential $W$ on each chart $U_i$ has the form
\begin{equation}W_i=x_i^i y_i^{i-1}+q^\varepsilon x_i y_i-q^\varepsilon,\end{equation}
where $q^\varepsilon\in\mathbb{K}^\ast$. Elementary calculations show that $W:M^\vee\rightarrow\mathbb{K}$ has $m$ non-degenerate critical points $p_1,\cdot\cdot\cdot,p_m$. (Note that the critical point $(-q^\varepsilon,0)$ on $U_2$ can be identified with the critical point $\left(0,-q^{-\varepsilon}\right)$ on $U_1$, all the other critical points are origins of the charts $U_2,\cdot\cdot\cdot,U_m$). It follows that the triangulated category $D^\pi_\mathit{sing}\left(W^{-1}(-q^\varepsilon)\right)$ is split-generated by skyscraper sheaves $\mathcal{O}_{p_i}$, $1\leq i\leq m$. It's not hard to compute $\mathit{Hom}(\mathcal{O}_{p_i},\mathcal{O}_{p_i})$ in $D^\pi_\mathit{sing}\left(W^{-1}(-q^\varepsilon)\right)$, from which one sees that it is isomorphic to $\mathit{HW}^\ast(L_i,L_i)$. Note that the computation of $\mathit{Hom}(\mathcal{O}_{p_i},\mathcal{O}_{p_i})$ is the simplest instance of Kn\"{o}rrer periodicity, which holds over any field $\mathbb{K}$ with $\mathrm{char}(\mathbb{K})=0$.
\end{proof}
More generally, the monotone wrapped Fukaya category should play a role in understanding the mirror symmetry for blow-ups of toric varieties in higher dimensions $\cite{aak}$, which will then involve blowing up along non-compact submanifolds. For example, let $\mathbb{S}\subset\mathbb{C}^3$ be the disjoint union of two complex planes
\begin{equation}\mathbb{C}_x\sqcup\mathbb{C}_y\subset\mathbb{C}_{x,y}^2\times\{z_1\}\sqcup\mathbb{C}^2_{x,y}\times\{z_2\}\subset\mathbb{C}^3_{x,y,z},\end{equation}
with $z_1\neq z_2$. Then $M=\mathit{Bl}_\mathbb{S}(\mathbb{C}^3)$ can be equipped with a monotone symplectic structure. It has a mirror Landau-Ginzburg model $(M^\vee,W)$ described in Section 11 of $\cite{aak}$, where $M^\vee$ is an open dense subset of the resolved conifold $\mathcal{O}(-1)^{\oplus2}\rightarrow\mathbb{CP}^1$. In fact, algebraically, $M$ can be realized as a partial compactification of $T^\ast S^3$. The symplectic topology of these manifolds will be studied elsewhere.


\begin{thebibliography}{99}
\fontsize{9pt}{9pt}\selectfont
\bibitem{ma}M. Abouzaid, \textit{A geometric criterion for generating the Fukaya category}, Publ. Math. Inst. Hautes. \'{E}tudes Sci., (112):191-240, 2010.
\bibitem{aak}M. Abouzaid, D. Auroux, and L. Katzarkov, \textit{Lagrangian fibrations on blowups of toric varieties and mirror symmetry for hypersurfacese}, Publ. Math. IHES 123 (2016), 199-282.
\bibitem{as}M. Abouzaid and P. Seidel, \textit{An open string analogue of Viterbo functoriality}, Geom. Topol. 14 (2010), no. 2, 627-718.
\bibitem{ak}P. Albers and J. Kang, \textit{Vanishing of Rabinowitz Floer homology on negative line bundles}, Math. Z. (2016).
\bibitem{am}P. Albers and M. McLean, \textit{Non-displaceable contact embeddings and infinitely many leaf-wise intersections}, J. Symplectic Geom. Volume 9, Number 3 (2011), 271-284.
\bibitem{ako}D. Auroux, L. Katzarkov and D. Orlov, \textit{Mirror symmetry for Del Pezzo surfaces: Vanishing cycles and coherent sheaves}, Invent. math. 166, 537-582 (2006).
\bibitem{ab}A. Bayer, \textit{Semisimple quantum cohomology and blowups}, Int. Math. Res. Not., (40):2069-2083, 2004.
\bibitem{bc}P. Biran and O. Cornea, \textit{Rigidity and uniruling for Lagrangian submanifolds}, Geom. Topol. 13, 2881-2989 (2009).
\bibitem{bc1}P. Biran and O. Cornea, \textit{Quantum structures for Lagrangian submanifolds}, arXiv:0708.4221.
\bibitem{bw}W.M. Boothby and H.C. Wang, \textit{On contact manifolds}, Ann. of Math. (2) 68 (1958), 721-734.
\bibitem{cc}C.H. Cho, \textit{Products of Floer cohomology of torus fibers in toric manifolds}, Commun. Math. Phys. 260, 613-640 (2005).
\bibitem{kc}K. Cieliebak, \textit{Handle attaching in symplectic homology and the chord conjecture}, J. Eur. Math. Soc. 4 (2002), no.2, 115-142.
\bibitem{co}K. Cieliebak and A. Oancea, \textit{Symplectic homology and the Eilenberg-Steenrod axioms}, Algebraic and Geometric Topology 18 (2018) 1953-2130.
\bibitem{cls}D. Cox, J. Little and H. Schenk, \textit{Toric varieties}, Graduate Studies in Mathematics, 124: AMS (2011).
\bibitem{cv}K. Cieliebak and E. Volkov, \textit{First steps in stable Hamiltonian topology}, J. Eur. Math. Soc. 17, 321-404 (2015).
\bibitem{ep}M. Entov, L. Polterovich, \textit{Rigid subsets of symplectic manifolds}, Compos. Math. 145, 773-826 (2009).
\bibitem{fooo}K. Fukaya, Y.G. Oh, H. Ohta, and K. Ono, \textit{Lagrangian intersection Floer theory: anomaly and obstruction}, volume 46 of AMS/IP Studies in Advanced Mathematics. American Mathematical Society, Providence, RI, 2009.
\bibitem{gz}H. Geiges and K. Zehmisch, \textit{The Weinstein conjecture for connected sums}, arXiv:1407.4307.
\bibitem{gnw}P. Ghiggini, K. Niederkr\"{u}ger and C. Wendl, \textit{Subcritical contact surgeries and the topology of symplectic fillings}, Journal de l'\'{E}cole polytechnique-Math\'{e}matiques, 3 (2016), 163-208.
\bibitem{yg}Y. Groman, \textit{Floer theory and reduced cohomology on open manifolds}, arXiv:1510.04265.
\bibitem{hs}H. Hofer and D. Salamon, \textit{Floer homology and Novikov rings}, The Floer memorial volume, 483-524, Progr. Math., 133, Birkh\"{a}user, 1995.
\bibitem{km}M. Kontsevich and Y. Manin, \textit{Gromov-Witten classes, quantum cohomology, and enumerative geometry}, Communications in Mathematical Physics 164, 3 (1994), 525-562.
\bibitem{lw}J. Latschev and C. Wendl, \textit{Algebraic torsion in contact manifolds}, Geom. Funct. Anal. 21(5), 1144-1195 (2011). With an appendix by M. Hutchings.
\bibitem{ypl}Y.P. Lee, \textit{Quantum Lefschetz hyperplane theorem}, Invent. Math. 145(1) (2001) 121-149.
\bibitem{el}E. Lerman, \textit{Contact toric manifolds}, J. Symplectic Geom. 1 (2003), 785-828.
\bibitem{sl}S. Luo, \textit{Cohomology rings of good contact toric manifolds}, arXiv:1012.2146.
\bibitem{yk}Y. Kawamata, \textit{D-equivalence and K-equivalence}, J. Differential Geom., 61(1):147-171, 2002.
\bibitem{lm}Y. Lekili and M. Maydanskiy, \textit{The symplectic topology of some rational homology balls}, Comment. Math. Helv. 89 (2014), 571-596.
\bibitem{ll}N.C. Leung and Y. Li, \textit{Twin Lagrangian fibrations in mirror symmetry}, Journal of Symplectic Geometry 17 (5) (2019), 1331-1387.
\bibitem{mnw} P. Massot, K. Niederkr\"{u}ger, and C. Wendl, \textit{Weak and strong fillability of higher dimensional contact manifolds}, Invent. Math., 192(2):287-373, 2013.
\bibitem{ms}D. McDuff and D. Salamon, \textit{J-holomorphic Curves and Symplectic Topology}, AMS Colloquium Publications, Vol 52, 2004.
\bibitem{mt}D. McDuff and S. Tolman, \textit{Topological properties of Hamiltonian circle actions}, IMRP Int. Math. Res. Pap., no. 72826, 1-77, 2006.
\bibitem{mm}M. McLean, \textit{Lefschetz fibrations and symplectic homology}, Geom. Topol. 13 (2009), 1877-1944.
\bibitem{ao}A. Oancea, \textit{Fibered symplectic cohomology and the Leray-Serre spectral sequence}, J. Symplectic Geom. 6 (2008), 267-351.
\bibitem{do}D. Orlov, \textit{Triangulated categories of singularities and D-branes in Landau-Ginzburg models}, Proc. Steklov Inst. Math. 246, no. 3 (2004), 227-248.
\bibitem{ar}A. Ritter, \textit{Floer theory for negative line bundles via Gromov-Witten invariants}, Adv. Math., Vol 262, 1035-1106, 2014.
\bibitem{ar1}A. Ritter, \textit{Circle-actions, quantum cohomology, and the Fukaya category of Fano toric varieties}, Geom. Topol 20 (2016) 1941-2052.
\bibitem{ar2}A. Ritter, \textit{Topological quantum field theory structure on symplectic cohomology}, J. Topol. 6(2) (2013) 391-489.
\bibitem{ar3}A. Ritter, \textit{Deformations of symplectic cohomology and exact Lagrangians in ALE spaces}, Geom. Funct. Anal. 20(3) (2010) 779-816.
\bibitem{rs}A. Ritter and I. Smith, \textit{The monotone wrapped Fukaya category and the open-closed string map}, Sel. Math. New Ser. (2016).
\bibitem{ps}P. Seidel, \textit{Disjoinable Lagrangian spheres and dilations}, Invent math (2014) 197:299-359.
\bibitem{ps1}P. Seidel, \textit{Lagrangian homology spheres in} ($A_m$) \textit{Milnor fibres via} $\mathbb{C}^\ast$-\textit{equivariant} $A_\infty$-\textit{modules}, Geom. Topol. 16, 2343-2389 (2012).
\bibitem{ps2}P. Seidel, \textit{A biased view of symplectic cohomology}, Current Developments in Mathematics, Volume 2006 (2008), 211-253, 2006.
\bibitem{ps3}P. Seidel, $\pi_1$ \textit{of symplectic automorphism groups and invertibles in quantum homology rings}, Geom. Funct. Anal. 7, no. 6, 1046-1095, 1997.
\bibitem{ps4}P. Seidel, \textit{Fukaya categories and Picard-Lefschetz theory}, Z\"{u}rich Lect. in Adv. Math., European Math. Soc., Z\"{u}rich, 2008.
\bibitem{ps5}P. Seidel, \textit{Fukaya categories and deformations}, In: Proceedings of the International Congress of Mathematicians, Beijing, vol. 2, 351-360. Higher Education Press, Bristol (2002).
\bibitem{ps6}P. Seidel, \textit{Picard-Lefschetz theory and dilating} $\mathbb{C}^\ast$-\textit{actions}, J. Topology (2015) 8 (4): 1167-1201.
\bibitem{ss}P. Seidel and J. Solomon, \textit{Symplectic cohomology and q-intersection numbers}, Geom. Funct. Anal. 22, 443-477 (2012).
\bibitem{ss1}P. Seidel and I. Smith, \textit{A link invariant from the symplectic geometry of nilpotent slices}, Duke Math. J. 134:453-514, 2006.
\bibitem{ns}N. Sheridan, \textit{On the Fukaya category of a Fano hypersurface in projective space}, arXiv:1306.4143, to appear on Publications math\'{e}matiques de l'IH\'{E}S.
\bibitem{is}I. Smith, \textit{Floer cohomology and pencils of quadrics}, Invent. Math. 189, no. 1, 149-250, 2012.
\bibitem{sty}I. Smith, R. Thomas, S.T. Yau, \textit{Symplectic conifold transitions}, J. Differ. Geom. 62, 209-242 (2002).
\bibitem{cvi} C. Viterbo, \textit{Functors and computations in Floer homology with applications, Part I}, Geom. Funct. Anal., 9:985-1033, 1999.
\bibitem{aw}A. Weinstein, \textit{Contact surgery and symplectic handlebodies}, Hokkaido Math. J. 20 (1991), no. 2, 241-251.
\bibitem{cw}F. Charest and C. Woodward, \textit{Floer theory and flips}, arXiv:1508.01573.
\end{thebibliography}
\end{document}